\documentclass{amsart}

\pdfoutput=1
\usepackage[utf8]{inputenc}
\usepackage[T1]{fontenc}
\usepackage{lmodern}
\usepackage[leqno]{amsmath}
\allowdisplaybreaks
\usepackage{amssymb}
\usepackage{amsthm}
\usepackage{mathtools}
\usepackage[dvipsnames,svgnames,table]{xcolor}
\usepackage{xspace}
\usepackage{embrac}
\ChangeEmph{(}[-.01em,.04em]{)}[.04em,-.05em]
\usepackage{microtype}
\usepackage{textcase}
\usepackage{cancel}

\usepackage{enumitem}
\usepackage[acronym]{glossaries}
\usepackage{graphicx}
\usepackage{float}
\usepackage{booktabs}
\usepackage{csquotes}
\usepackage{multirow}
\usepackage{tikz}
\usepackage{caption}
\usepackage{subcaption}
\usetikzlibrary{angles,quotes,calc,positioning,shapes,matrix,arrows}
\usepackage{tikz-cd}
\usetikzlibrary{}
\usepackage{tkz-euclide}
\usepackage{etoolbox}
\usepackage{url}
\usepackage[hyperindex=true, colorlinks=true, linkcolor=Maroon, citecolor=DarkGreen, urlcolor=NavyBlue, breaklinks=true,linktocpage=true,linktoc=all]{hyperref}
\usepackage[english,capitalise,noabbrev]{cleveref}
\crefname{equation}{equation}{equations}
\usepackage[logical-quotes,backrefs,initials]{amsrefs}

\newcommand{\N}{\mathbb{N}}

\newcommand{\R}{\mathbb{R}}
\newcommand{\C}{\mathbb{C}}

\renewcommand{\P}{\mathbb{P}}
\newcommand{\RP}{\mathbb{RP}}

\newcommand{\M}{\mathcal{M}}

\newcommand{\calK}{\mathcal{K}}
\newcommand{\dd}{\mathop{}\!\mathrm{d}}
\newcommand{\LL}[1]{L^{#1}}
\newcommand{\contained}{\subset}

\newcommand{\intersection}{\cap}

\newcommand{\union}{\cup}
\newcommand{\bigunion}{\bigcup}

\newcommand{\suchthat}{\, : \,}
\newcommand{\st}{\suchthat}
\newcommand{\card}{\#}
\newcommand{\defined}{\coloneqq}

\newcommand{\identically}{\equiv}

\newcommand{\bigdirectsum}{\bigoplus}

\newcommand{\transpose}{^{\text{\tiny$\mathsf{T}$}}}

\newcommand{\orthogonal}{\perp}
\newcommand{\conj}[1]{\overline{#1}}
\renewcommand{\O}{\operatorname{O}}

\newcommand{\rank}{\operatorname{rank}}
\renewcommand{\Re}{\operatorname{Re}} 
\newcommand{\gradient}{\nabla}

\newcommand{\Laplacian}{\Delta}

\newcommand{\tendsto}{\rightarrow}

\newcommand{\from}{\colon}

\newcommand{\diam}{\operatorname{diam}}
\newcommand{\vol}{\operatorname{vol}}

\newcommand{\distributed}{\sim}

\DeclarePairedDelimiterX\norm[1]\lVert\rVert{\ifblank{#1}{\:\cdot\:}{#1}}
\DeclarePairedDelimiterX\abs[1]\lvert\rvert{\ifblank{#1}{\:\cdot\:}{#1}}
\DeclarePairedDelimiterX\set[1]{\{}{\}}{\ifblank{#1}{\: \:}{#1}}
\DeclarePairedDelimiterX\innerprod[2]\langle\rangle
{\ifblank{#1#2}{\,\cdot,\cdot\,}
	{#1,#2}}
 \DeclarePairedDelimiterX\mutualEnergy[2]\langle\rangle{\ifblank{#1#2}{\,\cdot,\cdot\,}{#1,#2}}
\DeclarePairedDelimiterX\floor[1]\lfloor\rfloor{\ifblank{#1}{\:\cdot\:}{#1}}
\DeclarePairedDelimiterXPP\expectation[2]{\ifblank{#1}{\mathbb{E}}{\mathbb{E}_{#1}}}\lbrack\rbrack{}{\ifblank{#2}{\:\cdot\:}{#2}}
\DeclarePairedDelimiterXPP\Pochhammer[2]{}{(}{)}{_{#2}}{\ifblank{#1}{\cdot}{#1}}
\newcommand{\Grass}[3][]{\ifblank{#1}{\operatorname{Gr}_{#2,#3}}
	{\operatorname{Gr}_{#2,#3}(#1)}}
\newcommand{\embeddedGrass}[3][]{\ifblank{#1}{\widetilde{\operatorname{Gr}}_{#2,#3}}
	{\widetilde{\operatorname{Gr}}_{#2,#3}(#1)}}
\newcommand{\Stief}[3][]{\ifblank{#1}{\operatorname{St}_{#2,#3}}
	{\operatorname{St}_{#2,#3}(#1)}}
 \newcommand{\HalfStief}[3][]{\ifblank{#1}{\overline{\operatorname{St}}_{#2,#3}}
	{\overline{\operatorname{St}}_{#2,#3}(#1)}}
\newcommand{\Gegenbauer}[2]{C_{#1}^{(#2)}}
\newcommand{\Jacobi}[3]{P_{#1}^{(#2,#3)}}

\newcommand{\JacobiGrass}[1]{P_{#1}}
\newcommand{\Legendre}[1]{P_{#1}}
\newcommand{\hypergeom}[2]{\prescript{}{#1}{F}^{}_{#2}}
\newcommand{\BesselJ}[1]{J_{#1}}

\newcommand{\Pol}{\operatorname{Pol}}
\newcommand{\chordalGrass}{d_c}
\newcommand{\Energy}{E}
\newcommand{\minEnergy}{\mathcal{E}}
\newcommand{\contEnergy}{I}
\newcommand{\Wiener}{W}
\newcommand{\Measures}{\mathcal{M}}
\newcommand{\SignedMeasures}{\mathcal{S}}
\newcommand{\ZeroMeasures}{\mathcal{Z}}
\newcommand{\Partitions}{\mathcal{P}}
\newcommand{\DegreePart}[1]{\abs{#1}}
\newcommand{\HarmonicDPP}[1]{\mathfrak{X}^{(#1)}}
\newcommand{\pushforward}[2]{#2_{*}#1}

\newacronym{dpp}{\textup{\textsc{{dpp}}}}{\emph{determinantal point process}}

\theoremstyle{plain}
\newtheorem{theorem}{Theorem}[section]

\newtheorem{corollary}[theorem]{Corollary}
\newtheorem{lemma}[theorem]{Lemma}
\newtheorem{proposition}[theorem]{Proposition}

\theoremstyle{definition}
\newtheorem{definition}[theorem]{Definition}

\theoremstyle{remark}
\newtheorem{remark}[theorem]{Remark}

\begin{document}

	\title[Minimal Riesz and logarithmic energies on the Grassmannian]{Minimal Riesz and logarithmic energies\\ on the Grassmannian $\operatorname{Gr}_{2,4}$}
	
	\author{Ujué Etayo}
	\address{Ujué Etayo: Departamento de Métodos Cuantitativos, CUNEF, Leonardo Prieto Castro, 2. Ciudad Universitaria, 28040 Madrid, Spain}
	\email{ujue.etayo@cunef.edu}
	
	\author{Pedro R. López-Gómez}
	\address{Pedro R. López-Gómez: Departamento de Matemáticas, Estadística y Computación, Universidad de Cantabria,  Avda. Los Castros, s/n, 39005 Santander, Spain}
	\email{lopezpr@unican.es}
	
	\date{\today{}}
	
	\thanks{The authors have been supported by grant PID2020-113887GB-I00 funded by MCIN/AEI/10.13039/501100011033. The first author has also been supported by the starting grant from FBBVA associated with the prize José Luis Rubio de Francia. The second author has also been supported by grant PRE2021-097772 funded by MCIN/AEI/10.13039/501100011033 and by ``ESF Investing in your future''.}
	
	

	\begin{abstract}
		We study the Riesz and logarithmic energies on the Grassmannian $\operatorname{Gr}_{2,4}$ of $2$-dimensional subspaces of $\mathbb{R}^4$. We prove that the continuous Riesz and logarithmic energies are uniquely minimized by the uniform measure, and we obtain asymptotic upper and lower bounds for the minimal discrete energies, with matching orders for the next-order terms. Additionally, we define a determinantal point process on $\operatorname{Gr}_{2,4}$ and compute the expected energy of the points coming from this random process, thereby obtaining explicit constants in the upper bounds for the Riesz and logarithmic energies.
	\end{abstract}
	
	\maketitle


\section{Introduction}

The study of energy minimization problems has long been a central theme in various branches of mathematics and physics. Of particular interest are the Riesz and logarithmic energies given their relevance in areas like potential theory, approximation theory, or mathematical physics, among others. In this paper, we explore the problem of minimizing the Riesz and logarithmic energies on the Grassmann manifold, a space of significant interest not only for the disciplines above, but also for areas like coding theory and signal processing.

\subsection{The Grassmann manifold}\label{subsec:grassmannian}

In this section, we introduce the basic notions about the Grassmann manifold that we will use throughout this document. For our purposes, a brief description will suffice. The reader interested in more information about this space can consult, for example, \cites{AbsilMahonySepulchre2004, BendokatZimmermannAbsil2024}.

The (real, unoriented) Grassmann manifold or Grassmannian $\Grass{m}{d}$ is the set of linear subspaces of dimension $m$ of $\R^d$. The orthogonal group $\O(d)$ acts as a transitive group of transformations on $\Grass{m}{d}$, that is, given any two $m$-dimensional subspaces of $\R^{d}$, there exists an isometry in $\O(d)$ that maps one into the other. The stabilizer of an element $P\in \Grass{m}{d}$ also stabilizes its orthogonal complement $P^\orthogonal$. Hence, we obtain the classical identification of the Grassmannian as a homogeneous space:
\begin{equation*}
    \Grass{m}{d}=\frac{\O(d)}{\O(m)\times \O(d-m)}.
\end{equation*}
From this construction, the Grassmannian inherits the structure of a compact smooth manifold of dimension $m(d-m)$. 

Unlike spheres and projective spaces, the Grassmann manifold is not a two-point homogeneous space. This means that the orbits of pairs of points under the action of the isometry group $\O(d)$ cannot be characterized by a single invariant, such as the distance between points. Instead, to understand the orbits of pairs $(P,Q)\in\Grass{m}{d}^2$ under the action of $\O(d)$ we have to introduce the concept of \emph{principal angles} between subspaces of $\R^d$. These principal angles are defined recursively as follows. Let $P,Q\in\Grass{m}{d}$ be two $m$-dimensional subspaces of $\R^d$. The first principal angle $\theta_1(P,Q)$ between $P$ and $Q$ is defined as the smallest angle between any two unit vectors $u\in P$ and $v\in Q$, that is,
\begin{equation*}
    \theta_1(P,Q)\defined \min_{u\in P,\, v\in Q}\arccos\abs{\innerprod{u}{v}}.
\end{equation*}
Let $u_1$ and $v_1$ be the vectors for which this minimum is attained. The second principal angle $\theta_2(P,Q)$ is then defined as the smallest angle between any two unit vectors $u\in P\intersection u_1^\orthogonal$ and $v\in Q\intersection v_1^\orthogonal$. The remaining principal angles are defined analogously. When the elements $P$ and $Q$ are understood, we simply write $\theta_i$ instead of $\theta_i(P,Q)$. This way, we obtain an $m$-tuple $(\theta_1,\dotsc,\theta_m)$, with $0\leq \theta_1\leq\dotsb\leq \theta_m\leq \pi/2$, that fully characterizes the orbits of $(P,Q)$ under the action of $\O(d)$, in the following sense.

\begin{proposition}[{\cite[Proposition 1.2]{Bachoc2006}}]
    Two pairs $(P,Q),(P',Q')\in\Grass{m}{d}^2$ are in the same orbit under the the action of $\O(d)$ if and only if
    \begin{equation*}
        \theta_i(P,Q)=\theta_i(P',Q')\quad \text{for all $1\leq i\leq m$}.
    \end{equation*}
\end{proposition}

From the definition of the principal angles it is clear that $P=Q$ if and only if all the principal angles between them are zero. 

It is a classical result in the geometry of Grassmann manifolds that any notion of distance that depends only on the relative position between subspaces must be a function of the principal angles (see \cite[Theorem 3]{Wong1967} and \cite[Theorem 2]{YeLim2016}). The geodesic or Riemannian distance between two elements $P,Q\in\Grass{m}{d}$ is given by
\begin{equation*}
    d_g(P,Q)=\sqrt{\theta_1^2+\dotsb+\theta_m^2},
\end{equation*}
where $\theta_1,\dotsc,\theta_m$ are the principal angles between $P$ and $Q$. However, this distance is not smooth everywhere. Therefore, we will work instead with the \emph{chordal distance} introduced in \cite{ConwayHardinSloane1996}, which is defined as
\begin{equation*}
    \chordalGrass(P,Q)=\sqrt{\sin^2\theta_1+\dotsb+\sin^2\theta_m}.
\end{equation*}
From the definition of the Grassmann manifold, it is clear that $\Grass{1}{d}$ is just the real projective space $\RP^d$. Moreover, by taking orthogonal complements we can identify $\Grass{m}{d}$ with $\Grass{d-m}{d}$. Therefore, $\Grass{2}{4}$ is the simplest Grassmannian that is algebraically different from a projective space. In this work, we focus on the Grassmannian $\Grass{2}{4}$. In that case, the orbits of pairs $(P,Q)\in\Grass{2}{4}^2$ are characterized by the two principal angles $0\leq \theta_1\leq\theta_2\leq \pi/2$ between $P$ and $Q$, and the chordal distance is simply
\begin{equation}\label{eq:chordal-distance-angles}
    \chordalGrass(P,Q)=\sqrt{\sin^2\theta_1+\sin^2\theta_2}.
\end{equation}
We will use the following substitution throughout this work:
\begin{equation}\label{eq:xi_+-}
\begin{split}
	\xi_{+}(P,Q)&=\cos(\theta_1(P,Q)+\theta_2(P,Q)),\\
	\xi_{-}(P,Q)&=\cos(\theta_1(P,Q)-\theta_2(P,Q)).
\end{split}
\end{equation}
Unless necessary, we will omit the dependence on $P$ and $Q$. It can be checked that
\begin{equation}\label{eq:relation-xi-costheta}
	\xi_{+}\xi_{-}=\cos^2\theta_1+\cos^2\theta_2-1.
\end{equation}
With this notation, the chordal distance can be written as
\begin{equation}\label{eq:chordal-distance-xi}
    \chordalGrass(P,Q)=\sqrt{1-\xi_{+}\xi_{-}}.
\end{equation}
We denote the $\O(4)$-invariant probability measure on the Grassmannian $\Grass{2}{4}$ induced by the Haar measure on $\O(4)$ as $\sigma$. We will refer to it as the \emph{uniform measure} on the Grassmann manifold. We present now a classical result. Recall that a set of $m$ orthonormal vectors in $\mathbb{R}^{d}$ is called an $m$-frame and that $m$-frames are the points of the Stiefel manifold $\Stief{m}{d}$. We denote the set of $m$-frames with positive first components as $\HalfStief{m}{d}$

\begin{proposition}[{\cite[p. 67]{James1954}; \cite[(3.10)]{Davis1999b}}]\label{lemma_decomp_measure_24}
The density of the uniform measure $\sigma$ on $\Grass{2}{4}$ is given by
\begin{equation*}
    \dd\sigma=(\dd\tilde{\theta})(\dd V)(\dd Z),
\end{equation*}
where
\begin{equation*}
    \dd \tilde{\theta}=2(\cos^2\theta_1-\cos^2\theta_2)\dd\theta_1\dd\theta_2,
\end{equation*}
and $\dd V$ and $\dd Z$ are the densities of the uniform measures on $\HalfStief{2}{4}$ and $\Stief{2}{4}$, respectively.
\end{proposition}

\begin{corollary}\label{lemma_integral}
    Let $F\from \Grass{2}{4} \times \Grass{2}{4} \to \R$ be such that $F(P,Q) = f(\theta_1, \theta_2)$ where $\theta_1,\theta_2$ are the principal angles between $P$ and $Q$. Then,
    \begin{equation*}
        \iint\limits_{\mathclap{\Grass{2}{4}\times \Grass{2}{4}}} F(P,Q) \dd\sigma(P)\dd\sigma(Q)
        =\int_{0}^{\pi/2}\int_{0}^{\theta_2}
        f(\theta_1, \theta_2)
        2(\cos^2\theta_1-\cos^2\theta_2)\dd \theta_1 \dd \theta_2.
    \end{equation*}
\end{corollary}

In order to write \cref{lemma_integral} in terms of the variables $\xi_{+}$ and $\xi_{-}$ introduced in \eqref{eq:xi_+-}, note that the Jacobian of that change of variables is precisely $2(\cos^2\theta_1-\cos^2\theta_2)$. This, together with some simple symmetry arguments, leads to the following consequence of \cref{lemma_integral}.

\begin{lemma}\label{lemma:integral-Grass24-2}
	Let $F\from \Grass{2}{4} \times \Grass{2}{4} \to \R$ be a function such that $F(P,Q)=g(\xi_{+}(P,Q),\xi_{-}(P,Q))$, where $g\from [-1,1]^2\to \R$ is an integrable function such that $g(\xi_{+},\xi_{-})=g(\xi_{-},\xi_{+})$ and $g(-\xi_{+},-\xi_{-})=g(\xi_{+},\xi_{-})$. Then,
	\begin{align*}
		\iint\limits_{\mathclap{\Grass{2}{4}\times \Grass{2}{4}}} F(P,Q)\dd\sigma(P)\dd\sigma(Q)
		&=\int_{0}^{1}\int_{-\xi_{-}}^{\xi_{-}}g(\xi_{+},\xi_{-})\dd \xi_{+}\dd\xi_{-}\\
		&=\frac{1}{2}\int_{0}^{1}\int_{-1}^{1}g(\xi_{+},\xi_{-})\dd \xi_{+}\dd\xi_{-}\\
		&=\frac{1}{4}\int_{-1}^{1}\int_{-1}^{1}g(\xi_{+},\xi_{-})\dd \xi_{+}\dd\xi_{-}.
	\end{align*}
\end{lemma}

We conclude this section with a classical fact. It is well known (see, for example, \cite[Section 5]{ConwayHardinSloane1996}) that the Grassmannian $\Grass{m}{d}$ can be isometrically embedded in $\R^D$, with $D=\binom{d+1}{2}-1$. In the case of the Grassmannian $\Grass{2}{4}$, this means that there exists an isometric embedding $\varphi\from \Grass{2}{4}\to \embeddedGrass{2}{4}$, where $\embeddedGrass{2}{4}\contained \R^9$, of the Grassmannian into an Euclidean space so that the chordal distance is precisely the Euclidean distance in the embedding.

In recent years, the Grassmann manifold has been explored from different perspectives and for multiple purposes; see, for example, \cites{BachocCoulangeonNebe2002,
BachocBannaiCoulangeon2004,
Bachoc2006, 
BachocBenHaimLitsyn2008, 
BregerEhler2017,
BregerEhlerGraf2017,
BregerEhlerGraef2018,
BregerEhlerGraefPeter2017,
EhlerGraefNeumayerSteidl2021,
ConwayHardinSloane1996,
CalderbankHardinRainsShorSloane1999, DhillonHeathStrohmerTropp2008,
AlvarezVizosoBeltranCuevasSantamariaTucekPeters2023,
CuevasAlvarezVizosoBeltranSantamariaTucekPeters2022, CuevasAlvarezVizosoBeltranSantamariaTucekPeters2023}.
However, to the best of our knowledge, no systematic study of the Riesz and logarithmic energies on the Grassmannian has been conducted prior to this work.


\subsection{Energies on the Grassmann manifold}

Among the different criteria used to study the distribution of a set of points on a particular space, one that has garnered significant attention over the past decades involves computing certain repulsive discrete energies associated with the given set (see the monograph \cite{BorodachovHardinSaff2019} for a comprehensive survey on the subject). In this work, we consider pairwise interaction energies given by symmetric and lower semicontinuous kernels $K\from \Grass{2}{4}\times \Grass{2}{4}\to (-\infty,\infty]$ on the Grassmann manifold. For a collection $\omega_N=\set{P_1,\dotsc,P_N}\contained\Grass{2}{4}$, we define the \emph{discrete $K$-energy} of $\omega_N$ as
\begin{equation*}
	\Energy_K(\omega_N)=\sum_{i\neq j} K(P_i,P_j).
\end{equation*}
We denote
\begin{equation*}
    \minEnergy_K(\Grass{2}{4},N)\defined\min \Energy_K(\omega_N),
\end{equation*}
where the minimum is taken over all the collections of $N$ elements on $\Grass{2}{4}$.

Analogous definitions can be given in the continuous setting. For a Borel probability measure $\mu$ supported on the Grassmannian $\Grass{2}{4}$, the \emph{continuous $K$-energy} of $\mu$ is defined as
\begin{equation*}
    \contEnergy_{K}[\mu]=\iint\limits_{\mathclap{\Grass{2}{4}\times\Grass{2}{4}}} K(P,Q)\dd\mu(P)\dd\mu(Q).
\end{equation*}
We define the \emph{Wiener constant} $\Wiener_K(\Grass{2}{4})$ as the smallest such energy, that is,
\begin{equation*}
    \Wiener_K(\Grass{2}{4})\defined\inf\contEnergy_K[\mu],
\end{equation*}
where the infimum is taken over all the probability measures supported on $\Grass{2}{4}$. A measure $\mu^*$ is called an \emph{equilibrium measure} of $\contEnergy_K$ if $\contEnergy_K[\mu^*]=\Wiener_K(\Grass{2}{4})$.

Prominent examples of symmetric and lower semicontinuous kernels are the Riesz and logarithmic kernels, which are the primary focus of this work. For the Grassmannian $\Grass{2}{4}$, we define the Riesz $s$-kernel, for $s>0$, as
\begin{equation}\label{eq:Riesz-kernel}
    K_s(P,Q)=\frac{1}{\chordalGrass(P,Q)^s},
\end{equation}
and the logarithmic kernel as
\begin{equation}\label{eq:log-kernel}
    K_{\log}(P,Q)=-\log\chordalGrass(P,Q).
\end{equation}
We denote $\Energy_{s}\defined \Energy_{K_s}$, $\minEnergy_{s}\defined \minEnergy_{K_s}$, $\contEnergy_{s}\defined \contEnergy_{K_s}$, and $\Wiener_{s}\defined \Wiener_{K_s}$. From now on, following well-established conventions in the literature (see \cite[Definition 2.2.2]{BorodachovHardinSaff2019}), we will refer to the logarithmic energy as the $s=\log$ case of the Riesz energy.

The minimum discrete and continuous energies are related through the following fundamental result. 

\begin{theorem}[{see \cite[Theorem 4.2.2]{BorodachovHardinSaff2019}}]\label{thm:fundamentalThm}
    Let $K$ be a symmetric and lower semicontinuous kernel $K$ on $\Grass{2}{4}$. Then,
    \begin{equation*}
        \lim_{N\tendsto\infty}\frac{\minEnergy_K(\Grass{2}{4},N)}{N^2}=\Wiener_K(\Grass{2}{4}).
    \end{equation*}
    Moreover, if $(\omega_N)_{N=2}^{\infty}$ is any sequence of $N$-point configurations on $\Grass{2}{4}$ satisfying 
    \begin{equation*}
    	\frac{\Energy_K(\omega_N)}{N^2}\tendsto \Wiener_K(\Grass{2}{4})\quad \text{as $N\tendsto\infty$},
    \end{equation*}    
    and $\nu$ is a weak$^*$ limit point of the sequence of normalized counting measures 
    \begin{equation*}
    	\nu(\omega_N)\defined \frac{1}{N}\sum_{x\in\omega_{N}}\delta_x,
    \end{equation*}
    where $\delta_x$ is the Dirac measure centered at $x$, then $\nu$ is an equilibrium measure of $\contEnergy_K$. In particular, this holds for any sequence $(\omega_N^*)_{N=2}^{\infty}$ of optimal $K$-energy $N$-point configurations on $\Grass{2}{4}$.
\end{theorem}

In the case of the Riesz energy, it follows from \cite[Theorem 4.3.1]{BorodachovHardinSaff2019}) that $\Wiener_s(\Grass{2}{4})=\infty$ for $s\geq 4=\dim(\Grass{2}{4})$. In other words, we have $\contEnergy_s[\mu]=\infty$ for every Borel probability measure $\mu$ supported on $\Grass{2}{4}$, and hence every such measure is an equilibrium measure. In this case, which is called \emph{hypersingular}, minimizers of the Riesz $s$-energy are known to be asymptotically uniformly distributed, that is, the corresponding sequence of normalized counting measures converges (in the weak$^*$ sense) to the uniform measure $\sigma$ (see \cite[Theorem 8.5.2]{BorodachovHardinSaff2019} for the case $s>4$ and \cite[Theorem 9.5.4]{BorodachovHardinSaff2019} for the case $s=4$). Moreover, for $s=4$ the leading term of the minimal discrete $4$-energy is known (see \cite[Theorem 9.5.4]{BorodachovHardinSaff2019}):
\begin{equation}\label{eq:leading-term-hypersingular}
    \lim_{N\tendsto\infty}\frac{\minEnergy_{4}(\Grass{2}{4},N)}{N^2\log{N}}=\frac{\beta_4}{\vol(\Grass{2}{4})}=\frac{\pi^2/2}{2\pi^2}=\frac{1}{4},
\end{equation}
where $\beta_4$ is the volume of the unit ball in $\R^4$ and $\vol(\Grass{2}{4})$ is the volume of the Grassmannian $\Grass{2}{4}$ (see \cite[Eq. (1.4.11)]{Chikuse2003}).

In contrast, from \cite[Theorem 4.3.3]{BorodachovHardinSaff2019} we have that $\Wiener_s(\Grass{2}{4})<\infty$ for $0<s<4$ and $s=\log$. This case is called \emph{singular} and it is the subject of classical potential theory. In this range of values of $s$, however, the uniformity of the minimizers is not guaranteed. From \cref{thm:fundamentalThm}, we know that the minimal discrete $s$-energy in this case is of the form
\begin{equation*}
    \minEnergy_s(\Grass{2}{4},N)=\Wiener_s(\Grass{2}{4})N^2+o(N^2).
\end{equation*}
Improving our knowledge of the $o(N^2)$ term is one of the fundamental goals of this work.

If we denote by $\M$ a compact, connected two-point homogeneous space (that is, a sphere or a projective space), it has been proved in \cite{AndersonDostertGrabnerMatzkeStepaniuk2023} that the continuous Riesz energy is uniquely minimized by the uniform measure in $\M$ for $0<s<\dim(\M)$ and $s=\log$. Additionally, the authors prove that the next-order term is of the order of $N^{1+s/\dim(\M)}$ for $s>0$ and of the order of $N\log N$ for $s=\log$. In this work, we establish similar results for the Grassmannian $\Grass{2}{4}$ (see \cref{sec:main-results}).

In general, finding collections of points that minimize the Riesz and logarithmic energies is extremely challenging even in simple spaces; see, for instance, the survey \cite{BrauchartGrabner2015} for the case of the sphere. In fact, in most cases our knowledge of the minimal energy is incomplete. A related open problem consists then in giving the asymptotic expansion of these minimal energies in terms of the number of points as precisely as possible. In this sense, obtaining precise asymptotic upper and lower bounds for the minimal energies is essential. Regarding the upper bounds, a common strategy is to study the energy of specific configurations in the space of interest, as this energy will always be greater than or equal to the minimal one. For deterministic sets of points, computing the asymptotic expansion of the energy has proved to be a very difficult task. An alternative approach involves studying random configurations of points. In this context, a newly developed technique that has yielded excellent results is the use of \emph{determinantal point processes}; see, for example, \cites{AlishahiZamani2015, BeltranMarzoOrtegaCerda2016, MarzoOrtegaCerda2017, BeltranEtayo2019, BeltranFerizovic2020, Hirao2021, AndersonDostertGrabnerMatzkeStepaniuk2023}.

\subsection{Determinantal point processes}

Determinantal point processes are random point processes closely related to random matrix theory.
Here, we provide a brief introduction to the topic, covering the basic concepts that we will use. For a more comprehensive reference,
we recommend \cite[Chapter 4]{HoughKrishnapurPeresVirag2009}.

Let $\Lambda$ be a locally compact, Polish topological space with a Radon measure $\mu$. A simple point process $\mathfrak{X}$ of $n$ points in $\Lambda$ is a random variable taking values in the space of $n$-point subsets of $\Lambda$.

The joint intensities (if any exists, as only some point processes have them) are functions $\rho_{k}\from \Lambda^k\to[0,\infty)$, with $k\geq1$, such that for any family of mutually disjoint subsets $D_1,\dotsc,D_k$ of $\Lambda$ we have
\begin{equation}
  \expectation[\Big]{x\sim\mathfrak{X}}{\prod_{i=1}^k \card(x\cap D_i)}
  =
  \int_{\prod D_i}\rho_k(x_1,\dotsc,x_k)\dd\mu(x_1)\dotsb\dd\mu(x_k).
\end{equation} 
By $x\sim\mathfrak{X}$ we mean that $x=\set{x_1,\dotsc,x_n}$ is a subset of $n$ elements of $\Lambda$ sampled from the point process $\mathfrak{X}$.

From \cite[Formula (1.2.2)]{HoughKrishnapurPeresVirag2009}, for any measurable function $\phi\from \Lambda^{k} \to [0, \infty)$ the following equality holds:
\begin{equation}\label{eq:joint-intensities}
	\expectation[\Big]{x\distributed\mathfrak{X}}{\sum_{\substack{i_1,\dotsc,i_k\\ \text{distinct}}}\!\!\phi(x_{i_1},\dotsc,x_{i_k})}=\int\limits_{\Lambda^k}\phi(y_1,\dotsc,y_k)\rho_k(y_1,\dotsc,y_k)\dd\mu(y_1)\dotsb\dd\mu(y_k).
\end{equation}
If there exists a measurable function $\calK\from \Lambda\times\Lambda\to\C$ such that these joint intensity functions can be written as 
\begin{equation}
    \rho_{k}(x_{1},\dotsc,x_{k}) = \det(\calK(x_{i},x_{j}))_{1\leq i,j\leq k},
\end{equation}
then we say that $\mathfrak{X}$ is a \gls{dpp} with kernel $\calK$. A particularly suitable collection of such processes is obtained by choosing $\calK$ as the reproducing kernel of an $n$-dimensional subspace $H$ of the Hilbert space $\LL{2}(\Lambda,\C)$. Recall that the reproducing kernel of $H$ is the only continuous, hermitian, positive definite function $\calK_H\from \Lambda\times\Lambda\to\C$ such that $\calK_H(\cdot,x)\in H$ and
\begin{equation}
    f(x)=\innerprod{f}{\calK_H(\cdot,x)}=\int_{\Lambda}f(y)\calK_H(x,y)\dd y,\qquad \forall x\in\Lambda,\ \forall f\in H.
\end{equation}
Given any orthonormal basis $\varphi_{1},\dotsc,\varphi_{n}$ of $H$, we have
\begin{equation*}
     \calK_H(x,y)=\sum_{i=1}^n\varphi_i(x)\overline{\varphi_i(y)},
\end{equation*}
and we say that $\calK_H$ is a projection kernel of trace $n$.	

An existence theorem for determinantal point processes was independently proved by Odile Macchi and Alexander Soshnikov (see \cites{Macchi1975, Soshnikov2000}). The full version of the theorem can be found in \cite[Theorem 4.5.5]{HoughKrishnapurPeresVirag2009}. Here we present a weaker result (this is all we need in this paper) that is a direct consequence of the Macchi--Soshnikov Theorem.

\begin{proposition}\label{prop:expectation-dpp}		
	Let $\Lambda$ be a locally compact, Polish topological space with a Radon measure $\mu$ and let $H \subset \LL{2}(\Lambda,\C)$ have dimension $n$. Let $\calK_H$ be the reproducing kernel of $H$. Then, there exists a point process $\mathfrak{X}_H$ in $\Lambda$ of $n$ points with associated joint intensity functions
	\begin{equation*}
		\rho_{k}(x_{1},\dotsc,x_{k}) = \det(\calK_H(x_{i}, x_{j}))_{1\leq i,j\leq k}.
	\end{equation*}
	In particular, for any measurable function $f\from \Lambda \times \Lambda \to [0,\infty)$ we have
	\begin{equation*}
		\expectation[\Big]{x \sim \mathfrak{X}_H}{
		\displaystyle\sum_{i \neq j} f(x_{i}, x_{j})}  
		= \iint\limits_{\Lambda\times\Lambda}
		\left( \calK_H(x,x)\calK_H(y,y) - \abs{\calK_H(x,y)}^{2} \right) f(x,y)\dd\mu(x)\dd\mu(y).
	\end{equation*}
	We will call $\mathfrak{X}_H$ a projection determinantal point process with kernel $\calK_H$.
\end{proposition}

\begin{remark}\label{remark:KH-xx}
	Under the hypotheses of Proposition \ref{prop:expectation-dpp}, from \eqref{eq:joint-intensities} with $\phi\equiv 1$ and $k=1$ we have
	\begin{equation*}
	    n=\expectation{x\sim\mathfrak{X}_H}{n} 
	    = \int_{\Lambda} \calK_H(x,x)\dd\mu(x).
	\end{equation*}
	In particular, if $\calK_H(x,x)$ is constant and $\mu$ is a probability measure, then we must have $\calK_H(x,x)=n$.
\end{remark}

\subsection{Decomposition of $\LL{2}(\Grass{2}{4})$ and generalized Jacobi polynomials}\label{sec:L2-generalized-Jacobi}

In this section, we present some classical results about the Hilbert space $\LL{2}(\Grass{2}{4})$ of square-integrable functions on the Grassmannian $\Grass{2}{4}$:
\begin{equation*}
	\LL{2}(\Grass{2}{4})=\set[\Big]{f\from \Grass{2}{4}\to \C\st \int_{\Grass{2}{4}}\abs{f(P)}^2\dd\sigma(P)<\infty}.
\end{equation*}
We introduce the following notation:
\begin{align*}
	\Partitions&=\set{\tau =(\tau_1,\tau_2)\in \N^2\st \tau_1\geq \tau_2\geq  0},\\
	\Partitions^*&=\set{\tau\in \Partitions \st \tau\neq (0,0)}.
\end{align*}
We will refer to the elements of $\Partitions$ as \emph{integer partitions} or simply \emph{partitions}. The \emph{degree} of a partition $\tau = (\tau_1,\tau_2)$ is $\DegreePart{\tau}\defined\tau_1+\tau_2$. Let $\Laplacian$ be the Laplace--Beltrami operator on the Grassmannian $\Grass{2}{4}$. We follow the convention of choosing the sign of the operator so that the eigenvalues of $\Laplacian$ are nonnegative. It is known that the eigenvalues and eigenfunctions of $\Laplacian$ are indexed by partitions $\tau\in\Partitions$. This can be derived using tools from representation theory; see, for example, \cite{JamesConstantine1974} or \cite[Section 3]{BachocCoulangeonNebe2002} for a more in-depth discussion of this topic. As in the case of any compact manifold, the space $\LL{2}(\Grass{2}{4})$ can be decomposed as the direct sum of the eigenspaces of the Laplace--Beltrami operator, that is,
\begin{equation*}
	\LL{2}(\Grass{2}{4})=\overline{\bigdirectsum_{\tau\in\Partitions}H_{\tau}},\quad H_{\tau}\orthogonal H_{\tau'},\quad \tau\neq \tau',
\end{equation*}
where $H_{\tau}$ is the eigenspace of the Laplace--Beltrami operator associated with the eigenvalue $\lambda_{\tau}$. The dimension of the space $H_{\tau}$ is
\begin{equation}\label{eq:d-tau}
	d_{\tau}\defined \dim(H_{\tau})=(2-\delta_{0,\tau_2})(2(\tau_1+\tau_2)+1)(2(\tau_1-\tau_2)+1),
\end{equation}
where $\delta_{0,\tau_2}$ is the Kronecker delta (see, for example, \cite[Eq. (2.3)]{Davis1999b}). The reproducing kernel of the space $H_{\tau}$ is given by
\begin{equation}\label{eq:K-tau-1}
	\calK_{\tau}(P,Q)=d_{\tau}\JacobiGrass{\tau}(y_1(P,Q),y_2(P,Q)), \qquad P,Q\in\Grass{2}{4},
\end{equation}
where $y_j(P,Q)=\cos^2(\theta_j(P,Q))$, for $j=1,2$, and the $\JacobiGrass{\tau}$ are \emph{generalized Jacobi polynomials}. These polynomials were introduced by James and Constantine \cite{JamesConstantine1974} and, in the context of representation theory, they are the \emph{zonal spherical functions} of the Grassmannian; see \cites{JamesConstantine1974,Davis1999a,Davis1999b}. In the case of $\Grass{2}{4}$, the polynomials $\JacobiGrass{\tau}\from[0,1]^2\to\R$ are symmetric polynomials of degree $\DegreePart{\tau}$ in the variables $y_j$. They form a complete orthogonal system in $[0,1]^2$ with respect to the measure
\begin{equation}\label{eq:measure-y}
	\dd \mu(y_1,y_2)=\frac{1}{2} y_1^{-1/2}(1-y_1)^{-1/2}y_2^{-1/2}(1-y_2)^{-1/2}\abs{y_1-y_2}\dd y_1 \dd y_2.
\end{equation}
These polynomials satisfy $\JacobiGrass{(0,0)}\identically 1$ and they are normalized so that $\JacobiGrass{\tau}(1,1)=1$. Using the substitution given by the variables $\xi_+$ and $\xi_{-}$ introduced in \eqref{eq:xi_+-}, the measure in \eqref{eq:measure-y} is transformed to the following measure in the region $-\xi_{-}\leq \xi_{+}\leq \xi_{-}$, $0\leq \xi_{-}\leq 1$:
\begin{equation}\label{eq:measure-xi}
	\dd\eta(\xi_{+},\xi_{-})=\dd\xi_{+}\dd\xi_{-}.
\end{equation}
Using simple symmetry arguments, it can be checked that the polynomials $\JacobiGrass{\tau}$ are orthogonal with respect to the measure \eqref{eq:measure-xi} in $[-1,1]^2$. We know from \cite[Eq. (4.13)]{Davis1999b} that the polynomials $\JacobiGrass{\tau}$, when written in terms of the variables $\xi_{+}$ and $\xi_{-}$, admit the following simple expression in terms of classical Legendre polynomials (see \cref{appendix:orthogonal-polynomials} for more information about these polynomials):
\begin{equation}\label{eq:JacobiGrass-xi}
	\JacobiGrass{\tau}(\xi_{+},\xi_{-})=\frac{1}{2}\Bigl(\Legendre{\tau_1+\tau_2}(\xi_{+})\Legendre{\tau_1-\tau_2}(\xi_{-})+\Legendre{\tau_1+\tau_2}(\xi_{-})\Legendre{\tau_1-\tau_2}(\xi_{+})\Bigr),
\end{equation}
where $\Legendre{\tau_1+\tau_2}$ and $\Legendre{\tau_1-\tau_2}$ are Legendre polynomials of degree $\tau_1+\tau_2$ and $\tau_1-\tau_2$, respectively. Then, the reproducing kernel $\calK_{\tau}$ in \eqref{eq:K-tau-1} can be expressed in the variables $\xi_{+}$ and $\xi_{-}$ as
\begin{align}
	\calK_{\tau}(P,Q)&=\calK_{\tau}(\xi_{+}(P,Q),\xi_{-}(P,Q))=d_{\tau}\JacobiGrass{\tau}(\xi_{+},\xi_{-})\notag\\
	&=\frac{d_{\tau}}{2}\Bigl(\Legendre{\tau_1+\tau_2}(\xi_{+})\Legendre{\tau_1-\tau_2}(\xi_{-})+\Legendre{\tau_1+\tau_2}(\xi_{-})\Legendre{\tau_1-\tau_2}(\xi_{+})\Bigr), \qquad P,Q\in\Grass{2}{4}.\label{eq:K-tau-2}
\end{align}
See also \cite[Eq. (F.1)]{DickEhlerGraefKrattenthaler2023}. 

Given a function $h\in\LL{2}([-1,1]^2,\dd\eta)$, we can consider its expansion in terms of generalized Jacobi polynomials:
\begin{equation*}
	h(x,y)=\sum_{\tau\in\Partitions}\hat{h}(\tau)\JacobiGrass{\tau}(x,y),
\end{equation*}
where the convergence is \emph{a priori} in the $\LL{2}$ sense. The \emph{generalized Fourier--Jacobi coefficients} $\hat{h}(\tau)$ are then given by
\begin{equation}\label{eq:generalizedJacobiCoefficients}
	\hat{h}(\tau)=\frac{\innerprod{h}{\JacobiGrass{\tau}}_{\eta}}{\norm{\JacobiGrass{\tau}}_{\eta}^2}=\frac{1}{\norm{\JacobiGrass{\tau}}_{\eta}^2}\int_{-1}^{1}\int_{-1}^{1}h(x,y)\JacobiGrass{\tau}(x,y)\dd x\dd y,
\end{equation}
where
\begin{equation*}
	\norm{\JacobiGrass{\tau}}_{\eta}^2=\innerprod{\JacobiGrass{\tau}}{\JacobiGrass{\tau}}_{\eta}=\int_{-1}^{1}\int_{-1}^{1}\JacobiGrass{\tau}(x,y)^2\dd x\dd y.
\end{equation*}
In particular, since $\JacobiGrass{(0,0)}\identically 1$, we have
\begin{equation}\label{eq:generalizedJacobiCoefficients-0-0}
	\hat{h}(0,0)=\frac{1}{4}\int_{-1}^{1}\int_{-1}^{1}h(x,y)\dd x\dd y.
\end{equation}
We conclude this section with the classical \emph{addition formula} for zonal polynomials. Given any orthonormal basis $\set{e_1,\dotsc,e_{d_{\tau}}}$ of the subspace $H_{\tau}$ of $\LL{2}(\Grass{2}{4})$, we have
\begin{equation}\label{eq:addition-formula-G24}
	\JacobiGrass{\tau}(\xi_{+}(P,Q),\xi_{-}(P,Q))=\frac{1}{d_{\tau}}\sum_{i=1}^{d_{\tau}} e_i(P)\conj{e_i(Q)},
\end{equation}
see \cite[Lemma 4]{Roy2010}. From this, the positive definiteness of the generalized Jacobi polynomials follows (see \cite[Proposition 2.1]{Bachoc2006}).

\subsection{Notation}

The following notations will be used throughout this work. Let $f$ and $g$ be two real-valued functions.

\begin{itemize}

    \item $f(x)\lesssim g(x)$ means that there exists a constant $C>0$ such that $f(x)\leq Cg(x)$ for all $x$.
    
     \item We say that $f(x)=O(g(x))$ as $x\tendsto\infty$ if there exists a positive constant $C$ and $x_0\in \R$ such that $\abs{f(x)}\leq C\abs{g(x)}$ for all $x>x_0$. If $g(x)\neq 0$, we can express this condition as
     \begin{equation*}
         \limsup_{x\tendsto\infty} \abs[\bigg]{\frac{f(x)}{g(x)}}\leq C.
     \end{equation*}
     
    \item The expression $f(x)\asymp g(x)$ as $x\tendsto\infty$ means that $f(x)=O(g(x))$ and $g(x)=O(f(x))$, that is, $f(x)$ and $g(x)$ are asymptotically of the same order.

    \item We say that $f(x)=o(g(x))$ as $x\tendsto\infty$ if
    \begin{equation*}
        \lim_{x\tendsto\infty}\abs[\bigg]{\frac{f(x)}{g(x)}}=0.
    \end{equation*}
    
    \item Finally, $f(x)\sim g(x)$ as $x\tendsto\infty$  means that 
    \begin{equation*}
        \lim_{x\tendsto\infty}\frac{f(x)}{g(x)}=1.
    \end{equation*}
    
\end{itemize}

We have defined this asymptotic notation for $x\tendsto\infty$, but one can substitute $\infty$ by any real number $a$ and the definitions hold in all cases considering neighbourhoods of $a$ when necessary.

\section{Main results}\label{sec:main-results}

Our first main result is analogous to similar results proved for the case of two-point homogeneous spaces, namely that the continuous Riesz and logarithmic energies are uniquely minimized by the uniform measure. In this paper, we establish this result for the Grassmannian $\Grass{2}{4}$. 

\begin{theorem}\label{thm:minimizers-Riesz-G24}
	The logarithmic energy $\contEnergy_{\log}$ and the Riesz $s$-energy $\contEnergy_{s}$ on $\Grass{2}{4}$, for $0<s<4$, are uniquely minimized by the uniform measure $\sigma$, with
	\begin{align*}
		\Wiener_s(\Grass{2}{4})&=\contEnergy_s[\sigma]=\hypergeom{3}{2}\biggl(\frac{1}{2},\frac{s}{4}+\frac{1}{2},\frac{s}{4}; \frac{3}{2},\frac{3}{2};1\biggr),\\
		\Wiener_{\log}(\Grass{2}{4})&=\contEnergy_{\log}[\sigma]=1-\frac{\pi^2}{16}-\frac{\log{2}}{2},
	\end{align*}
	where $\hypergeom{3}{2}$ is a generalized hypergeometric function \embparen{see \cref{appendix:orthogonal-polynomials}}. Moreover, if $(\omega_N)_{N=2}^{\infty}$ is a sequence of configurations such that
	\begin{equation*}
		\lim_{N\tendsto\infty}\frac{\Energy_{s}(\omega_{N})}{N^2}=\Wiener_s(\Grass{2}{4}),
	\end{equation*}
    where $0<s<4$ or $s=\log$, then the sequence of normalized counting measures
	\begin{equation*}
		\nu(\omega_{N})=\frac{1}{N}\sum_{x\in\omega_N}\delta_{x}
	\end{equation*}
	converges weakly$^*$ to $\sigma$. In particular, this holds for a sequence $(\omega_N^*)_{N=2}^{\infty}$ of minimizers for the discrete problem. 
\end{theorem}

The remaining main results provide asymptotic upper and lower bounds for the minimal discrete Riesz and logarithmic energies when the number of points tends to infinity.

\begin{theorem}\label{thm:bounds-Riesz}
    For $0<s<4$, there exist constants $C_s,C_s'<0$ such that, for $N$ sufficiently large,
    \begin{equation*}
        C_sN^{1+s/4}\leq \minEnergy_s(\Grass{2}{4},N)-\Wiener_s(\Grass{2}{4})N^2\leq C_s'N^{1+s/4}.
    \end{equation*}
\end{theorem}

In the logarithmic case, we determine exactly the next-order term.

\begin{theorem}\label{thm:bounds-log}
    The minimal logarithmic energy on the Grassmannian $\Grass{2}{4}$ satisfies
    \begin{equation*}
        \minEnergy_{\log}(\Grass{2}{4},N)=\Wiener_{\log}(\Grass{2}{4})N^2-\frac{1}{4}N\log{N}+O(N),\qquad N\tendsto\infty.
    \end{equation*}
\end{theorem}

\begin{theorem}\label{thm:bounds-hypersingular}
    The minimal Riesz $4$-energy on the Grassmannian $\Grass{2}{4}$ satisfies
    \begin{equation*}
        \minEnergy_{4}(\Grass{2}{4},N)=\frac{N^2\log{N}}{4}+O(N^2),\qquad N\tendsto\infty.
    \end{equation*}
\end{theorem}

The lower bounds in \cref{thm:bounds-Riesz,thm:bounds-log,thm:bounds-hypersingular} are obtained through linear programming techniques in the spirit of \cites{Wagner1990, Brauchart2006, AndersonDostertGrabnerMatzkeStepaniuk2023}. See also \cite[Chapter 5]{BorodachovHardinSaff2019} for more information about these techniques. Note that in the cited references the corresponding kernels depend only on one parameter, namely the distance between points, and so everything is essentially one-dimensional; in the case of the Grassmannian $\Grass{2}{4}$, however, the kernels are functions of the two principal angles. To overcome this limitation, we develop a bivariate version of this kind of reasoning. We provide the proof of these lower bounds in \cref{sec:lower-bounds}.
 
To prove the upper bounds in \cref{thm:bounds-Riesz,thm:bounds-log,thm:bounds-hypersingular} we use determinantal point processes. The order of those bounds in  \cref{thm:bounds-Riesz,thm:bounds-log} is established using jittered sampling. However, to obtain bounds with explicit constants for the next order terms we need to use more sophisticated determinantal point processes. The Grassmannian $\Grass{2}{4}$ endowed with the uniform measure $\sigma$ is a locally compact, Polish topological space. Therefore, from \cref{prop:expectation-dpp}, to define a determinantal point process of $N$ points on $\Grass{2}{4}$ it suffices to choose a suitable subspace of $\LL{2}(\Grass{2}{4})$ of dimension $N$. Recall from \cref{sec:L2-generalized-Jacobi} the decomposition of $\LL{2}(\Grass{2}{4})$ as a direct sum of eigenspaces of the Laplace--Beltrami operator:
\begin{equation*}
    \LL{2}(\Grass{2}{4})=\overline{\bigdirectsum_{\tau\in\Partitions} H_{\tau}}.
\end{equation*}
For every natural number $k>0$, we consider the \gls{dpp} associated with the subspace $H_k=\bigdirectsum_{\DegreePart{\tau}\leq k}H_{\tau}$ of $\LL{2}(\Grass{2}{4})$. The reproducing kernel of $H_k$ is then given by
\begin{equation}\label{eq:kernel-sum}
    \calK_k(P,Q)=\sum_{\DegreePart{\tau}\leq k} \calK_{\tau}(P,Q),
\end{equation}
where $\calK_{\tau}$ is the reproducing kernel of the eigenspace $H_{\tau}$. From \eqref{eq:K-tau-2}, the kernels $\calK_{\tau}(P,Q)$ are given by bivariate orthogonal polynomials in the variables $\xi_{+}(P,Q)$ and $\xi_{-}(P,Q)$. In this work, we derive a simple, explicit expression for the kernel \eqref{eq:kernel-sum} using a multivariate version of the Christoffel--Darboux formula. 

\begin{theorem}\label{thm:reproducing-kernel-grassmannian-dpp}
	For every natural number $k>0$, the reproducing kernel of $H_k$ is
	\begin{equation*}
		\calK_k(P,Q)=\calK_k(\xi_{+},\xi_{-})=\Gegenbauer{k}{3/2}(\xi_{+})\Gegenbauer{k}{3/2}(\xi_{-})+\Gegenbauer{k-1}{3/2}(\xi_{+})\Gegenbauer{k-1}{3/2}(\xi_{-}),
	\end{equation*}
    where $\Gegenbauer{k}{3/2}$ is the Gegenbauer polynomial of degree $k$ and parameter $3/2$.
\end{theorem}

\begin{remark}
    The Grassmannian $\Grass{2}{4}$ can be identified with the set of orthogonal projectors on $\R^4$ of rank $2$:
    \begin{equation*}
        \Grass{2}{4}=\set{P\in \R^{4\times 4}\st P\transpose =P,\ P^2=P,\ \rank(P)=2}.
    \end{equation*}
    With this definition of the Grassmannian, the space $\Pol_k(\Grass{2}{4})$ of polynomials of degree at most $k$ on $\C^{4\times 4}$ restricted to $\Grass{2}{4}$ satisfies
    \begin{equation*}
        \Pol_k(\Grass{2}{4})=\bigdirectsum_{\DegreePart{\tau}\leq k}H_{\tau}=H_k,
    \end{equation*}
    see \cite[Section 4.2]{BregerEhlerGraf2017}. This space has been extensively studied; see, for example, \cites{BachocEhler2013, BregerEhlerGraf2017, BodmannEhlerGraef2018, EhlerGraf2019}. Thus, \cref{thm:reproducing-kernel-grassmannian-dpp} provides a simple expression for the reproducing kernel of this polynomial space.
\end{remark}

The dimension of $H_k$ is given by
\begin{equation*}
    d_k=\sum_{\DegreePart{\tau}\leq k} d_{\tau}=\frac{1}{2}(k+1)^2(k^2+2k+2).
\end{equation*}
Therefore, for every $k> 0$, the \gls{dpp} associated with the subspace $H_k$ samples $N=d_k$ points. By analogy with the harmonic ensembles presented in \cites{BeltranMarzoOrtegaCerda2016, AndersonDostertGrabnerMatzkeStepaniuk2023}, we will call this family of \glspl{dpp} \emph{harmonic ensemble} and, to emphasize the dependence on $N$, we will denote it by $\HarmonicDPP{N}$. It is worth mentioning that in the recent paper \cite{GarciaArias2024}, the author proves that the harmonic ensemble on homogeneous manifolds has the optimal rate of convergence to the uniform measure with respect to the Wasserstein distance $W_2$.

To the best of our knowledge, this is the first time that a determinantal point process has been considered in the case of the Grassmann manifold. Note that in the paper \cite{KasselLevy2022} the authors introduce and study a class of determinantal probability measures that generalize the class of discrete determinantal point processes. These measures live on the Grassmannian of a real, complex, or quaternionic inner product space. 
Although this might seem similar, they do not introduce any determinantal point process on the Grassmanian.

From \cref{prop:expectation-dpp}, the expected Riesz and logarithmic energies of points coming from the harmonic ensemble are given by
\begin{align*}
	\expectation{x\distributed\HarmonicDPP{N}}{\Energy_s(x)}&=\iint\limits_{\mathclap{\Grass{2}{4}\times\Grass{2}{4}}} \frac{\calK_k(P,P)\calK_k(Q,Q)-\calK_k(P,Q)^2}{\chordalGrass(P,Q)^s}\dd\sigma(P)\dd\sigma(Q),\\
	\expectation{x\distributed\HarmonicDPP{N}}{\Energy_{\log}(x)}&=-\iint\limits_{\mathclap{\Grass{2}{4}\times\Grass{2}{4}}} (\calK_k(P,P)\calK_k(Q,Q)-\calK_k(P,Q)^2)\log\chordalGrass(P,Q)\dd\sigma(P)\dd\sigma(Q).
\end{align*}
Here we study the asymptotic behavior of these expected energies in terms of the number of points sampled by the \gls{dpp}. The results are summarized in the following theorems.

\begin{theorem}\label{thm:Energy-DPP-Riesz}
	The expected Riesz $s$-energy, with $0<s<4$, of the points coming from the harmonic ensemble is
	\begin{equation*}
		\expectation{x\distributed\HarmonicDPP{N}}{\Energy_s(x)}
        =\Wiener_s(\Grass{2}{4})N^2-C_sN^{1+s/4}+o(N^{1+s/4}),
	\end{equation*}
	where
	\begin{equation}\label{eq:constant-dpp}
		C_s=2^{2+3s/4}\int_0^{\infty}\int_0^{\infty} \frac{\BesselJ{1}(x)^2\BesselJ{1}(y)^2}{xy(x^2+y^2)^{s/2}}\dd x\dd y,
	\end{equation}
    and $\BesselJ{1}$ is the Bessel function of the first kind of order $1$. For $s=2$, we have
	\begin{equation*}
		C_{2}=-\frac{2^{7/2}(4-24G+3\pi)}{48\pi},
	\end{equation*}
	where $G$ is Catalan's constant.
\end{theorem}

\begin{theorem}\label{thm:Energy-DPP-log}
	The expected logarithmic energy of the points coming from the harmonic ensemble is
	\begin{equation*}
		\expectation{x\distributed\HarmonicDPP{N}}{\Energy_{\log}(x)}=\Wiener_{\log}(\Grass{2}{4})N^2-\frac{1}{4}N\log{N}+C_{\log}N+o(N),
	\end{equation*}
	where
	\begin{equation*}
		C_{\log}=\frac{1+2G}{\pi}+\frac{1}{4}-\gamma+\frac{\log{2}}{4},
	\end{equation*}
    and $\gamma$ is the Euler--Mascheroni constant.
\end{theorem}

\begin{theorem}\label{thm:Energy-DPP-hypersingular}
	The expected Riesz $4$-energy of the points coming from the harmonic ensemble is
	\begin{equation*}
		\expectation{x\distributed\HarmonicDPP{N}}{\Energy_4(x)}=\frac{N^2\log{N}}{4}+C_4N^2+o(N^2),
	\end{equation*}
	where
	\begin{multline}
		C_4=\frac{7\log 2}{4}+2\int_{0}^{1}\int_{0}^{1}
		\frac{x^2y^2-16\BesselJ{1}(x)^2\BesselJ{1}(y)^2}{xy(x^2+y^2)^2}\dd x\dd y\label{eq:C4-dpp}\\
		-32\int_{1}^{\infty}\int_{1}^{\infty}\frac{\BesselJ{1}(x)^2\BesselJ{1}(y)^2}{xy(x^2+y^2)^2}\dd x\dd  y-64\int_{1}^{\infty}\int_{0}^{1}\frac{\BesselJ{1}(x)^2\BesselJ{1}(y)^2}{xy(x^2+y^2)^2}\dd x\dd y\approx 0.991.
	\end{multline}
\end{theorem}

Comparing \cref{thm:Energy-DPP-hypersingular} with \cref{eq:leading-term-hypersingular}, we see that our \gls{dpp} gives the correct leading term in the asymptotic expansion of the minimal $4$-energy.

\Cref{thm:Energy-DPP-Riesz,thm:Energy-DPP-hypersingular,thm:Energy-DPP-log} provide asymptotic upper bounds for the corresponding minimal discrete energies, with explicit constants for the next-order terms when the number of points is of the form $N=d_k$. Using the same idea as in \cite[Corollary 2]{BeltranMarzoOrtegaCerda2016}, we are able to overcome this last restriction for $3<s<4$.

\begin{corollary}\label{cor:upper-bound-dpp-all-N}
    For any $N\geq 1$ \embparen{not necessarily of the form $d_k$}, and for $3<s<4$, we have
    \begin{equation*}
        \minEnergy_s(\Grass{2}{4},N)\leq \Wiener_s(\Grass{2}{4})N^2-C_sN^{1+s/4}+o(N^{1+s/4}),
    \end{equation*}
    where $C_s$ is the constant in \eqref{eq:constant-dpp}.
\end{corollary}

\Cref{cor:upper-bound-dpp-all-N} follows from the fact that, for $N\in(d_{k},d_{k+1})$,
\begin{equation*}
    \minEnergy_s(\Grass{2}{4},N)\leq \minEnergy_s(\Grass{2}{4},d_{k+1})\leq \expectation{x\distributed\HarmonicDPP{d_{k+1}}}{\Energy_s(x)},
\end{equation*}
and both $\expectation{x\distributed\HarmonicDPP{d_{k}}}{\Energy_s(x)}$ and $\expectation{x\distributed\HarmonicDPP{d_{k+1}}}{\Energy_s(x)}$ have the same first two asymptotic terms. More precisely, we have $d_k=k^4/2+O(k^3)$ and also $d_{k+1}=k^4/2+O(k^3)$. Then, for $3<s<4$ and $N\in(d_k,d_{k+1})$,
\begin{align*}
    \expectation{x\distributed\HarmonicDPP{d_{k+1}}}{\Energy_s(x)}&=\Wiener_s(\Grass{2}{4})\frac{1}{4}k^8+O(k^7)-C_s\frac{1}{2^{1+s/4}}k^{4+s}+o(k^{4+s})\\
    &=\Wiener_s(\Grass{2}{4})\frac{1}{4}k^8-C_s\frac{1}{2^{1+s/4}}k^{4+s}+o(k^{4+s})\\
    &=\Wiener_s(\Grass{2}{4})N^2-C_sN^{1+s/4}+o(N^{1+s/4}).
\end{align*}
It can be checked that this same argument is also valid in the case $s=4$. Therefore, we obtain the following corollary, which also proves the order of the upper bound in \cref{thm:bounds-hypersingular}.

\begin{corollary}\label{cor:upper-bound-dpp-all-N-hypersingular}
    For any $N\geq 1$ \embparen{not necessarily of the form $d_k$}, we have
    \begin{equation*}
        \minEnergy_4(\Grass{2}{4},N)\leq \frac{N^2\log{N}}{4}+C_4N^{2}+o(N^2),
    \end{equation*}
    where $C_4$ is the constant in \eqref{eq:C4-dpp}.
\end{corollary}

\subsection{Structure of the paper}

In \cref{sec:minimizing-measure}, we prove \cref{thm:minimizers-Riesz-G24}. \Cref{sec:lower-bounds} is dedicated to the proof of the lower bounds in \cref{thm:bounds-Riesz,thm:bounds-log,thm:bounds-hypersingular}. In \cref{sec:jittered}, we establish the order of the upper bounds in \cref{thm:bounds-Riesz,thm:bounds-log} using jittered sampling. In \cref{sec:reproducing-kernel-harmonic}, we prove \cref{thm:reproducing-kernel-grassmannian-dpp}. In \Cref{sec:energy-harmonic-ensemble}, we prove \cref{thm:Energy-DPP-Riesz,thm:Energy-DPP-hypersingular,thm:Energy-DPP-log}. In \cref{sec:auxiliary-results}, we gather some auxiliary results and technical lemmas used in the proofs of \cref{sec:energy-harmonic-ensemble}. Finally, \cref{appendix:orthogonal-polynomials} contains some definitions and results about orthogonal polynomials and special functions that are used throughout this work.

\section{Minimizers of the Riesz and logarithmic energies}\label{sec:minimizing-measure}

In this section we prove \cref{thm:minimizers-Riesz-G24}. The key point consists in proving that the Riesz and logarithmic kernels on $\Grass{2}{4}$ are conditionally strictly positive definite.

We denote the set of Borel probability measures on $\Grass{2}{4}$ by $\Measures(\Grass{2}{4})$, the set of finite signed Borel measures by $\SignedMeasures(\Grass{2}{4})$, and the set of finite signed Borel measures with total mass zero, that is, $\nu\in\SignedMeasures(\Grass{2}{4})$ satisfying $\nu(\Grass{2}{4})=0$, by $\ZeroMeasures(\Grass{2}{4})$. Below we recall the concept of well-defined energy of a measure $\nu\in\SignedMeasures(\Grass{2}{4})$.

\begin{definition}[{\cite[Definition 4.2.4]{BorodachovHardinSaff2019}}]\label{def:well-defined-energy}
	Let $A$ be an infinite compact metric space. Given a symmetric and lower semicontinuous kernel $K$ on $A\times A$ and two finite positive Borel measures $\mu$ and $\nu$ supported on $A$, we define their joint (or mutual) $K$-energy $\mutualEnergy{\mu}{\nu}_K$ by
	\begin{equation*}
		\mutualEnergy{\mu}{\nu}_K\defined\iint K(x,y)\dd\nu(x)\dd\mu(y).
	\end{equation*}
	When $\mu$ and $\nu$ are finite signed Borel measures supported on $A$, their joint energy is defined as
	\begin{equation*}
		\mutualEnergy{\mu}{\nu}_K\defined \mutualEnergy{\mu^+}{\nu^+}_K+\mutualEnergy{\mu^-}{\nu^-}_K-\mutualEnergy{\mu^+}{\nu^-}_K-\mutualEnergy{\mu^-}{\nu^+}_K,
	\end{equation*}
	provided that at least one of the sums  $\mutualEnergy{\mu^+}{\nu^+}_K+\mutualEnergy{\mu^-}{\nu^-}_K$ or  $\mutualEnergy{\mu^+}{\nu^-}_K+\mutualEnergy{\mu^-}{\nu^+}_K$ is finite (we say that $\mutualEnergy{\mu}{\nu}_K$ is \emph{well defined} in this case). Here $\mu=\mu^+-\mu^-$ and $\nu^+-\nu^-$ are the Jordan decompositions of the signed measures $\mu$ and $\nu$. We extend the definition of $\contEnergy_K$ to signed measures by setting
	\begin{equation*}
		\contEnergy_K[\nu]\defined \mutualEnergy{\nu}{\nu}_K.
	\end{equation*}
\end{definition}

Recall that a symmetric and lower semicontinuous kernel $K\from\Grass{2}{4}\times \Grass{2}{4}\to(-\infty,\infty]$ is called \emph{conditionally strictly positive definite} if for all $\nu\in\ZeroMeasures(\Grass{2}{4})$ whose energy is well defined, we have $\contEnergy_{K}[\nu]\geq 0$, and $\contEnergy_{K}[\nu]=0$ only if $\nu\identically 0$.

\begin{theorem}[{\cite[Theorems 4.4.5 and 4.4.8]{BorodachovHardinSaff2019}}]\label{thm:Riesz-log-cspd-R^p}
	Let $A\contained \R^p$ be an arbitrary compact set and let $x,y\in A$. Let $d=\dim(A)$ be the Hausdorff dimension of $A$. Then, the logarithmic kernel $K_{\log}(x,y)=-\log\norm{x-y}$ and the Riesz kernel $K_s(x,y)=\norm{x-y}^{-s}$, for $0<s<d$, are conditionally strictly positive definite.
\end{theorem}

The following theorem is given in \cite{BorodachovHardinSaff2019} for compact subsets of $\R^d$. The interested reader can check in \cite[Theorem 4.2.7]{BorodachovHardinSaff2019} that the proof can be easily extended to general compact metric spaces.

\begin{theorem}\label{thm:cspd-implies-uniqueness}
	Let $K$ be a symmetric and lower semicontinuous kernel on $\Grass{2}{4}$. If $K$ is conditionally strictly positive definite, then $\contEnergy_K$ has a unique minimizer in $\Measures(\Grass{2}{4})$.
\end{theorem}

\begin{corollary}\label{prop:cspd-implies-uniform-measure}
	 Let $K$ be a symmetric and lower semicontinuous kernel on $\Grass{2}{4}$. If $K$ is conditionally strictly positive definite and isometry invariant, then the uniform measure $\sigma$ is the unique minimizer of $\contEnergy_K$.    
\end{corollary}

\begin{proof}
	Analogous to \cite[Corollary 2.15]{AndersonDostertGrabnerMatzkeStepaniuk2023}.
\end{proof}

\begin{theorem}\label{thm:Riesz-log-cspd-G24}
	The logarithmic kernel $K_{\log}$ and the Riesz kernel $K_s$, for $0<s<4$, are conditionally strictly positive definite on the Grassmannian $\Grass{2}{4}$.
\end{theorem}

\begin{proof}
	Let $\mu\in \ZeroMeasures(\Grass{2}{4})$ be a measure whose energy is well defined. Assume that $\mu$ is not identically zero. We have to prove that
	\begin{equation*}
		\contEnergy_{s}[\mu]=\iint\limits_{\mathclap{\Grass{2}{4}\times \Grass{2}{4}}}\frac{1}{\chordalGrass(P,Q)^s}\dd\mu(P)\dd\mu(Q)>0.
	\end{equation*}
	As we mentioned in \cref{subsec:grassmannian}, there exists an isometric embedding $\varphi\from \Grass{2}{4}\to \embeddedGrass{2}{4}$, where $\embeddedGrass{2}{4}\contained \R^9$, of the Grassmannian into an Euclidean space so that the chordal distance is the Euclidean distance in the embedding. Since $\varphi$ is an isometry, we have
	\begin{align*}
		\iint\limits_{\mathclap{\Grass{2}{4}\times \Grass{2}{4}}}\chordalGrass(P,Q)^{-s}\dd\mu(P)\dd\mu(Q)
		&=\iint\limits_{\mathclap{\Grass{2}{4}\times \Grass{2}{4}}}\norm{\varphi(P)-\varphi(Q)}^{-s}\dd\mu(P)\dd\mu(Q)\\
		&=\iint\limits_{\mathclap{\embeddedGrass{2}{4}\times \embeddedGrass{2}{4}}}\norm{x-y}^{-s}\dd\pushforward{\mu}{\varphi}(x)\dd\pushforward{\mu}{\varphi}(y)>0,
	\end{align*}
	where $\pushforward{\mu}{\varphi}$ is the pushforward measure of $\mu$ under $\varphi$. The last inequality follows from \cref{thm:Riesz-log-cspd-R^p}. The logarithmic case is analogous.
\end{proof}

\subsection{Proof of \cref{thm:minimizers-Riesz-G24}}
    From \cref{thm:Riesz-log-cspd-G24} and \cref{prop:cspd-implies-uniform-measure}, we have that the uniform measure $\sigma$ uniquely minimizes the continuous Riesz and logarithmic energies on $\Grass{2}{4}$ for $0<s<4$. The last claim of the theorem follows from \cref{thm:fundamentalThm} and the weak$^*$ compactness of $\Measures(\Grass{2}{4})$. We now compute the continuous energies $\contEnergy_s[\sigma]$ and $\contEnergy_{\log}[\sigma]$. The continuous Riesz $s$-energy is given, for $0<s<4$, by
	\begin{equation}\label{eq:cont-Riesz-energy-1}
        \contEnergy_s[\sigma]=\iint\limits_{\mathclap{\Grass{2}{4}\times \Grass{2}{4}}} \frac{1}{\chordalGrass(P,Q)^s}\dd\sigma(P)\dd\sigma(Q).
	\end{equation}
	From \cref{lemma:integral-Grass24-2}, we have, calling $x=\xi_{+}$ and $y=\xi_{-}$, 
	\begin{align}
		\contEnergy_s[\sigma]&=\frac{1}{4}\int_{-1}^{1}\int_{-1}^{1}\frac{1}{(1-xy)^{s/2}}\dd x\dd y \label{eq:cont-Riesz-energy-2}\\
        &=\frac{1}{2(s-2)}\int_{-1}^{1}\frac{(1-y)^{1-s/2}-(1+y)^{1-s/2}}{y}\dd y.\notag
	\end{align}
	Note that
	\begin{equation*}
		\frac{(1-y)^{1-s/2}-(1+y)^{1-s/2}}{y}=(-2)\sum_{n=0}^{\infty}\binom{1-s/2}{1+2n}y^{2n}.
	\end{equation*}
	This power series converges in $(-1,1)$. Since a power series converges uniformly within its radius of convergence, we can integrate term by term the previous expression, obtaining
	\begin{equation*}
		\int_{-1}^{1}\frac{(1-y)^{1-s/2}-(1+y)^{1-s/2}}{y}\dd y
		=2(s-2)\hypergeom{3}{2}\biggl(\frac{1}{2},\frac{s}{4}+\frac{1}{2},\frac{s}{4};\frac{3}{2};\frac{3}{2};1\biggr).
	\end{equation*}
	Hence, we conclude that
	\begin{equation*}
		\contEnergy_s[\sigma]
		=\hypergeom{3}{2}\biggl(\frac{1}{2},\frac{s}{4}+\frac{1}{2},\frac{s}{4};\frac{3}{2};\frac{3}{2};1\biggr).
	\end{equation*}
	Finally, we compute the continuous logarithmic energy. Recall that
    \begin{equation}\label{eq:cont-log-energy}
        \contEnergy_{\log}[\sigma]=-\iint\limits_{\mathclap{\Grass{2}{4}\times \Grass{2}{4}}} \log{\chordalGrass(P,Q)}\dd\sigma(P)\dd\sigma(Q).
    \end{equation}
    From \cref{lemma:integral-Grass24-2}, we have, calling $x=\xi_{+}$ and $y=\xi_{-}$, 
	\begin{equation*}
		\contEnergy_{\log}[\sigma]=-\frac{1}{8}\int_{-1}^{1}\int_{-1}^{1}\log(1-xy)\dd x\dd y
        =1-\frac{\pi^2}{16}-\frac{\log{2}}{2}.\pushQED{\qed}\qedhere
	\end{equation*}

\section{Lower bounds for the Riesz and logarithmic energies}\label{sec:lower-bounds}

This section is dedicated to the proof of the lower bounds in \cref{thm:bounds-Riesz,thm:bounds-log}. Given a kernel of the form $K_f(P,Q)\defined f(\xi_{+}(P,Q),\xi_{-}(P,Q))$ for some function $f\from[-1,1]^2\to(-\infty,\infty]$, we denote
\begin{equation*}
	\Energy^f(\omega_N)\defined \Energy_{K_f}(\omega_N)=\sum_{\substack{P,Q\in\omega_N\\ P\neq Q}}f(\xi_{+}(P,Q),\xi_{-}(P,Q)).
\end{equation*}
Particular examples of this type of kernels are the Riesz and logarithmic kernels $K_s$ and $K_{\log}$ introduced in \eqref{eq:Riesz-kernel} and \eqref{eq:log-kernel}.

Recall that we use the following notation:
\begin{align*}
	\Partitions&=\set{\tau =(\tau_1,\tau_2)\in \N^2\st \tau_1\geq \tau_2\geq  0},\\
	\Partitions^*&=\set{\tau\in \Partitions \st \tau\neq (0,0)}.
\end{align*}

\begin{definition}
	Let $\tau\in\Partitions$ and let $\omega_N=\set{P_1,\dotsc,P_N}\contained \Grass{2}{4}$. We define the \emph{$\tau$-th moment} of $\omega_N$ as
	\begin{equation*}
		M_{\tau}(\omega_N)\defined \sum_{i,j=1}^{N}\JacobiGrass{\tau}(\xi_{+}(P_i,P_j),\xi_{-}(P_i,P_j)),
	\end{equation*}
    where $\JacobiGrass{\tau}$ is the generalized Jacobi polynomial associated with the partition $\tau$ (see \cref{eq:JacobiGrass-xi}).
\end{definition}

\begin{lemma}\label{lemma:nth-moment-nonnegative-G24}
	Let $\omega_N=\set{P_1,\dotsc,P_N}\contained \Grass{2}{4}$ be an $N$-point configuration and let $\tau\in\Partitions$. Then, the moment $M_{\tau}(\omega_N)$ is nonnegative.
\end{lemma}

\begin{proof}
	It follows from the positive definiteness of the generalized Jacobi polynomials (see \cite[Proposition 2.1]{Bachoc2006}).
\end{proof}

\begin{proposition}\label{prop:Energy-Fourier-Grassmann}
	Suppose $g\from[-1,1]^2\to \R$ has an expansion in generalized Jacobi polynomials that converges pointwise on $[-1,1]^2$, that is, there is a sequence $(\hat{g}(\tau))_{\tau}$ such that $g(x,y)=\sum_{\tau\in \Partitions}\hat{g}(\tau)\JacobiGrass{\tau}(x,y)$ for each $(x,y)\in [-1,1]^2$. If $\omega_N\contained \Grass{2}{4}$ is any $N$-point configuration for $N\geq 2$, then
	\begin{equation*}
		\Energy^g(\omega_N)=\hat{g}(0,0)N^2-g(1,1)N+\sum_{\tau\in\Partitions^*} \hat{g}(\tau)M_{\tau}(\omega_N).
	\end{equation*}
\end{proposition}

\begin{proof}
	We have
	\begin{align*}
		\Energy^g(\omega_N)&=\sum_{\substack{P,Q\in\omega_N\\P\neq Q}}g(\xi_{+}(P,Q),\xi_{-}(P,Q))\\
		&=\sum_{P,Q\in\omega_N}g(\xi_{+}(P,Q),\xi_{-}(P,Q))-\sum_{P\in\omega_N}g(\xi_{+}(P,P),\xi_{-}(P,P))\\
		&=\sum_{P,Q\in\omega_N}g(\xi_{+}(P,Q),\xi_{-}(P,Q))-Ng(1,1).
	\end{align*}
	Note that
	\begin{align*}
		\sum_{P,Q\in\omega_{N}}g(\xi_{+}(P,Q),\xi_{-}(P,Q))&=\sum_{P,Q\in\omega_N}\sum_{\tau\in\Partitions}\hat{g}(\tau)\JacobiGrass{\tau}(\xi_{+}(P,Q),\xi_{-}(P,Q))\\
		&=N^2\hat{g}(0,0)+\sum_{\tau\in\Partitions^*}\hat{g}(\tau)\sum_{P,Q\in\omega_N}\JacobiGrass{\tau}(\xi_{+}(P,Q),\xi_{-}(P,Q))\\
		&=N^2\hat{g}(0,0)+\sum_{\tau\in\Partitions^*}\hat{g}(\tau)M_{\tau}(\omega_{N}).\qedhere
	\end{align*}
\end{proof}

We now adapt \cite[Theorem 5.5.1]{BorodachovHardinSaff2019} to the case of $\Grass{2}{4}$. Given a function $f\from [-1,1]^2\to(-\infty,\infty]$, let 
\begin{multline*}
	\mathcal{A}(f)\defined\left\{g\in C([-1,1]^2)\st \text{$\hat{g}(\tau)\geq 0$ for $\tau\in\Partitions^*$ and $g(x,y)\leq f(x,y)$}\right.\\
    \left. \text{for $(x,y)\in[-1,1]^2$}\right\}.
\end{multline*}

\begin{theorem}\label{thm:LP-lower-bound}
	Suppose $f\from [-1,1]^2\to(-\infty,\infty]$ and $g\in\mathcal{A}(f)$. If $\omega_N$ is an $N$-point configuration on $\Grass{2}{4}$, then
	\begin{equation*}
		\Energy^{f}(\omega_{N})\geq \Energy^{g}(\omega_{N})\geq \hat{g}(0,0)N^2-g(1,1)N. 
	\end{equation*}
\end{theorem}

\begin{proof}
	Since $g\leq f$, we have $\Energy^{g}(\omega_{N})\leq\Energy^{f}(\omega_{N})$. Therefore, from \cref{prop:Energy-Fourier-Grassmann} and \cref{lemma:nth-moment-nonnegative-G24},
	\begin{align*}
		\Energy^{f}(\omega_{N})&\geq \Energy^{g}(\omega_{N})=\hat{g}(0,0)N^2-g(1,1)N+\sum_{\tau\in\Partitions^*}\hat{g}(\tau)M_{\tau}(\omega_N)\\
        &\geq \hat{g}(0,0)N^2-g(1,1)N. \qedhere
	\end{align*}
\end{proof}

To prove the lower bounds in \cref{thm:bounds-Riesz,thm:bounds-log} we will need the following lemma.

\begin{lemma}\label{lemma:positive-coefficients-lower-bound}
	Let $\delta>0$. Let $f_{s,\delta}\from[-1,1]^2\to\R$ be the function given by
	\begin{equation*}
		f_{s,\delta}(x,y)=
		\begin{cases}
			(1-xy+\delta)^{-s} & s>0,\\
			-\log(1-xy+\delta) & s=0.
		\end{cases}
	\end{equation*}
	Then, $\hat{f}_{s,\delta}(\tau)\geq 0$ for all $\tau\in\Partitions^*$.
\end{lemma}
 
 \begin{proof}
 	We have
 	\begin{align}
 		(1-xy+\delta)^{-s/2}&=(\delta+1)^{-s/2}\biggl(1-\frac{xy}{\delta+1}\biggr)^{-s/2}\label{eq:expansion-1}\\
 		&=(\delta+1)^{-s/2}\sum_{k=0}^{\infty}\binom{k+s/2-1}{k}\frac{x^ky^k}{(\delta+1)^k},\notag\\
 		-\log(1-xy+\delta)&=-\log(\delta+1)+\log\biggl(1-\frac{xy}{\delta+1}\biggr)\label{eq:expansion-2}\\
 		&=-\log(\delta+1)+\sum_{k=1}^{\infty}\frac{1}{k}\biggl(\frac{x^ky^k}{(\delta+1)^k}\biggr)\notag,
 	\end{align}
 	where the series in \cref{eq:expansion-1,eq:expansion-2} converge uniformly on $[-1,1]^2$. Therefore, from \eqref{eq:generalizedJacobiCoefficients} the generalized Fourier--Jacobi coefficients $\hat{f}_{s,\delta}(\tau)$ are given, for $\tau\in\Partitions^*$, by
 	\begin{equation*}
 		\frac{(\delta+1)^{-s/2}}{\norm{\JacobiGrass{\tau}}_{\eta}^2}\sum_{k=0}^{\infty}\binom{k+s/2-1}{k}\frac{1}{(\delta+1)^k}\int_{-1}^{1}\int_{-1}^{1}x^ky^k\JacobiGrass{\tau}(x,y)\dd x\dd y,
 	\end{equation*}
 	for $s>0$, and by
 	\begin{align*}
 		\MoveEqLeft-\frac{\log(\delta+1)}{\norm{\JacobiGrass{\tau}}_{\eta}^2}\int_{-1}^{1}\int_{-1}^{1}\JacobiGrass{\tau}(x,y)\dd x\dd y\\
 		&\qquad+\frac{1}{\norm{\JacobiGrass{\tau}}_{\eta}^2}\sum_{k=1}^{\infty}\frac{1}{k(\delta+1)^k}\int_{-1}^{1}\int_{-1}^{1}x^ky^k\JacobiGrass{\tau}(x,y)\dd x\dd y\\
 		&=\frac{1}{\norm{\JacobiGrass{\tau}}_{\eta}^2}\sum_{k=1}^{\infty}\frac{1}{k(\delta+1)^k}\int_{-1}^{1}\int_{-1}^{1}x^ky^k\JacobiGrass{\tau}(x,y)\dd x\dd y,
 	\end{align*}
 	for $s=\log$, where we have used that $\JacobiGrass{(0,0)}\identically 1$ together with the orthogonality of the generalized Jacobi polynomials. Since $\JacobiGrass{\tau}(x,y)$ is positive definite, we have, using the characterization of positive definiteness given by \cite[Eq. (3)]{Bochner1941}, that
 	\begin{align*}
 		\int_{-1}^{1}\int_{-1}^{1}x^ky^k\JacobiGrass{\tau}(x,y)\dd x\dd y\geq 0
 	\end{align*}
 	for all $\tau\in\Partitions^*$. Hence, the lemma follows.\qedhere
 \end{proof}

\subsection{Proof of the lower bounds in \cref{thm:bounds-Riesz,thm:bounds-log,thm:bounds-hypersingular}}

\subsubsection{Proof of the lower bound for $0\leq s<4$}

Consider the function
\begin{equation*}
	f_s(u)=\begin{cases}
		u^{-s/2} & s>0,\\
		-\frac{1}{2}\log u & s=0,
	\end{cases}
\end{equation*}
where in this proof $s=0$ corresponds to the logarithmic case. The derivatives of this function are given by
\begin{equation*}
	f_s^{(k)}(u)=(-1)^kc_{s,k}f_{s+2k}(u),
\end{equation*}
where
\begin{equation}\label{eq:def-csk}
	c_{s,k}=\begin{cases}
		1 & k=0,\\
		\Pochhammer{\frac{s}{2}}{k} & \text{$s>0$, $k>0$},\\
		\frac{1}{2}(k-1)! & \text{$s=0$, $k>0$}.
	\end{cases}
\end{equation}
By Taylor's theorem with integral form of the remainder, we have, for $u>0$,
\begin{align*}
	f_s(u)&=\sum_{k=0}^{n}\frac{\delta^k}{k!}(-1)^kf_s^{(k)}(u+\delta)+\frac{(-1)^{n+1}}{n!}\int_{u}^{u+\delta}(w-u)^nf_s^{(n+1)}(w)\dd w\\
	&=\sum_{k=0}^{n}\frac{\delta^k}{k!}(-1)^kf_{s}^{(k)}(u+\delta)+\frac{(-1)^{n+1}}{n!}\int_{0}^{\delta}t^nf_s^{(n+1)}(u+t)\dd t.
\end{align*}
Consider the function
\begin{align}
	F_{s,n,\delta}(u)=\sum_{k=0}^{n}\frac{\delta^k}{k!}(-1)^kf_s^{(k)}(u+\delta)&\leq f_s(u)-\frac{(-1)^{n+1}}{n!}\int_{0}^{\delta}t^nf_s^{(n+1)}(u+t)\dd t\label{eq:F-n-delta}\\
    &=f_s(u)-\frac{c_{s,n+1}}{n!}\int_{0}^{\delta}t^nf_{s+2n+2}(u+t)\dd t,\notag
\end{align}
with equality for $u>0$. Since $(-1)^{n+1}f_s^{(n+1)}\geq 0$, we have $F_{s,n,\delta}(u)\leq f_s(u)$.
Given $P,Q\in\Grass{2}{4}$, let 
\begin{equation}\label{eq:tilde-f-s}
    \tilde{f}_s(\xi_{+}(P,Q),\xi_{-}(P,Q))\defined f_s(1-\xi_{+}(P,Q)\xi_{-}(P,Q)).
\end{equation}
Recall that the chordal distance in $\Grass{2}{4}$ in terms of the variables $\xi_{+}$ and $\xi_{-}$ is given in \cref{eq:chordal-distance-xi}. Then, the discrete Riesz energy of an $N$-point configuration $\omega_N=\set{P_1,\dotsc,P_N}\contained\Grass{2}{4}$ is precisely 
\begin{equation*}
    E_s(\omega_N)=E^{\tilde{f}_s}(\omega_N)=\sum_{i\neq j}f_s(1-\xi_{+}(P_i,P_j)\xi_{-}(P_i,P_j)).
\end{equation*}
To determine a lower bound for $E_s(\omega_N)$ we apply \cref{thm:LP-lower-bound} to the function $g\from[-1,1]^2\to\R$ given by
\begin{equation*}
	g(x,y)=F_{s,n,\delta}(1-xy).
\end{equation*}
From \cref{lemma:positive-coefficients-lower-bound}, it follows that the generalized Fourier--Jacobi coefficients $\hat{g}(\tau)$ of $g(x,y)$ are nonnegative for all $\tau\in\Partitions^*$.
Therefore, from \cref{thm:LP-lower-bound} we have, for any $N$-point configuration $\omega_N$,
\begin{equation}\label{eq:lower-bound-Es}
    \Energy_s(\omega_N)\geq \hat{g}(0,0)N^2-g(1,1)N,
\end{equation}
where the coefficient $\hat{g}(0,0)$ is given by \eqref{eq:generalizedJacobiCoefficients-0-0} and
\begin{equation}\label{eq:g-1-1}
	g(1,1)=F_{s,n,\delta}(0)=\sum_{k=0}^{n}\frac{\delta^k}{k!}(-1)^kf_s^{(k)}(\delta)=f_{s}(\delta)+\delta^{-s/2}\sum_{k=1}^{n}\frac{c_{s,k}}{k!}.
\end{equation} 
We now compute $\hat{g}(0,0)$. From \eqref{eq:F-n-delta} and \eqref{eq:cont-Riesz-energy-2} (and recall that $\Wiener_s(\Grass{2}{4})=\contEnergy_s[\sigma]$), we have
\begin{align}
	\hat{g}(0,0)&=\frac{1}{4}\int_{-1}^{1}\int_{-1}^{1}g(x,y)\dd x\dd y \label{eq:g-0-0-aux}\\
	&=\Wiener_s(\Grass{2}{4})-\frac{c_{s,n+1}}{4n!}\int_{0}^{\delta}t^n\int_{-1}^{1}\int_{-1}^{1}(1-xy+t)^{-s/2-n-1}\dd x\dd y \dd t.\notag
\end{align}
The inner integral is
\begin{align*}
	\MoveEqLeft\int_{-1}^{1}\int_{-1}^{1}(1-xy+t)^{-s/2-n-1}\dd x\dd y\\
	&=\frac{1}{s/2+n}\int_{-1}^{1}\frac{(1-y+t)^{-s/2-n}-(1+y+t)^{-s/2-n}}{y}\dd y.
\end{align*}
To compute this integral, note that
\begin{align*}
	\MoveEqLeft\frac{(1-y+t)^{-s/2-n}-(1+y+t)^{-s/2-n}}{y}\\
    &=-2\sum_{k=0}^{\infty}(1+t)^{-n-s/2-2k-1}\binom{-n-s/2}{2k+1}y^{2k}.
\end{align*}
Hence, integrating term-by-term we obtain
\begin{align*}
	\MoveEqLeft\frac{1}{s/2+n}\int_{-1}^{1}\frac{(1-y+t)^{-s/2-n}-(1+y+t)^{-s/2-n}}{y}\dd y\\
	&=\frac{-4}{s/2+n}\sum_{k=0}^{\infty}(1+t)^{-n-s/2-2k-1}\binom{-n-s/2}{2k+1}\frac{1}{2k+1}\\
	&=4(1+t)^{-n-s/2-1}\hypergeom{3}{2}\biggl(\frac{1}{2},\frac{n}{2}+\frac{s}{4}+\frac{1}{2},\frac{n}{2}+\frac{s}{4}+1;\frac{3}{2},\frac{3}{2};\frac{1}{(1+t)^2}\biggr).
\end{align*}
From \cref{prop:hypergeom32-limit-divergent} we have
\begin{align*}
   	\MoveEqLeft 4(1+t)^{-n-s/2-1}\hypergeom{3}{2}\biggl(\frac{1}{2},\frac{n}{2}+\frac{s}{4}+\frac{1}{2},\frac{n}{2}+\frac{s}{4}+1;\frac{3}{2},\frac{3}{2};\frac{1}{(1+t)^2}\biggr)\\
   	&\sim 2^{3-s/2-n}\frac{\Gamma(n+s/2-1)\Gamma(3/2)^2}{\Gamma(1/2)\Gamma(n/2+s/4+1/2)\Gamma(n/2+s/4+1)}t^{1-s/2-n}\\
   	&=\frac{2}{(n+s/2)(n+s/2-1)}t^{1-s/2-n}
\end{align*}
as $t\tendsto 0$, where in the last equality we have used the duplication formula for the gamma function. This estimate is valid for $n>1-s/2$. Calling
\begin{equation}\label{eq:def-An}
	A_{s,n}=\frac{2}{(n+s/2)(n+s/2-1)},
\end{equation}
we have, putting all the previous computations together,
\begin{equation}\label{eq:aux-low-bound}
	\int_{-1}^{1}\int_{-1}^{1}(1-xy+t)^{-s/2-n-1}\dd x\dd y\sim A_{s,n}t^{1-s/2-n},
\end{equation}
as $t\tendsto 0$. Substituting \eqref{eq:aux-low-bound} into \eqref{eq:g-0-0-aux}, we obtain
\begin{equation}\label{eq:g-0-0-final}
	\hat{g}(0,0)=\Wiener_s(\Grass{2}{4})-\frac{A_{s,n}c_{s,n+1}}{4n!(2-s/2)}\delta^{2-s/2}+o(\delta^{2-s/2}),
\end{equation}
as $\delta\tendsto 0^+$. These expressions are valid for $0\leq s<4$. Inserting \eqref{eq:g-0-0-final} and \eqref{eq:g-1-1} into \eqref{eq:lower-bound-Es}, we have
\begin{multline}\label{eq:Exp-delta-asymp}
   	\Energy_s(\omega_N)\geq N^2\biggl(\Wiener_s(\Grass{2}{4})-\frac{A_{s,n}c_{s,n+1}}{4n!(2-s/2)}\delta^{2-s/2}+o(\delta^{2-s/2})\biggr)\\-N\biggl(f_{s}(\delta)+\delta^{-s/2}\sum_{k=1}^{n}\frac{c_{s,k}}{k!}\biggr).
\end{multline}
To finish the proof, we take $\delta=N^{-1/2}$ in \cref{eq:Exp-delta-asymp}. In the case $0<s<4$, we obtain
\begin{equation*}
   	\Energy_s(\omega_N) \geq \Wiener_s(\Grass{2}{4})N^2+C_{s,n}N^{1+s/4}+o(N^{1+s/4}),
\end{equation*}
where
\begin{equation}\label{eq:C_sn}
    C_{s,n}=-\biggl(\frac{A_{s,n}c_{s,n+1}}{4n!(2-s/2)}+1+\sum_{k=1}^{n}\frac{c_{s,k}}{k!}\biggr).
\end{equation}
In the logarithmic case $s=0$, we have
\begin{equation*}
    \Energy_{\log}(\omega_N)\geq \Wiener_{\log}(\Grass{2}{4})N^2-\frac{1}{4}N\log{N}+C_{0,n}N+o(N),
\end{equation*}
where
\begin{equation}\label{eq:C_0n}
    C_{0,n}=-\biggl(\frac{A_{0,n}c_{0,n+1}}{8n!}+\sum_{k=1}^{n}\frac{c_{0,k}}{k!}\biggr).
\end{equation}
The constants $C_{s,n}$ and $C_{0,n}$ in \eqref{eq:C_sn} and \eqref{eq:C_0n} are negative and decreasing as a function of $n$. Hence, we choose $n$ as the smallest integer such that $n>1-s/2$. This implies choosing $n=1$ if $s>0$ and $n=2$ if $s=0$. The proof is now complete.\pushQED{\qed}\qedhere

\subsubsection{Proof of the lower bound for $s=4$}

In the hypersingular case $s=4$ we can use a simpler approach. Let 
$\psi(\xi_{+},\xi_{-})\defined (1-\xi_{+}\xi_{-}+\delta)^{-2}$. From \cref{lemma:positive-coefficients-lower-bound}, the generalized Fourier--Jacobi coefficients $\hat{\psi}(\tau)$ are nonnegative for all $\tau\in\Partitions^*$. Since $\psi(1,1)=\delta^{-2}$, from \cref{thm:LP-lower-bound} we have
\begin{equation}\label{eq:aux-lower-bound-hypersing}
   	\Energy_{4}(\omega_{N})=E^{\tilde{f}_4}(\omega_N)\geq \Energy^{\psi}(\omega_{N})\geq \hat{\psi}(0,0)N^2-\delta^{-2}N,
\end{equation}
where $\tilde{f}_s$ is as in \eqref{eq:tilde-f-s} and
\begin{align*}
   	\hat{\psi}(0,0)&=\frac{1}{4}\int_{-1}^{1}\int_{-1}^{1}\psi(x,y)\dd x\dd y=\frac{1}{4}\int_{-1}^{1}\int_{-1}^{1}(1-xy+\delta)^{-2}\dd x\dd y.
\end{align*}
Note that
\begin{equation*}
   	\int_{-1}^{1}\int_{-1}^{1}\frac{1}{(1-xy+\delta)^2}\dd x\dd y=\frac{2}{1+\delta}\log\biggl(1+\frac{2}{\delta}\biggr)
   	=-2\log\delta+2\log 2+o(1),
\end{equation*}
as $\delta\tendsto 0^{+}$. Hence,
\begin{equation}\label{eq:psi-0-0}
   	\hat{\psi}(0,0)=-\frac{1}{2}\log\delta+\frac{1}{2}\log 2 +o(1).
\end{equation}
Substituting \eqref{eq:psi-0-0} into \eqref{eq:aux-lower-bound-hypersing} and letting $\delta=N^{-1/2}$, we obtain
\begin{align*}
   	\Energy_{4}(\omega_{N})&\geq \frac{1}{4}N^2\log N +N^2\biggl(\frac{\log 2}{2}-1\biggr)+o(N^2).\pushQED{\qed}\qedhere
\end{align*}

\section{Upper bounds using jittered sampling}\label{sec:jittered}

In this section we use jittered sampling to prove the order of the upper bounds in \cref{thm:bounds-Riesz,thm:bounds-log}. This technique has already been used in other works; see, for example, \cite[Theorem 6.4.2]{BorodachovHardinSaff2019}. Although we do it for $\Grass{2}{4}$, the proof could be trivially extended to the general case of $\Grass{m}{d}$.

The following result is proved in \cite[Theorem 2]{GiganteLeopardi2017} for the more general setting of connected Ahlfors regular metric measure spaces. Here we state it for the particular case of the Grassmannian $\Grass{2}{4}$. 

\begin{lemma}\label{prop:partition}
    For the Grassmannian $\Grass{2}{4}$, there exist positive constants $c_1$ and $c_2$ such that for every sufficiently large $N$, there is a partition of $\Grass{2}{4}$ into $N$ regions of measure $1/N$, each contained in a geodesic ball of radius $c_1N^{-1/4}$ and containing a geodesic ball of radius $c_2N^{-1/4}$.
\end{lemma}

\begin{proposition}\label{prop:jitt}
    For the Grassmannian $\Grass{2}{4}$ and for $0<s<4$ there exists a positive constant $c_s$ such that, for $N$ sufficiently large,
    \begin{equation*}
        \minEnergy_s(N)\leq \Wiener_s(\Grass{2}{4}) N^2-c_sN^{1+s/4}.
    \end{equation*}
    For $s=\log$, there exists a constant $c_{\log}$ such that, for $N$ sufficiently large, 
    \begin{equation*}
        \minEnergy_{\log}(N)\leq \Wiener_{\log}(\Grass{2}{4}) N^2-\frac{1}{4}N\log{N}+c_{\log}N.
    \end{equation*}
\end{proposition}

\begin{proof}
    Since $\sin\theta<\theta<2\sin\theta$ for sufficiently small $\theta$, we have
    \begin{equation*}
        \sin^2\theta_1+\sin^2\theta_2<\theta_1^2+\theta_2^2<4(\sin^2\theta_1+\sin^2\theta_2).
    \end{equation*}
    Hence, recalling the expression \eqref{eq:chordal-distance-angles} for the chordal distance on $\Grass{2}{4}$, from \cref{prop:partition} there exists a positive constant $c_3$ such that for $N$ sufficiently large there is a partition of $\Grass{2}{4}$ into $N$ regions $D_1,\dotsc,D_N$, each of measure $1/N$ and contained in a chordal ball of radius $c_3N^{-1/4}$. Let $\dd\sigma_j(P)\defined N\chi_{D_j}(P)\dd\sigma(P)$, where $\chi_{D_j}$ is the indicator function of $D_j$. Then, we have, for $0<s<4$,
    \begin{align*}
        \minEnergy_s(N)&\leq \int_{\Grass{2}{4}}\dotsi\int_{\Grass{2}{4}}\sum_{i\neq j} K_s(P_i,P_j)\dd\sigma_1(P_1)\dotsb\dd\sigma_N(P_N)\\
        &=N^2\sum_{i\neq j}\iint_{D_i\times D_j}K_s(P_i,P_j)\dd\sigma(P_i)\dd\sigma(P_j)\\
        &=N^2\biggl(\iint\limits_{\mathclap{\Grass{2}{4}\times\Grass{2}{4}}} K_s(P,Q)\dd\sigma(P)\dd\sigma(Q)-\sum_{j=1}^{N}\iint_{D_j\times D_j}K_s(P,Q)\dd\sigma(P)\dd\sigma(Q)\biggr)\\
        &\leq \Wiener_s(\Grass{2}{4})N^2-\sum_{j=1}^{N}\frac{1}{\diam(D_j)^s}\leq \Wiener_s(\Grass{2}{4})N^2-(2c_3)^{-s}N^{1+s/4}.
    \end{align*}
    The logarithmic case is analogous.
\end{proof}

\section{Computation of the kernel of the harmonic ensemble}\label{sec:reproducing-kernel-harmonic}

We devote this section to the proof of \cref{thm:reproducing-kernel-grassmannian-dpp}. Recall from \eqref{eq:K-tau-2} that the reproducing kernels $\calK_{\tau}$ of the subspaces $H_{\tau}$ of $\LL{2}(\Grass{2}{4})$ are in one-to-one correspondence with the generalized Jacobi polynomials $\JacobiGrass{\tau}$. To compute the sum of these reproducing kernels we use a modified version of the multivariate Christoffel--Darboux formula.

\subsection{A Christoffel--Darboux type formula}

In this subsection, we present a result that is essential for the proof of \cref{thm:reproducing-kernel-grassmannian-dpp}. Let $\kappa=(\kappa_1,\kappa_2)$ be a partition of degree $\DegreePart{\kappa}=\kappa_1+\kappa_2=k$. We define $\kappa^{(i)}$ as the partition obtained from $\kappa$ by increasing the $i$th part $\kappa_i$ by one. Note that this is not always possible for all $i$, since $\kappa^{(i)}$ should also be a partition, that is, the new parts should be in decreasing order. 
The following result is an immediate consequence of \cite[Theorem 2.6]{Bachoc2006}. In \cite{Bachoc2006}, the corresponding result is given in terms of the variables $y_i=\cos^2\theta_i$. We present it instead in terms of the variables $\xi_{+}$ and $\xi_{-}$ introduced in \eqref{eq:xi_+-}.

\begin{theorem}\label{thm:Christoffel-Darboux-Grassmannian}
	Given a partition $\mu$ of degree $k+1$, let $A_k[\kappa,\mu]$ be zero whenever $\mu\neq \kappa^{(i)}$ for $i=1,2$. Otherwise, let
	\begin{align*}
		A_k[\kappa,\kappa^{(1)}]&=\begin{cases}
			\frac{(k+1)^2}{(2k+1)^2}& \kappa_2=0,\\
			\frac{k+1}{2k+1}\frac{\kappa_1-\kappa_2+1}{2\kappa_1-2\kappa_2+1}& \kappa_2>0,
		\end{cases}\\
		A_k[\kappa,\kappa^{(2)}]&=
		\begin{cases}
			\frac{2k(k+1)}{(2k+1)^2}& \kappa_2=0,\\
			\frac{k+1}{2k+1}\frac{\kappa_1-\kappa_2}{2\kappa_1-2\kappa_2+1}& \kappa_2>0.
		\end{cases}
	\end{align*}
	Let $Q_{\kappa}\defined \sum_{\DegreePart{\mu}=k+1}A_k[\kappa,\mu]\JacobiGrass{\mu}$. Then, for all $k\geq 0$ we have	
	\begin{equation*}
		\sum_{\DegreePart{\tau}\leq k} d_{\tau}\JacobiGrass{\tau}(\xi_{+},\xi_{-})=\frac{\sum_{\DegreePart{\kappa}=k} d_{\kappa}(Q_{\kappa}(1,1)P_{\kappa}(\xi_{+},\xi_{-})-Q_{\kappa}(\xi_{+},\xi_{-}))}{1-\xi_{+}\xi_{-}},
	\end{equation*}
    where $d_{\tau}$ is given by \eqref{eq:d-tau}.
\end{theorem}

Note that the sum in the left-hand side of \cref{thm:Christoffel-Darboux-Grassmannian} is precisely the reproducing kernel $\calK_k$. Therefore, to prove \cref{thm:reproducing-kernel-grassmannian-dpp} it suffices to compute the sum in the right-hand side.

\subsection{Some preliminary lemmas}

In this subsection, we gather some auxiliary lemmas that will be used to prove \cref{thm:reproducing-kernel-grassmannian-dpp}. Recall from \cref{eq:JacobiGrass-xi} that, given a partition $\kappa=(\kappa_1,\kappa_2)$ of degree $\DegreePart{\kappa}=\kappa_1+\kappa_2=k$, the generalized Jacobi polynomial $\JacobiGrass{\kappa}$ can be expressed in terms of Legendre polynomials as
\begin{equation*}
    \JacobiGrass{\kappa}(\xi_{+},\xi_{-})=\frac{1}{2}\Bigl(\Legendre{k}(\xi_{+})\Legendre{\kappa_1-\kappa_2}(\xi_{-})+\Legendre{k}(\xi_{-})\Legendre{\kappa_1-\kappa_2}(\xi_{+})\Bigr).
\end{equation*}

\begin{lemma}\label{lemma:Q_k(1)}
	Let $k>0$ and let $\kappa=(\kappa_1,\kappa_2)$ be a partition with $\DegreePart{\kappa}=k$. 
    Then, with the notation of \cref{thm:Christoffel-Darboux-Grassmannian},
	\begin{equation*}
		Q_{\kappa}(1,1)=
		\begin{cases}
			\frac{(k+1)(3k+1)}{(2k+1)^2} & \kappa_2=0,\\
			\frac{k+1}{2k+1} & \kappa_2>0.
		\end{cases}
	\end{equation*}
\end{lemma}

\begin{proof}
	Since, by definition, $\JacobiGrass{\mu}(1,1)=1$ for every partition $\mu$, we have $Q_{\kappa}(1,1)=A_k[\kappa,\kappa^{(1)}]+A_k[\kappa,\kappa^{(2)}]$. If $\kappa_2>0$, then
	\begin{align*}
		Q_{\kappa}(1,1)&=A_k[\kappa,\kappa^{(1)}]+A_k[\kappa,\kappa^{(2)}]\\
        &=\frac{k+1}{2k+1}\frac{\kappa_1-\kappa_2+1}{2\kappa_1-2\kappa_2+1}+\frac{\kappa_1-\kappa_2}{2\kappa_1-2\kappa_2+1}\frac{k+1}{2k+1}=\frac{k+1}{2k+1}.
	\end{align*}
	If $\kappa_2=0$, then we obtain
	\begin{align*}
		Q_{\kappa}(1,1)&=A_k[\kappa,\kappa^{(1)}]+A_k[\kappa,\kappa^{(2)}]=
		\frac{(k+1)^2}{(2k+1)^2}+\frac{2k(k+1)}{(2k+1)^2}=\frac{(k+1)(3k+1)}{(2k+1)^2}.\qedhere
	\end{align*}
\end{proof}

\begin{lemma}\label{lemma:Q_k(y)}
	Let $k>0$ and let $\kappa=(\kappa_1,\kappa_2)$ be a partition with $\DegreePart{\kappa}=k$. Then, with the notation of \cref{thm:Christoffel-Darboux-Grassmannian},
	\begin{align*}
     Q_{\kappa}(\xi_{+},\xi_{-})=
		\begin{cases}
			\begin{aligned}
			    \frac{k+1}{2k+1}\Bigl(&\xi_{-}\Legendre{k}(\xi_{-})\Legendre{k+1}(\xi_{+})+\xi_{+}\Legendre{k}(\xi_{+})\Legendre{k+1}(\xi_{-})\Bigr)\\
                &-\frac{(k+1)^2}{(2k+1)^2}\Legendre{k+1}(\xi_{+})\Legendre{k+1}(\xi_{-}),
			\end{aligned}
			& \kappa_2=0,\\[24pt]
			\begin{aligned}
			    \frac{k+1}{2(2k+1)}\Bigl(&\xi_{-}\Legendre{\kappa_1-\kappa_2}(\xi_{-})\Legendre{k+1}(\xi_{+})\\
                &+\xi_{+}\Legendre{\kappa_1-\kappa_2}(\xi_{+})\Legendre{k+1}(\xi_{-})\Bigr),	
			\end{aligned}& \kappa_2>0.
		\end{cases}
	\end{align*}
\end{lemma}

\begin{proof}
	Note that
	\begin{equation*}
		Q_{\kappa}(\xi_{+},\xi_{-})=A_k[\kappa,\kappa^{(1)}]\JacobiGrass{\kappa^{(1)}}(\xi_{+},\xi_{-})+A_k[\kappa,\kappa^{(2)}]\JacobiGrass{\kappa^{(2)}}(\xi_{+},\xi_{-}),
	\end{equation*}
	where $\kappa^{(1)}=(\kappa_1+1,\kappa_2)$, $\kappa^{(2)}=(\kappa_1,\kappa_2+1)$, and
	\begin{align*}
		\JacobiGrass{\kappa^{(1)}}(\xi_{+},\xi_{-})&=\frac{1}{2}\Bigl(\Legendre{k+1}(\xi_{+})\Legendre{\kappa_1-\kappa_2+1}(\xi_{-})+\Legendre{k+1}(\xi_{-})\Legendre{\kappa_1-\kappa_2+1}(\xi_{+})\Bigr),\\
		\JacobiGrass{\kappa^{(2)}}(\xi_{+},\xi_{-})&=\frac{1}{2}\Bigl(\Legendre{k+1}(\xi_{+})\Legendre{\kappa_1-\kappa_2-1}(\xi_{-})+\Legendre{k+1}(\xi_{-})\Legendre{\kappa_1-\kappa_2-1}(\xi_{+})\Bigr).
	\end{align*}
	When $\kappa_2>0$, we have
	\begin{align}
		&Q_{\kappa}(\xi_{+},\xi_{-})=\frac{(k+1)}{2(2k+1)}\label{eq:aux-Qkappa}\\
		&\qquad\times\biggl(\frac{\kappa_1-\kappa_2+1}{2\kappa_1-2\kappa_2+1}\bigl(\Legendre{k+1}(\xi_{+})\Legendre{\kappa_1-\kappa_2+1}(\xi_{-})+\Legendre{k+1}(\xi_{-})\Legendre{\kappa_1-\kappa_2+1}(\xi_{+})\bigr)\notag\\
		&\qquad+ \frac{\kappa_2-\kappa_1}{2\kappa_2-2\kappa_1-1}\bigl(\Legendre{k+1}(\xi_{+})\Legendre{\kappa_1-\kappa_2-1}(\xi_{-})+\Legendre{k+1}(\xi_{-})\Legendre{\kappa_1-\kappa_2-1}(\xi_{+})\bigr)\biggr).\notag
	\end{align}
	From the recurrence relation \eqref{eq:recurrence-Legendre} for the Legendre polynomials, we have
	\begin{align}
		\Legendre{\kappa_1-\kappa_2+1}(\xi_{+})&=\frac{2\kappa_1-2\kappa_2+1}{\kappa_1-\kappa_2+1}\xi_{+}	\Legendre{\kappa_1-\kappa_2}(\xi_{+})-\frac{\kappa_1-\kappa_2}{\kappa_1-\kappa_2+1}\Legendre{\kappa_1-\kappa_2-1}(\xi_{+}),\label{eq:aux-Legendre-1}\\
		\Legendre{\kappa_1-\kappa_2+1}(\xi_{-})&=\frac{2\kappa_1-2\kappa_2+1}{\kappa_1-\kappa_2+1}\xi_{-}	\Legendre{\kappa_1-\kappa_2}(\xi_{-})-\frac{\kappa_1-\kappa_2}{\kappa_1-\kappa_2+1}\Legendre{\kappa_1-\kappa_2-1}(\xi_{-}).\label{eq:aux-Legendre-2}
	\end{align}
	Substituting \eqref{eq:aux-Legendre-1} and \eqref{eq:aux-Legendre-2} into \eqref{eq:aux-Qkappa}, we readily obtain the desired expression for $Q_{\kappa}(\xi_{+},\xi_{-})$ when $\kappa_2>0$. The case $\kappa_2=0$ follows similarly.
\end{proof}

\begin{lemma}\label{lemma:S1}
	With the notation of \cref{thm:Christoffel-Darboux-Grassmannian}, let $k>0$ and
	\begin{equation*}
		S_1=d_{(k,0)}(Q_{(k,0)}(1,1)P_{(k,0)}(\xi_{+},\xi_{-})-Q_{(k,0)}(\xi_{+},\xi_{-})).
	\end{equation*}
	Then,
	\begin{align*}
		S_1&=(k+1)(3k+1)\Legendre{k}(\xi_{+})\Legendre{k}(\xi_{-})\\
		&\qquad-(2k+1)(k+1)\Bigl(\xi_{-}\Legendre{k}(\xi_{-})\Legendre{k+1}(\xi_{+})+\xi_{+}\Legendre{k}(\xi_{+})\Legendre{k+1}(\xi_{-})\Bigr)\\
		&\qquad+(k+1)^2\Legendre{k+1}(\xi_{+})\Legendre{k+1}(\xi_{-}).
	\end{align*}
\end{lemma}

\begin{proof}
	Immediate from \cref{lemma:Q_k(1),lemma:Q_k(y)}, and the fact that $d_{(k,0)}=(2k+1)^2$.
\end{proof}

\begin{lemma}\label{lemma:S2}
	With the notation of \cref{thm:Christoffel-Darboux-Grassmannian}, let $k>0$ and
	\begin{equation*}
		S_2=\sum_{\DegreePart{\kappa}=k, \kappa_2>0} d_{\kappa}(Q_{\kappa}(1,1)P_{\kappa}(\xi_{+},\xi_{-})-Q_{\kappa}(\xi_{+},\xi_{-})).
	\end{equation*}
	Then,
	\begin{align*}
		S_2&=(\Legendre{k}(\xi_{+})-\xi_{-}\Legendre{k+1}(\xi_{+}))(k+1)\Bigl(-(2k+1)\Legendre{k}(\xi_{-})+\Gegenbauer{k}{3/2}(\xi_{-})\Bigr)\\
		&\qquad +(\Legendre{k}(\xi_{-})-\xi_{+}\Legendre{k+1}(\xi_{-}))(k+1)\Bigl(-(2k+1)\Legendre{k}(\xi_{+})+\Gegenbauer{k}{3/2}(\xi_{+})\Bigr).
	\end{align*}
\end{lemma}

\begin{proof}
	From \cref{lemma:Q_k(1),lemma:Q_k(y)} we have
	\begin{align*}
		 \MoveEqLeft Q_{\kappa}(1,1)\JacobiGrass{\kappa}(\xi_{+},\xi_{-})-Q_{\kappa}(\xi_{+},\xi_{-})\\
		&=\frac{k+1}{2(2k+1)}\Bigl(\Legendre{k}(\xi_+)\Legendre{\kappa_1-\kappa_2}(\xi_{-})+\Legendre{k}(\xi_-)\Legendre{\kappa_1-\kappa_2}(\xi_{+})\\
		&\qquad\qquad\qquad\qquad-\xi_{-}\Legendre{\kappa_1-\kappa_2}(\xi_{-})\Legendre{k+1}(\xi_{+})-\xi_{+}\Legendre{\kappa_1-\kappa_2}(\xi_{+})\Legendre{k+1}(\xi_{-})\Bigr).
	\end{align*}
	In this case, $d_{\kappa}=2(2k+1)(2(\kappa_1-\kappa_2)+1)$, so
	\begin{align*}
		\MoveEqLeft d_{\kappa}(Q_{\kappa}(1,1)P_{\kappa}(\xi_{+},\xi_{-})-Q_{\kappa}(\xi_{+},\xi_{-}))\\
		&=(2(\kappa_1-\kappa_2)+1)(k+1)\Bigl(\Legendre{k}(\xi_+)\Legendre{\kappa_1-\kappa_2}(\xi_{-})+\Legendre{k}(\xi_-)\Legendre{\kappa_1-\kappa_2}(\xi_{+})\\
		&\qquad-\xi_{-}\Legendre{\kappa_1-\kappa_2}(\xi_{-})\Legendre{k+1}(\xi_{+})-\xi_{+}\Legendre{\kappa_1-\kappa_2}(\xi_{+})\Legendre{k+1}(\xi_{-})\Bigr)\\
		&=\bigl(\Legendre{k}(\xi_{+})-\xi_{-}\Legendre{k+1}(\xi_{+})\bigr)(k+1)\bigl(2(\kappa_1-\kappa_2)+1\bigr)\Legendre{\kappa_1-\kappa_2}(\xi_{-})\\
		&\qquad+\bigl(\Legendre{k}(\xi_{-})-\xi_{+}\Legendre{k+1}(\xi_{-})\bigr)(k+1)\bigl(2(\kappa_1-\kappa_2)+1\bigr)\Legendre{\kappa_1-\kappa_2}(\xi_{+}).
	\end{align*}
	Hence,
	\begin{align*}
		S_2&=(\Legendre{k}(\xi_{+})-\xi_{-}\Legendre{k+1}(\xi_{+}))(k+1)\sum_{\substack{\kappa_1+\kappa_2=k\\ \kappa_1\geq\kappa_2>0}} (2(\kappa_1-\kappa_2)+1)\Legendre{\kappa_1-\kappa_2}(\xi_{-})\\
		&\qquad +(\Legendre{k}(\xi_{-})-\xi_{+}\Legendre{k+1}(\xi_{-}))(k+1)\sum_{\substack{\kappa_1+\kappa_2=k\\ \kappa_1\geq\kappa_2>0}}(2(\kappa_1-\kappa_2)+1)\Legendre{\kappa_1-\kappa_2}(\xi_{+}).
	\end{align*}
	To compute these sums, it is convenient to distinguish between $k$ being even and $k$ being odd, although the final result will not depend on the parity of $k$. Assume, then, that $k$ is even, and so $\kappa_1-\kappa_2=0,2,\dotsc,k-2$. Calling $\kappa_1-\kappa_2=2m$, we have
	\begin{align}\label{eq:aux-S2}
		S_2&=(\Legendre{k}(\xi_{+})-\xi_{-}\Legendre{k+1}(\xi_{+}))(k+1)\sum_{m=0}^{k/2-1} (4m+1)\Legendre{2m}(\xi_{-})\\
		&\qquad +(\Legendre{k}(\xi_{-})-\xi_{+}\Legendre{k+1}(\xi_{-}))(k+1)\sum_{m=1}^{k/2-1}(4m+1)\Legendre{2m}(\xi_{+}).\notag
	\end{align}
	Note that
	\begin{equation*}
		\sum_{m=0}^{k/2-1} (4m+1)\Legendre{2m}(\xi_{-})=-(2k+1)\Legendre{k}(\xi_{-})+\sum_{m=0}^{k/2} (4m+1)\Legendre{2m}(\xi_{-}).
	\end{equation*}
	From \cref{eq:sum-legendre-even-odd}, we have
	\begin{equation*}
		\sum_{m=0}^{k/2} (4m+1)\Legendre{2m}(\xi_{-})=\Gegenbauer{k}{3/2}(\xi_{-}).
	\end{equation*}
	Therefore,
	\begin{equation}\label{eq:aux-sum-Legendre-S2-minus}
		\sum_{m=0}^{k/2-1} (4m+1)\Legendre{2m}(\xi_{-})=-(2k+1)\Legendre{k}(\xi_{-})+\Gegenbauer{k}{3/2}(\xi_{-}).
	\end{equation}
	Similarly,
	\begin{equation}\label{eq:aux-sum-Legendre-S2-plus}
		\sum_{m=0}^{k/2-1} (4m+1)\Legendre{2m}(\xi_{+})=-(2k+1)\Legendre{k}(\xi_{+})+\Gegenbauer{k}{3/2}(\xi_{+}).
	\end{equation}
	Then, substituting \eqref{eq:aux-sum-Legendre-S2-minus} and \eqref{eq:aux-sum-Legendre-S2-plus} into \eqref{eq:aux-S2}, we obtain
	\begin{align*}
	 S_2&=(\Legendre{k}(\xi_{+})-\xi_{-}\Legendre{k+1}(\xi_{+}))(k+1)\Bigl(-(2k+1)\Legendre{k}(\xi_{-})+\Gegenbauer{k}{3/2}(\xi_{-})\Bigr)\\
	 &\qquad +(\Legendre{k}(\xi_{-})-\xi_{+}\Legendre{k+1}(\xi_{-}))(k+1)\Bigl(-(2k+1)\Legendre{k}(\xi_{+})+\Gegenbauer{k}{3/2}(\xi_{+})\Bigr).
	\end{align*}
    The case where $k$ is odd is handled analogously, yielding the same expression for $S_2$ as in the even case. Hence, the lemma follows.
\end{proof}

We are finally ready to prove \cref{thm:reproducing-kernel-grassmannian-dpp}.

\subsection{Proof of \cref{thm:reproducing-kernel-grassmannian-dpp}}

From \cref{thm:Christoffel-Darboux-Grassmannian} we know that
	\begin{equation*}
		\calK_k(\xi_{+},\xi_{-})=\sum_{\DegreePart{\tau}\leq k}d_{\tau}\JacobiGrass{\tau}(\xi_{+},\xi_{-})=\frac{\sum_{\DegreePart{\kappa}=k} d_{\kappa}(Q_{\kappa}(1,1)\JacobiGrass{\kappa}(\xi_{+},\xi_{-})-Q_{\kappa}(\xi_{+},\xi_{-}))}{1-\xi_{+}\xi_{-}}.
	\end{equation*}
	We call
	\begin{equation*}
		N_k(\xi_{+},\xi_{-})=\sum_{\DegreePart{\kappa}=k} d_{\kappa}(Q_{\kappa}(1,1)\JacobiGrass{\kappa}(\xi_{+},\xi_{-})-Q_{\kappa}(\xi_{+},\xi_{-})).
	\end{equation*}
	This sum can be decomposed as
	\begin{equation*}
		N_k(\xi_{+},\xi_{-})=S_1+S_2,
	\end{equation*}
	where
	\begin{align*}
		S_1&=d_{(k,0)}\bigl((Q_{(k,0)}(1,1)P_{(k,0)}(\xi_{+},\xi_{-})-Q_{(k,0)}(\xi_{+},\xi_{-})\bigr),\\
		S_2&=\sum_{\DegreePart{\kappa}=k, \kappa_2>0} d_{\kappa}\bigl(Q_{\kappa}(1,1)P_{\kappa}(\xi_{+},\xi_{-})-Q_{\kappa}(\xi_{+},\xi_{-})\bigr).
	\end{align*}
	Using \cref{lemma:S1,lemma:S2}, we have
	\begin{align*}
		N_k(\xi_{+},\xi_{-})&=(k+1)(3k+1)\Legendre{k}(\xi_{+})\Legendre{k}(\xi_{-})\\
		&\quad-(2k+1)(k+1)\Bigl(\xi_{-}\Legendre{k}(\xi_{-})\Legendre{k+1}(\xi_{+})+\xi_{+}\Legendre{k}(\xi_{+})\Legendre{k+1}(\xi_{-})\Bigr)\\
		&\quad+(k+1)^2\Legendre{k+1}(\xi_{+})\Legendre{k+1}(\xi_{-})\\
		&\quad+(\Legendre{k}(\xi_{+})-\xi_{-}\Legendre{k+1}(\xi_{+}))(k+1)\Bigl(-(2k+1)\Legendre{k}(\xi_{-})+\Gegenbauer{k}{3/2}(\xi_{-})\Bigr)\\
		&\quad +(\Legendre{k}(\xi_{-})-\xi_{+}\Legendre{k+1}(\xi_{-}))(k+1)\Bigl(-(2k+1)\Legendre{k}(\xi_{+})+\Gegenbauer{k}{3/2}(\xi_{+})\Bigr).
	\end{align*}
	This expression can be simplified to
	\begin{align}
		\frac{N_k(\xi_{+},\xi_{-})}{k+1}
		&=(k+1)\Bigl(\Legendre{k+1}(\xi_{+})\Legendre{k+1}(\xi_{-})-\Legendre{k}(\xi_{+})\Legendre{k}(\xi_{-})\Bigr)\label{eq:Nk-1}\\
		&\qquad+(\Legendre{k}(\xi_{+})-\xi_{-}\Legendre{k+1}(\xi_{+}))\Gegenbauer{k}{3/2}(\xi_{-})\notag\\
		&\qquad+(\Legendre{k}(\xi_{-})-\xi_{+}\Legendre{k+1}(\xi_{-}))\Gegenbauer{k}{3/2}(\xi_{+}).\notag
	\end{align}
	From \cref{eq:legendre-gegenbauer-n,eq:legendre-gegenbauer-n+1} we know that
	\begin{align}
		\Legendre{k}(x)&=\frac{1}{k+1}(\Gegenbauer{k}{3/2}(x)-x\Gegenbauer{k-1}{3/2}(x)),\label{eq:aux-recurrence-Legendre-1}\\
		\Legendre{k+1}(x)&=\frac{1}{k+1}(x\Gegenbauer{k}{3/2}(x)-\Gegenbauer{k-1}{3/2}(x)).\label{eq:aux-recurrence-Legendre-2}
	\end{align}
	Substituting \eqref{eq:aux-recurrence-Legendre-1} and \eqref{eq:aux-recurrence-Legendre-2} into \eqref{eq:Nk-1} and simplifying we obtain
	\begin{equation*}
		 \frac{N_k(\xi_{+},\xi_{-})}{k+1}=\frac{1-\xi_{+}\xi_{-}}{k+1}\Bigl(\Gegenbauer{k}{3/2}(\xi_{+})\Gegenbauer{k}{3/2}(\xi_{-})+\Gegenbauer{k-1}{3/2}(\xi_{+})\Gegenbauer{k-1}{3/2}(\xi_{-})\Bigr).
	\end{equation*}
	Therefore,
	\begin{equation*}
		N_k(\xi_{+},\xi_{-})=(1-\xi_{+}\xi_{-})\Bigl(\Gegenbauer{k}{3/2}(\xi_{+})\Gegenbauer{k}{3/2}(\xi_{-})+\Gegenbauer{k-1}{3/2}(\xi_{+})\Gegenbauer{k-1}{3/2}(\xi_{-})\Bigr),
	\end{equation*}
	and hence
	\begin{equation*}
		\calK_k(\xi_{+},\xi_{-})=\frac{N_k(\xi_{+},\xi_{-})}{1-\xi_{+}\xi_{-}}=\Gegenbauer{k}{3/2}(\xi_{+})\Gegenbauer{k}{3/2}(\xi_{-})+\Gegenbauer{k-1}{3/2}(\xi_{+})\Gegenbauer{k-1}{3/2}(\xi_{-}). \pushQED{\qed}\qedhere
	\end{equation*}

\section{Expected energy of the harmonic ensemble}\label{sec:energy-harmonic-ensemble}

To determine the expected energy of points coming from the harmonic ensemble with the desired precision we need to perform a thorough asymptotic analysis of certain double integrals involving orthogonal polynomials. For the reader's convenience, we have encapsulated the necessary technical results as lemmas and deferred the proofs to \cref{sec:auxiliary-results}.

\begin{proposition}\label{prop:energy-dpp-finite}
	The expected Riesz $s$-energy of the harmonic ensemble is finite if and only if $s<6$.
\end{proposition}

\begin{proof}
	From \cref{prop:expectation-dpp} and \cref{lemma:integral-Grass24-2}, the expected Riesz $s$-energy of this \gls{dpp} is
	\begin{align*}    
		\MoveEqLeft\expectation{x\distributed\HarmonicDPP{N}}{\Energy_s(x)}=\iint\limits_{\mathclap{\Grass{2}{4}\times \Grass{2}{4}}} \frac{\calK_{k}(P,P)\calK_{k}(Q,Q)-\calK_{k}(P,Q)^2}{\chordalGrass(P,Q)^s}\dd \sigma(P)\dd \sigma(Q)\\
		&=\frac{1}{2}\int_{0}^{1}\int_{-1}^{1}\frac{\calK_k(1,1)^2-\calK_k(\xi_{+},\xi_{-})^2}{(1-\xi_{+}\xi_{-})^{s/2}}\dd\xi_{+}\dd\xi_{-}.
	\end{align*}
	Making the substitution $(\xi_{+},\xi_{-})\mapsto (\cos\theta,\cos\sigma)$, we can rewrite the previous integral as 
	\begin{equation*}    
		\frac{1}{2}\int_{0}^{\pi/2}\int_{0}^{\pi}\frac{(\calK_k(1,1)^2-\calK_k(\cos\theta,\cos\sigma)^2)\sin\theta\sin\sigma}{(1-\cos\theta\cos\sigma)^{s/2}}\dd\theta\dd\sigma.
	\end{equation*}
	Since the integrand can only be unbounded when $(\theta,\sigma)$ approaches $(0,0)$, it suffices to prove that the integral
	\begin{equation*}    
		\int_{0}^{\varepsilon}\int_{0}^{\varepsilon}\frac{(\calK_k(1,1)^2-\calK_k(\cos\theta,\cos\sigma)^2)\sin\theta\sin\sigma}{(1-\cos\theta\cos\sigma)^{s/2}}\dd\theta\dd\sigma
	\end{equation*}
	is finite if and only if $s<6$ for some sufficiently small $\varepsilon>0$. From \cref{lemma:K_k^2} we know that there exists a constant $c(k)$ depending only on $k$ such that
	\begin{equation*}
		\calK_k(1,1)^2-\calK_k(\cos\theta,\cos\sigma)^2=c(k)(\theta^2+\sigma^2)+O(\theta^4+\sigma^4).
	\end{equation*}
	This, together with the estimates $\sin \theta\sin\sigma \sim \theta\sigma$ and $(1-\cos \theta\cos\sigma)\sim (\theta^2+\sigma^2)/2$, gives
	\begin{align*}    
	\MoveEqLeft\int_{0}^{\varepsilon}\int_{0}^{\varepsilon}\frac{(\calK_k(1,1)^2-\calK_k(\cos\theta,\cos\sigma)^2)\sin\theta\sin\sigma}{(1-\cos\theta\cos\sigma)^{s/2}}\dd\theta\dd\sigma\asymp\int_{0}^{\varepsilon}\int_{0}^{\varepsilon}\frac{\theta\sigma(\theta^2+\sigma^2)}{(\theta^2+\sigma^2)^{s/2}}\dd\theta\dd\sigma.
	\end{align*}
	Since the integrand is positive, we have
	\begin{align*}
		\iint\limits_{\theta^2+\sigma^2<\varepsilon^2} \!\!\frac{\theta\sigma(\theta^2+\sigma^2)}{(\theta^2+\sigma^2)^{s/2}}\dd\theta\dd\sigma\leq \int_{0}^{\varepsilon}\int_{0}^{\varepsilon}\frac{\theta\sigma(\theta^2+\sigma^2)}{(\theta^2+\sigma^2)^{s/2}}\dd\theta\dd\sigma\leq \iint\limits_{\mathclap{\theta^2+\sigma^2<2\varepsilon^2}}\ \,\frac{\theta\sigma(\theta^2+\sigma^2)}{(\theta^2+\sigma^2)^{s/2}}\dd\theta\dd\sigma,
	\end{align*}
	where the integrals in the bounds are finite if and only if $s<6$.
\end{proof}

To prove \cref{thm:Energy-DPP-Riesz} we need the following lemma.

\begin{lemma}\label{lemma:integral-dpp}
	Let $k\geq 1$, $0<s<4$, and $\beta\in\set{0,1}$. Then,
	\begin{align*}
		\MoveEqLeft\int_{0}^{1}\int_{-1}^{1}\frac{\Gegenbauer{k}{3/2}(\xi_{+})\Gegenbauer{k-\beta}{3/2}(\xi_{+})\Gegenbauer{k}{3/2}(\xi_{-})\Gegenbauer{k-\beta}{3/2}(\xi_{-})}{(1-\xi_{+}\xi_{-})^{s/2}}\dd \xi_{+} \dd \xi_{-}\\
		&=k^{4+s}2^{s/2}\int_0^{\infty}\int_0^{\infty} \frac{\BesselJ{1}(x)^2\BesselJ{1}(y)^2}{xy(x^2+y^2)^{s/2}}\dd x\dd y +o(k^{4+s}),
	\end{align*}
	where $J_1$ is the Bessel function of the first kind of order $1$.
\end{lemma}

\begin{proof}
	See \cref{sec:proof-lemma-integral-dpp}.
\end{proof}

\subsection*{Proof of \cref{thm:Energy-DPP-Riesz}}

From \cref{prop:expectation-dpp}, \cref{remark:KH-xx}, \cref{lemma:integral-Grass24-2}, \cref{eq:cont-Riesz-energy-1,eq:chordal-distance-xi}, and the fact that $\calK_k(P,P)=N$, the expected Riesz energy of this \gls{dpp} is
\begin{align*}    
	\expectation{x\distributed\HarmonicDPP{N}}{\Energy_s(x)}&=\iint\limits_{\mathclap{\Grass{2}{4}\times \Grass{2}{4}}} \frac{\calK_{k}(P,P)\calK_{k}(Q,Q)-\calK_{k}(P,Q)^2}{\chordalGrass(P,Q)^s}\dd \sigma(P)\dd \sigma(Q)\\
	&=\Wiener_s(\Grass{2}{4})N^2-\frac{1}{2}\int_{0}^{1}\int_{-1}^{1}\frac{\calK_k(\xi_{+},\xi_{-})^2}{(1-\xi_{+}\xi_{-})^{s/2}}\dd \xi_{+}\dd \xi_{-}.
\end{align*}
From \cref{thm:reproducing-kernel-grassmannian-dpp} and \cref{lemma:integral-dpp}, we have
\begin{align*}
	\expectation{x\distributed\HarmonicDPP{N}}{\Energy_s(x)}
    &=\Wiener_s(\Grass{2}{4})N^2\\
    &\qquad-2^{1+s/2}k^{4+s}\int_0^{\infty}\int_0^{\infty} \frac{\BesselJ{1}(x)^2\BesselJ{1}(y)^2}{xy(x^2+y^2)^{s/2}}\dd x\dd y+o(k^{4+s}).
\end{align*}
Finally, using that $N=\frac{1}{2}k^4+O(k^3)$, the theorem follows. \qed

To prove \cref{thm:Energy-DPP-log} we need the following lemma.

\begin{lemma}\label{lemma:integral-dpp-log}
	Let $k\geq 1$ and let $\beta\in\set{0,1}$. The following equality holds:
	\begin{align*}
		\MoveEqLeft\int_{0}^{1}\int_{-1}^{1}\Gegenbauer{k}{3/2}(\xi_{+})\Gegenbauer{k-\beta}{3/2}(\xi_{+})\Gegenbauer{k}{3/2}(\xi_{-})\Gegenbauer{k-\beta}{3/2}(\xi_{-})\log(1-\xi_{+}\xi_{-})\dd \xi_{+}\dd \xi_{-}\\
		&=-\frac{k^4\log{k}}{2}+Bk^4+(-1)^{\beta}\frac{\log(2)}{4}\Bigl(1-\BesselJ{0}(1)^2-\BesselJ{1}(1)^2\Bigr)^2k^4+o(k^4),
	\end{align*}
	with
    \begin{equation*}
        B=\frac{1+2G}{2\pi}+\frac{1}{8}-\frac{\gamma}{2}+\frac{\log{2}}{4},
    \end{equation*}
	where $G$ is Catalan's constant and $\gamma$ is the Euler--Mascheroni constant.
\end{lemma}

\begin{proof}
	See \cref{sec:proof-lemma-integral-dpp-log}.
\end{proof}

\subsection*{Proof of \cref{thm:Energy-DPP-log}}

From \cref{prop:expectation-dpp}, \cref{remark:KH-xx}, \cref{lemma:integral-Grass24-2}, and \cref{eq:cont-log-energy,eq:chordal-distance-xi}, the expected logarithmic energy of this \gls{dpp} is
\begin{align*}    
	\expectation{x\distributed\HarmonicDPP{N}}{\Energy_{\log}(x)}&=-\iint\limits_{\mathclap{\Grass{2}{4}\times \Grass{2}{4}}} (\calK_{k}(P,P)\calK_{k}(Q,Q)-\calK_{k}(P,Q)^2)\log\chordalGrass(P,Q)\dd \sigma(P)\dd \sigma(Q)\\
	&=\Wiener_{\log}(\Grass{2}{4})N^2+\frac{1}{4}\int_{0}^{1}\int_{-1}^{1}\calK_k(\xi_{+},\xi_{-})^2\log(1-\xi_{+}\xi_{-})\dd \xi_{+}\dd \xi_{-}.
\end{align*}
From \cref{thm:reproducing-kernel-grassmannian-dpp} and \cref{lemma:integral-dpp-log}, we have
\begin{align*}    
	\MoveEqLeft\expectation{x\distributed\HarmonicDPP{N}}{\Energy_{\log}(x)}\\
    &=\Wiener_{\log}(\Grass{2}{4})N^2+\frac{1}{4}\biggl(-\frac{k^4\log{k}}{2}+Bk^4+\frac{\log(2)}{4}\Bigl(1-\BesselJ{0}(1)^2-\BesselJ{1}(1)^2\Bigr)^2k^4\biggr)\\
	&\qquad+\frac{1}{4}\biggl(-\frac{k^4\log{k}}{2}+Bk^4+\frac{\log(2)}{4}\Bigl(1-\BesselJ{0}(1)^2-\BesselJ{1}(1)^2\Bigr)^2k^4\biggr)\\
	&\qquad+\frac{1}{2}\biggl(-\frac{k^4\log{k}}{2}+Bk^4-\frac{\log(2)}{4}\Bigl(1-\BesselJ{0}(1)^2-\BesselJ{1}(1)^2\Bigr)^2k^4\biggr)+o(k^4),
\end{align*}
with
\begin{equation*}
	B=\frac{1+2G}{2\pi}+\frac{1}{8}-\frac{\gamma}{2}+\frac{\log{2}}{4},
\end{equation*}
Simplifying, we obtain
\begin{equation*}
	\expectation{x\distributed\HarmonicDPP{N}}{\Energy_{\log}(x)}=\Wiener_{\log}(\Grass{2}{4})N^2-\frac{1}{2}k^4\log k+Bk^4+o(k^4).
\end{equation*}
Using that $N=\frac{1}{2}k^4+O(k^3)$, the theorem follows. \qed

Finally, to prove \cref{thm:Energy-DPP-hypersingular} we need the following lemmas. In these three lemmas, we consider the region $R=[-1,1]\times[0,1]$ and divide it according to \cref{fig:regions-dpp-hypersingular-2}.

\begin{figure}[htbp]
    \centering
    \begin{tikzpicture}[scale=0.8]
        \def\r{0.3\textwidth}
        \def\t{0.2\textwidth}
        \coordinate (A0) at (0,0);
        \coordinate (A1) at (\r,0);
        \coordinate (A2) at (0,\r);
        \coordinate (A3) at (-\r,0);
        \coordinate (A6) at (\t,0);
        \coordinate (A7) at (0,\t);
        \draw (A0) node[below left]{$0$};
        \draw (A1) node[below]{$1$};
        \draw (A2) node[above left]{$1$};
        \draw (A3) node[below]{$-1$};
        \draw (A6) node[below]{$\cos(k^{-1})$};
        \draw (A7) node[left]{$\cos(k^{-1})$};
        \draw[->] (-0.35\textwidth,0)--(0.4\textwidth,0) node[below]{$\xi_{+}$};
        \draw[->] (0,-0.5)--(0,0.4\textwidth) node[left]{$\xi_{-}$};
        \draw (-\r,0) rectangle (\r,\r);
        \draw (\t,\t) rectangle (\r,\r) node[pos=.5]{$\hat{R}$};
        \draw[loosely dashed, color=gray!80] (A7)--(\t,\t);
        \draw[loosely dashed, color=gray!80] (A6)--(\t,\t);
    \end{tikzpicture}
    \caption{Regions of integration inside the rectangle $R=[-1,1]\times[0,1]$.}
    \label{fig:regions-dpp-hypersingular-2}
\end{figure}
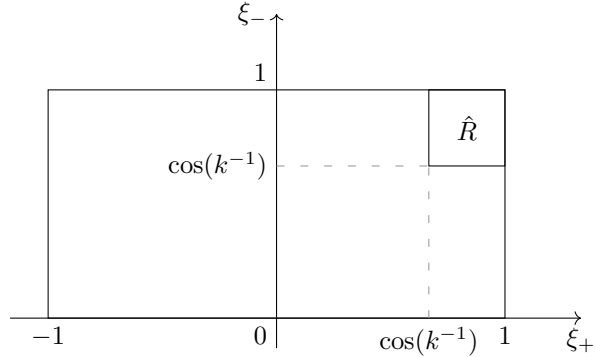

\begin{lemma}\label{lemma:integral-dpp-hyper-1}
	The following equality holds:
	\begin{equation*}
		\iint_{R\setminus \hat{R}}\frac{1}{(1-\xi_{+}\xi_{-})^2}\dd \xi_{+}\dd \xi_{-}=2\log{k}+3\log{2}+O(k^{-2}).
	\end{equation*}
\end{lemma}

\begin{proof}
	See \cref{sec:proof-lemma-integral-dpp-hyper-1}.
\end{proof}

\begin{lemma}\label{lemma:integral-dpp-hyper-2}
	Let $k\geq 1$ and let $\beta\in\set{0,1}$. The following equality holds:
	\begin{align*}
		\MoveEqLeft\iint_{R\setminus \hat{R}}\frac{\Gegenbauer{k}{3/2}(\xi_{+})\Gegenbauer{k-\beta}{3/2}(\xi_{+})\Gegenbauer{k}{3/2}(\xi_{-})\Gegenbauer{k-\beta}{3/2}(\xi_{-})}{(1-\xi_{+}\xi_{-})^2}\dd \xi_{+}\dd \xi_{-}\\
		&=4k^8\biggl(\int_{1}^{\infty}\int_{1}^{\infty}\frac{\BesselJ{1}(x)^2\BesselJ{1}(y)^2}{xy(x^2+y^2)^2}\dd x\dd y+2\int_{1}^{\infty}\int_{0}^{1}\frac{\BesselJ{1}(x)^2\BesselJ{1}(y)^2}{xy(x^2+y^2)^2}\dd x\dd y\biggr)+o(k^8).
	\end{align*}
\end{lemma}

\begin{proof}
	See \cref{sec:proof-lemma-integral-dpp-hyper-2}.
\end{proof}

\begin{lemma}\label{lemma:integral-dpp-hyper-3}
	The following equality holds:
	\begin{align*}
		\MoveEqLeft\iint_{\hat{R}}\frac{\calK_{k}(1,1)^2-\calK_{k}(\xi_{+},\xi{-})^2}{(1-\xi_{+}\xi_{-})^2}\dd \xi_{+}\dd \xi_{-}\\
		&=k^8\int_{0}^{1}\int_{0}^{1}
		\frac{x^2y^2-16\BesselJ{1}(x)^2\BesselJ{1}(y)^2}{xy(x^2+y^2)^2}\dd x\dd y+o(k^8).
	\end{align*}
\end{lemma}

\begin{proof}
	See \cref{sec:proof-lemma-integral-dpp-hyper-3}.
\end{proof}

\subsection*{Proof of \cref{thm:Energy-DPP-hypersingular}}

From \cref{prop:expectation-dpp}, \cref{lemma:integral-Grass24-2}, and \cref{eq:chordal-distance-xi}, the expected Riesz $4$-energy of this \gls{dpp} is
\begin{align*}    
    \expectation{x\distributed\HarmonicDPP{N}}{\Energy_4(x)}&=\iint\limits_{\mathclap{\Grass{2}{4}\times \Grass{2}{4}}} \frac{\calK_{k}(P,P)\calK_{k}(Q,Q)-\calK_{k}(P,Q)^2}{\chordalGrass(P,Q)^2}\dd \sigma(P)\dd \sigma(Q)\\
    &=\frac{1}{2}\int_{0}^{1}\int_{-1}^{1}\frac{\calK_{k}(1,1)^2-\calK_{k}(\xi_{+},\xi{-})^2}{(1-\xi_{+}\xi_{-})^2}\dd \xi_{+}\dd \xi_{-}.
\end{align*}
Recall from \cref{remark:KH-xx} that
\begin{equation*}
    \calK_k(1,1)=\calK_k(\xi_+(P,P),\xi_{-}(P,P))=N.
\end{equation*}
Note that the Riesz kernel $(1-\xi_{+}\xi_{-})^{-2}$ is not integrable near $(1,1)$. Hence, we divide the region $R=[-1,1]\times[0,1]$ as in \cref{fig:regions-dpp-hypersingular-2} and write
\begin{align*}    
    \MoveEqLeft\expectation{x\distributed\HarmonicDPP{N}}{\Energy_4(x)}=\frac{N^2}{2}\iint_{R\setminus \hat{R}}\frac{1}{(1-\xi_{+}\xi_{-})^2}\dd \xi_{+}\dd \xi_{-}\\
    &\qquad-\frac{1}{2}\iint_{R\setminus \hat{R}}\frac{\calK_{k}(\xi_{+},\xi{-})^2}{(1-\xi_{+}\xi_{-})^2}\dd \xi_{+}\dd \xi_{-}+\frac{1}{2}\iint_{\hat{R}}\frac{\calK_{k}(1,1)^2-\calK_{k}(\xi_{+},\xi{-})^2}{(1-\xi_{+}\xi_{-})^2}\dd \xi_{+}\dd \xi_{-}.
\end{align*}
Then, using \cref{thm:reproducing-kernel-grassmannian-dpp} and \cref{lemma:integral-dpp-hyper-1,lemma:integral-dpp-hyper-2,lemma:integral-dpp-hyper-3}, we have
\begin{align*}
    \MoveEqLeft\expectation{x\distributed\HarmonicDPP{N}}{\Energy_4(x)}=\frac{N^2}{2}(2\log{k}+3\log{2}+O(k^{-2}))\\
    &\qquad+k^8\biggl(-8\int_{1}^{\infty}\int_{1}^{\infty}\frac{\BesselJ{1}(\hat\theta)^2\BesselJ{1}(\hat\sigma)^2}{\hat\theta\hat\sigma(\hat\theta^2+\hat\sigma^2)^2}\dd \hat\theta\dd\hat\sigma-16\int_{1}^{\infty}\int_{0}^{1}\frac{\BesselJ{1}(\hat\theta)^2\BesselJ{1}(\hat\sigma)^2}{\hat\theta\hat\sigma(\hat\theta^2+\hat\sigma^2)^2}\dd \hat\theta\dd\hat\sigma\\
    &\qquad +\frac{1}{2}\int_{0}^{1}\int_{0}^{1}
    \frac{x^2y^2-16\BesselJ{1}(x)^2\BesselJ{1}(y)^2}{xy(x^2+y^2)^2}\dd x\dd y\biggr)+o(k^8).
\end{align*}
Finally, using that $N=\frac{1}{2}k^4+O(k^3)$, the theorem follows. \qed

\section{Auxiliary results for \cref{sec:energy-harmonic-ensemble}}\label{sec:auxiliary-results}

\subsection{Auxiliary lemmas for the proof of \cref{prop:energy-dpp-finite}}

\begin{lemma}\label{lemma:Taylor-Gegenbauer}
	Let $k\geq 1$ and let $B(k)=k(k+1)(k+2)(k+3)/8$. Then, for $0<t<1/k^{2}$, we have
	\begin{equation*}
	\Gegenbauer{k}{3/2}(1)-B(k)t	
    \leq \Gegenbauer{k}{3/2}(1-t)
	\leq \Gegenbauer{k}{3/2}(1)-B(k)t+K_0t^2,
	\end{equation*}
	for some constant $K_0\in (0,\infty)$.
\end{lemma}

\begin{proof}
	We proceed as in the proof of \cite[Lemma 4]{BeltranMarzoOrtegaCerda2016}. Let $q(t)=\Gegenbauer{k}{3/2}(1-t)$. In order to expand $q$ in the standard monomial basis, note that, for $0\leq j\leq k$, we have (see \cref{eq:derivatives-Gegenbauer})
	\begin{align*}
		\frac{\dd^j}{\dd t^j}q(0)&=(-1)^j\frac{\dd^j}{\dd t^j}\Gegenbauer{k}{3/2}(1)=\frac{(-1)^j2^j\Gamma(3/2+j)}{\Gamma(3/2)}\Gegenbauer{k-j}{3/2+j}(1)\\
		&=\frac{(-1)^j2^{j+1}\Gamma(3/2+j)}{\sqrt{\pi}}\frac{\Gamma(k+j+3)}{\Gamma(k-j+1)\Gamma(2j+3)},
	\end{align*}
    where in the last equality we have used \cref{eq:Gegenbauer(1)}. Then,
	\begin{equation*}
		q(t)=\sum_{j=0}^{k}\frac{(-1)^j2^{j+1}\Gamma(3/2+j)}{\sqrt{\pi}}\frac{\Gamma(k+j+3)}{\Gamma(k-j+1)\Gamma(2j+3)\Gamma(j+1)}t^j.
	\end{equation*}
	By the duplication formula for the gamma function, we have
	\begin{equation*}
		\frac{2^{j+1}\Gamma(3/2+j)}{\sqrt{\pi}\Gamma(2j+3)\Gamma(j+1)}=\frac{2^{-j-1}}{\Gamma(j+1)\Gamma(j+2)}.
	\end{equation*}
	Hence,
	\begin{equation*}
		q(t)=\frac{(k+1)(k+2)}{2}-\frac{k(k+1)(k+2)(k+3)}{8}t+R,
	\end{equation*}
	where
	\begin{equation}\label{eq:sum-R}
		R=\sum_{j=2}^{k}(-1)^j\frac{2^{-j-1}}{\Gamma(j+1)\Gamma(j+2)}\frac{\Gamma(k+j+3)}{\Gamma(k-j+1)}t^j.
	\end{equation}
    Note that
    \begin{equation*}
        \frac{(k+1)(k+2)}{2}=\Gegenbauer{k}{3/2}(1).
    \end{equation*}
	We now show that $R\geq 0$, which finishes the proof. The terms in $R$ have alternating signs, so it suffices to show that for $j=2,4,\dotso$ and $j<k$, the $j$th term in the summation \eqref{eq:sum-R} is greater than the absolute value of the $(j+1)$th term. That is, we have to show that, for those values of $j$,
	\begin{equation*}
		\frac{2^{-j-1}}{\Gamma(j+1)\Gamma(j+2)}\frac{\Gamma(k+j+3)}{\Gamma(k-j+1)}t^j\
		\geq \frac{2^{-j-2}}{\Gamma(j+2)\Gamma(j+3)}\frac{\Gamma(k+j+4)}{\Gamma(k-j)}t^{j+1},
	\end{equation*}
	which is true whenever
	\begin{equation*}
		t\leq \frac{2(j+1)(j+2)}{(k-j)(k+j+3)}.
	\end{equation*}
	A trivial computation shows that this inequality is satisfied if $t\leq 1/k^2$.
\end{proof}

\begin{lemma}\label{lemma:K_k^2}
	Let $k\geq 1$ and let
    \begin{equation*}
        \calK_k(\cos\theta,\cos\sigma)=\Gegenbauer{k}{3/2}(\cos\theta)\Gegenbauer{k}{3/2}(\cos\sigma)+\Gegenbauer{k-1}{3/2}(\cos\theta)\Gegenbauer{k-1}{3/2}(\cos\sigma).
    \end{equation*}
    For sufficiently small $\theta$ and $\sigma$, there exists a constant $c(k)$ depending only on $k$ such that
	\begin{equation*}
		\calK_k(\cos\theta,\cos\sigma)^2=\calK_k(1,1)^2-c(k)(\theta^2+\sigma^2)+O(\theta^4+\sigma^4).
	\end{equation*}
\end{lemma}

\begin{proof}
    We have
    \begin{equation}\label{eq:K_k^2}
        \calK_k(\cos\theta,\cos\sigma)^2=A_1+A_2+2A_3,
    \end{equation}
    where
    \begin{align*}
        A_1&=\Gegenbauer{k}{3/2}(\cos\theta)^2\Gegenbauer{k}{3/2}(\cos\sigma)^2,\\
        A_2&=\Gegenbauer{k-1}{3/2}(\cos\theta)^2\Gegenbauer{k-1}{3/2}(\cos\sigma)^2,\\
        A_3&=\Gegenbauer{k}{3/2}(\cos\theta)\Gegenbauer{k}{3/2}(\cos\sigma)\Gegenbauer{k-1}{3/2}(\cos\theta)\Gegenbauer{k-1}{3/2}(\cos\sigma).
    \end{align*}
    For sufficiently small $\theta$ and $\sigma$, we have
	\begin{align*}
		\cos\theta=1-\frac{\theta^2}{2}+O(\theta^4),\quad
		\cos\sigma=1-\frac{\sigma^2}{2}+O(\sigma^4).
	\end{align*}
	Let $B(k)=k(k+1)(k+2)(k+3)/8$. Using \cref{lemma:Taylor-Gegenbauer}, we have
	\begin{equation}
		\Gegenbauer{k}{3/2}(\cos\theta)=\Gegenbauer{k}{3/2}\biggl(1-\frac{\theta^2}{2}+O(\theta^4)\biggr)=\Gegenbauer{k}{3/2}(1)-B(k)\theta^2+O(\theta^4).\label{eq:aux-K^2-1}
	\end{equation}
	Hence,
	\begin{equation}
		\Gegenbauer{k}{3/2}(\cos\theta)^2=\Gegenbauer{k}{3/2}(1)^2-2B(k)\Gegenbauer{k}{3/2}(1)\theta^2+O(\theta^4).\label{eq:aux-K^2-5}
	\end{equation}
	Note that the same expressions hold for the polynomials $\Gegenbauer{k}{3/2}(\cos\sigma)$. Then, from \cref{eq:aux-K^2-5} it follows that
	\begin{align}
		A_1&=\Gegenbauer{k}{3/2}(1)^4-2B(k)\Gegenbauer{k}{3/2}(1)^3(\theta^2+\sigma^2)+O(\theta^4+\sigma^4),\label{eq:A1}\\
	    A_2&=\Gegenbauer{k-1}{3/2}(1)^4-2B(k-1)\Gegenbauer{k-1}{3/2}(1)^3(\theta^2+\sigma^2)+O(\theta^4+\sigma^4).\label{eq:A2}
	\end{align}
	From \cref{eq:aux-K^2-1}, we have
	\begin{multline*}
		\Gegenbauer{k}{3/2}(\cos\theta)\Gegenbauer{k-1}{3/2}(\cos\theta)=\Gegenbauer{k}{3/2}(1)\Gegenbauer{k-1}{3/2}(1)\\
        -(B(k)\Gegenbauer{k-1}{3/2}(1)+B(k-1)\Gegenbauer{k}{3/2}(1))\theta^2+O(\theta^4).
	\end{multline*}
	Again, the same expression holds for $\Gegenbauer{k}{3/2}(\cos\sigma)$. Therefore,
    \begin{equation}\label{eq:A3}
        A_3=\Gegenbauer{k}{3/2}(1)^2\Gegenbauer{k-1}{3/2}(1)^2-\tilde{B}(k)(\theta^2+\sigma^2)+O(\theta^4+\sigma^4),
    \end{equation}
	where
	\begin{equation*}
		\tilde{B}(k)=(B(k)\Gegenbauer{k-1}{3/2}(1)+B(k-1)\Gegenbauer{k}{3/2}(1))\Gegenbauer{k}{3/2}(1)\Gegenbauer{k-1}{3/2}(1).
	\end{equation*}
	Calling
	\begin{equation*}
		\hat{B}(k)=2B(k)\Gegenbauer{k}{3/2}(1)^3+2B(k-1)\Gegenbauer{k-1}{3/2}(1)^3+2\tilde{B}(k),
	\end{equation*}
	and substituting \eqref{eq:A1}, \eqref{eq:A2}, and \eqref{eq:A3} into \eqref{eq:K_k^2}, we conclude that
	\begin{align*}
		\calK_k(\cos\theta,\cos\sigma)^2&=\Gegenbauer{k}{3/2}(1)^4+\Gegenbauer{k-1}{3/2}(1)^4+2\Gegenbauer{k}{3/2}(1)^2\Gegenbauer{k-1}{3/2}(1)^2\\
        &\qquad-\hat{B}(k)(\theta^2+\sigma^2)+O(\theta^4+\sigma^4)\\
		&=\calK_k(1,1)^2-\hat{B}(k)(\theta^2+\sigma^2)+O(\theta^4+\sigma^4).\qedhere
	\end{align*}
\end{proof}

\subsection{Auxiliary lemmas for the proofs of \cref{thm:Energy-DPP-Riesz,thm:Energy-DPP-log,thm:Energy-DPP-hypersingular}}

Throughout this section we denote
\begin{equation}\label{eq:auxiliar-function-lemmas}
	f(\theta,\sigma)=\Gegenbauer{k}{3/2}(\cos\theta)\Gegenbauer{k-\beta}{3/2}(\cos\theta)\Gegenbauer{k}{3/2}(\cos\sigma)\Gegenbauer{k-\beta}{3/2}(\cos\sigma),\quad \beta\in\set{0,1},
\end{equation}
and we will be working on the region sketched in \cref{fig:regions-dpp}.

\begin{figure}[htbp]
    \centering
    \begin{tikzpicture}[scale=0.9]
        \def\R{0.8\textwidth}
        \def\r{0.4\textwidth}
        \def\s{0.1\textwidth}
        \def\t{0.15\textwidth}
        \def\u{0.7\textwidth}
        \coordinate (A0) at (0,0);
        \coordinate (A1) at (\R,0);
        \coordinate (A2) at (0,\r);
        \coordinate (A3) at (\r,0);
        \coordinate (A4) at (\s,0);
        \coordinate (A5) at (0,\s);
        \coordinate (A6) at (\t,0);
        \coordinate (A7) at (0,\t);
        \coordinate (A8) at (\u,0);
        \coordinate (A9) at (\R,\s);
        \draw (A0) node[below left]{$0$};
        \draw[->] (-0.5,0)--(0.87\textwidth,0) node[below]{$\theta$};
        \draw[->] (0,-0.5)--(0,0.47\textwidth) node[left]{$\sigma$};
        \draw[loosely dashed, color=gray!80] (A5)--(\r,\s);
        \draw (0,0) rectangle (\R,\r);
        \draw (\r,0) rectangle (\R,\r);
        \draw (A8) rectangle (A9) node[pos=.5]{$R_{1}$};
        \draw (\r,0) rectangle (\u,\s) node[pos=.5]{$R_{2}$};
        \draw (\u,\s) rectangle (\R,\r) node[pos=.5]{$R_{3}$};
        \draw (\r,\s) rectangle (\u,\r) node[pos=.5]{$R_{4}$};
        \draw (\t,\t) rectangle (\r,\r) node[pos=.5]{$R_{5}$};
        \draw (0,\t) rectangle (\t,\r) node[pos=.5]{$R_{6}$};
        \draw (\t,0) rectangle (\r,\t) node[pos=.5]{$R_{7}$};
        \draw (0,0) rectangle (\t,\t) node[pos=.5]{$R_{8}$};
        \draw (A5) node[left]{$k^{-1}$};
        \draw (A1) node[below]{$\pi$};
        \draw (A2) node[left]{$\pi/2$};
        \draw (A3) node[below]{$\pi/2$};
        \draw (A6) node[below]{$k^{-1/2}$};
        \draw (A7) node[left]{$k^{-1/2}$};
        \draw (A8) node[below]{$\pi-k^{-1}$};
    \end{tikzpicture}
    \caption{Regions of integration inside the rectangle $[0,\pi]\times[0,\pi/2]$.}
    \label{fig:regions-dpp}
\end{figure}
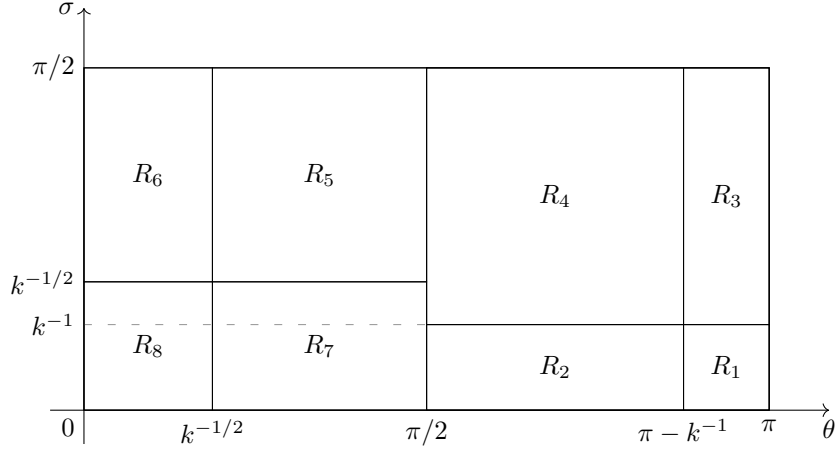

\begin{lemma}\label{lemma:integral-R1}
	The following equality holds:
	\begin{equation*}
		\iint_{R_1}f(\theta,\sigma)\sin\theta\sin\sigma\dd\theta\dd\sigma=O(k^4).
	\end{equation*}
\end{lemma}

\begin{proof}
	In this region we have $\sin\theta,\sin\sigma=O(k^{-1})$. Using the bounds \eqref{eq:bound-gegenbauer} for the Gegenbauer polynomials, we obtain
	\begin{align*}
		\abs[\bigg]{\iint_{R_{1}}f(\theta,\sigma)\sin\theta\sin\sigma\dd\theta\dd\sigma}&\leq \iint_{R_{1}}\abs{f(\theta,\sigma)\sin\theta\sin\sigma}\dd\theta\dd\sigma\\
		&=O\biggl(k^8k^{-2}\iint_{R_{1}}\dd\theta\dd\sigma\biggr)=O(k^4).\qedhere
	\end{align*}
\end{proof}

\begin{lemma}\label{lemma:integral-R2-R3}
	The following equality holds:
	\begin{equation*}
		\iint_{R_2}f(\theta,\sigma)\sin\theta\sin\sigma\dd\theta\dd\sigma=\iint_{R_3}f(\theta,\sigma)\sin\theta\sin\sigma\dd\theta\dd\sigma=O(k^3).
	\end{equation*}
\end{lemma}

\begin{proof}
	Since the analysis over the regions $R_2$ and $R_3$ is equivalent, we will focus on the region $R_{2}$. In this region we have
	\begin{align*}
		\frac{\pi}{2}\leq \theta\leq \pi-k^{-1},\quad 0\leq \sigma\leq k^{-1}.
	\end{align*}
	Using the bounds in \eqref{eq:bound-gegenbauer} for the Gegenbauer polynomials and the estimates $\sin\theta=O(\theta)$ and $\sin\sigma=O(k^{-1})$, we obtain
	\begin{equation*}
		\abs[\bigg]{\iint_{R_{2}}f(\theta,\sigma)\sin\theta\sin\sigma\dd\theta\dd\sigma}
        =O\biggl(k^4\int_{0}^{k^{-1}}\int_{\pi/2}^{\pi-k^{-1}}\theta^{-2}\dd\theta\dd\sigma\biggr)=O(k^3).\qedhere
	\end{equation*}
\end{proof}
\begin{lemma}\label{lemma:integral-R4}
	The following equality holds:
	\begin{equation*}
		\iint_{R_4}f(\theta,\sigma)\sin\theta\sin\sigma\dd\theta\dd\sigma=O(k^3).
	\end{equation*}
\end{lemma}

\begin{proof}
	In this region we have $\sin\theta=O(\theta)$ and $\sin\sigma=O(\sigma)$. Hence, using the bounds \eqref{eq:bound-gegenbauer}, we obtain
	\begin{equation*}
		\abs[\bigg]{\iint_{R_{4}}f(\theta,\sigma)\sin\theta\sin\sigma\dd\theta}
        =O\biggl(k^2\int_{k^{-1}}^{\pi/2}\int_{\pi/2}^{\pi-k^{-1}}\theta^{-2}\sigma^{-2}\dd\theta\dd\sigma\biggr)=O(k^3).\qedhere
	\end{equation*}
\end{proof}

\begin{lemma}\label{lemma:integral-R5}
	Let $s\geq 0$. The following equality holds:
	\begin{equation*}
		\iint_{R_5}\frac{f(\theta,\sigma)\sin\theta\sin\sigma}{(1-\cos\theta\cos\sigma)^{s/2}}\dd\theta\dd\sigma=O(k^{3+s/2}).
	\end{equation*}
\end{lemma}

\begin{proof}
	In this region we have $\sin\theta=O(\theta)$, $\sin\sigma=O(\sigma)$, and $\cos(\theta),\cos(\sigma)\leq \cos k^{-1/2}$. Using that $\cos{x}\leq 1-x^2/4$ for $x\in[0,1]$, we have
	\begin{equation*}
		1-\cos\theta\cos\sigma\geq 1-\biggl(1-\frac{k^{-1}}{4}\biggr)^2.
	\end{equation*}
	Hence,
	\begin{equation*}
		\frac{1}{(1-\cos\theta\cos\sigma)^{s/2}}=O(k^{s/2}).
	\end{equation*}
	Using the estimates above together with the bounds \eqref{eq:bound-gegenbauer} for the Gegenbauer polynomials, we obtain
	\begin{equation*}
		\abs[\bigg]{\iint_{R_5} \frac{f(\theta,\sigma)\sin\theta\sin\sigma}{(1-\cos\theta\cos\sigma)^{s/2}}\dd\theta\dd\sigma}
        =O\biggl(k^{2+s/2}\iint_{R_5}\theta^{-2}\sigma^{-2}\dd\theta\dd\sigma\biggr)=O(k^{3+s/2}).\qedhere
	\end{equation*}
\end{proof}

\begin{lemma}\label{lemma:integral-R6-R7}
	Let $s\geq 0$. The following equality holds:
	\begin{equation*}
		\iint_{R_6}\frac{f(\theta,\sigma)\sin\theta\sin\sigma}{(1-\cos\theta\cos\sigma)^{s/2}}\dd\theta\dd\sigma=\iint_{R_7}\frac{f(\theta,\sigma)\sin\theta\sin\sigma}{(1-\cos\theta\cos\sigma)^{s/2}}\dd\theta\dd\sigma=O(k^{15/4+s/2}).
	\end{equation*}
\end{lemma}

\begin{proof}
	Since the analysis over the regions $R_6$ and $R_7$ is equivalent, we will focus on the region $R_{7}$. In this region we have $\sin\theta=O(\theta)$, $\sin\sigma=O(\sigma)$, $\cos\sigma\leq 1$, and $\cos\theta\leq \cos k^{-1/2}$. Hence, using that $\cos{x}\leq 1-x^2/4$ for $x\in[0,1]$, we have
	\begin{equation*}
		1-\cos\theta\cos\sigma\geq 1-\biggl(1-\frac{k^{-1}}{4}\biggr),
	\end{equation*}
	and so
	\begin{equation*}
		\frac{1}{(1-\cos\theta\cos\sigma)^{s/2}}=O(k^{s/2}).
	\end{equation*}
    To estimate $f(\theta,\sigma)$ in this region, we use the bounds \eqref{eq:bound-gegenbauer} for the Gegenbauer polynomials as follows:
     \begin{align*}
        \Gegenbauer{k}{3/2}(\cos\theta)&=\theta^{-3/2}O(k^{1/2}),\qquad\Gegenbauer{k-\beta}{3/2}(\cos\theta)=\theta^{-3/2}O(k^{1/2}),\\
        \Gegenbauer{k}{3/2}(\cos\sigma)&=\sigma^{-3/2}O(k^{1/2}), 
        \qquad\Gegenbauer{k-\beta}{3/2}(\cos\sigma)=O(k^2).
    \end{align*}
	Using the estimates above, we obtain
    \begin{align*}
		\abs[\bigg]{\iint_{R_7} \frac{f(\theta,\sigma)\sin\theta\sin\sigma}{(1-\cos\theta\cos\sigma)^{s/2}}\dd\theta\dd\sigma}
        &=O\biggl(k^{7/2+s/2}\iint_{R_7}\theta^{-2}\sigma^{-1/2}\dd\theta\dd\sigma\biggr)\\
        &=O(k^{15/4+s/2}).\qedhere
	\end{align*}    
\end{proof}

\begin{lemma}\label{lemma:integral-Riesz-R8}
	Let $0\leq s<4$. The following equality holds:
	\begin{equation*}
		\iint_{R_8} \frac{f(\theta,\sigma)\sin\theta\sin\sigma}{(1-\cos\theta\cos\sigma)^{s/2}}\dd\theta\dd\sigma=2^{s/2}k^{4+s}\int_0^{\infty}\int_0^{\infty}\frac{\BesselJ{1}(x)^2\BesselJ{1}(y)^2}{xy(x^2+y^2)^{s/2}}\dd x\dd y+o(k^{4+s}).
	\end{equation*}
\end{lemma}

\begin{proof}
	It suffices to compute the limit
	\begin{equation}\label{eq:Riesz-R8-aux-1}
		\lim_{k\tendsto\infty}\frac{1}{k^{4+s}}\int_{0}^{k^{-1/2}}\int_{0}^{k^{-1/2}}\frac{f(\theta,\sigma)\sin\theta\sin\sigma}{(1-\cos\theta\cos\sigma)^{s/2}}\dd\theta\dd\sigma.
	\end{equation}
	Over this region we have $\sin\theta\sim\theta$, $\sin\sigma\sim \sigma$, and
	\begin{equation*}
		\frac{1}{(1-\cos\theta\cos\sigma)^{s/2}}\sim\frac{2^{s/2}}{(\theta^2+\sigma^2)^{s/2}}.
	\end{equation*}
	Hence, 	the limit in \eqref{eq:Riesz-R8-aux-1} equals
	\begin{equation*}
		2^{s/2}\lim_{k\tendsto\infty}\frac{1}{k^{4+s}}\int_0^{k^{-1/2}}\int_0^{k^{-1/2}}\frac{f(\theta,\sigma)\theta\sigma}{(\theta^2+\sigma^2)^{s/2}}(1+o(1))\dd\theta\dd\sigma.
	\end{equation*}
	Making the substitution $(\theta,\sigma)\mapsto(\hat\theta/k,\hat\sigma/k)$, the expression above can be rewritten as follows:
	\begin{align}
		\MoveEqLeft 2^{s/2}\lim_{k\tendsto\infty}\int_0^{k^{1/2}}\int_0^{k^{1/2}}\frac{1}{k^{8}} \frac{f(\hat\theta/k,\hat\sigma/k)\hat\theta\hat\sigma}{(\hat\theta^2+\hat\sigma^2)^{s/2}}(1+o(1))\dd\hat\theta\dd\hat\sigma\notag\\
		&=2^{s/2}\lim_{k\tendsto\infty}\int_0^{\infty}\int_0^{\infty}\chi_{[0,k^{1/2}]}(\hat\theta)\chi_{[0,k^{1/2}]}(\hat\sigma)\frac{1}{k^{8}}
		\frac{f(\hat\theta/k,\hat\sigma/k)\hat\theta\hat\sigma}{(\hat\theta^2+\hat\sigma^2)^{s/2}}(1+o(1))\dd\hat\theta\dd\hat\sigma,\label{eq:Riesz-R8-aux-2}
	\end{align}
	where $\chi$ is the indicator function. It can be checked, using the bounds in \eqref{eq:bound-gegenbauer} for the Gegenbauer polynomials, that the function $g\from[0,\infty)\times[0,\infty)\to\R$ given by
	\begin{equation*}
		g(\hat\theta,\hat\sigma)=\begin{cases}
			\hat\theta\hat\sigma(\hat\theta^2+\hat\sigma^2)^{-s/2}	& (\hat\theta,\hat\sigma)\in[0,1]\times[0,1],\\
			\hat\sigma^{-2-s}	& (\hat\theta,\hat\sigma)\in[0,1]\times(1,\infty),\\
			\hat\theta^{-2-s}	& (\hat\theta,\hat\sigma)\in(1,\infty)\times[0,1],\\
			\hat\theta^{-2}\hat\sigma^{-2}(\hat\theta^2+\hat\sigma^2)^{-s/2}	& (\hat\theta,\hat\sigma)\in(1,\infty)\times(1,\infty),
		\end{cases}
	\end{equation*}
	is integrable in $[0,\infty)\times[0,\infty)$ and satisfies
	\begin{equation*}
		\abs[\bigg]{\frac{\chi_{[0,k^{1/2}]}(\hat\theta)\chi_{[0,k^{1/2}]}(\hat\sigma)}{k^8}\frac{f(\hat\theta/k,\hat\sigma/k)\hat\theta\hat\sigma}{(\hat\theta^2+\hat\sigma^2)^{s/2}}}\lesssim g(\hat\theta,\hat\sigma).
	\end{equation*}
	Hence, using the Dominated Convergence Theorem and the Mehler--Heine formula \eqref{eq:Mehler-Heine-Gegenbauer}, the expression in \eqref{eq:Riesz-R8-aux-2} equals
	\begin{equation*}
		2^{s/2}\int_0^{\infty}\int_0^{\infty}\frac{\BesselJ{1}(\hat\theta)^2\BesselJ{1}(\hat\sigma)^2}{\hat\theta\hat\sigma(\hat\theta^2+\hat\sigma^2)^{s/2}}\dd\hat\theta\dd\hat\sigma,
	\end{equation*}
	as wanted.
\end{proof}

\begin{lemma}\label{lemma:integral-log-R1}
	Let $f$ be as in \cref{eq:auxiliar-function-lemmas}. The following equality holds:
	\begin{align*}
		\MoveEqLeft\iint_{R_1}f(\theta,\sigma)\sin\theta\sin\sigma\log(1-\cos\theta\cos\sigma)\dd\theta\dd\sigma\\
		&=(-1)^{\beta}\frac{\log(2)}{4}\Bigl(1-\BesselJ{0}(1)^2-\BesselJ{1}(1)^2\Bigr)^2k^4+o(k^4).
	\end{align*}
\end{lemma}

\begin{proof}
	It suffices to compute the following limit
	\begin{align*}
		\MoveEqLeft\lim_{k\tendsto\infty}\frac{1}{k^4}\int_{0}^{k^{-1}}\int_{\pi-k^{-1}}^{\pi}f(\theta,\sigma)\sin\theta\sin\sigma\log(1-\cos\theta\cos\sigma)\dd \theta\dd\sigma\\
		&=(-1)^{\beta}\lim_{k\tendsto\infty}\frac{1}{k^4}\int_{0}^{k^{-1}}\int_{0}^{k^{-1}}f(\theta,\sigma)\sin\theta\sin\sigma\log(1+\cos\theta\cos\sigma)\dd \theta\dd\sigma,
	\end{align*}
	where we have used that $\Gegenbauer{k}{3/2}(-x)=(-1)^k\Gegenbauer{k}{3/2}(x)$. Over this region we have $\sin\theta\sim\theta$, $\sin\sigma\sim\sigma$, and $\log(1+\cos\theta\cos\sigma)\sim\log(2)$. Hence, the limit above is
	\begin{equation*}
		(-1)^{\beta}\log(2)\lim_{k\tendsto\infty}\frac{1}{k^4}\int_{0}^{k^{-1}}\int_{0}^{k^{-1}}f(\theta,\sigma)\theta\sigma(1+o(1))\dd\theta\dd\sigma.
	\end{equation*}
	Making the substitution $(\theta,\sigma)\mapsto(\hat\theta/k,\hat\sigma/k)$, this expression reads as
	\begin{equation}\label{eq:R2-4-aux}
		(-1)^{\beta}\log(2)\lim_{k\tendsto\infty}\int_{0}^{1}\int_{0}^{1}\frac{1}{k^8}f\biggl(\frac{\hat\theta}{k},\frac{\hat\sigma}{k}\biggr)\hat\theta\hat\sigma(1+o(1))\dd \hat\theta\dd\hat\sigma.
	\end{equation}
	From \eqref{eq:bound-gegenbauer}, we have
	\begin{align*}
		\abs[\bigg]{\frac{1}{k^8}f\biggl(\frac{\hat\theta}{k},\frac{\hat\sigma}{k}\biggr)\hat\theta\hat\sigma}\lesssim \hat\theta\hat\sigma.
	\end{align*}
	Since $\hat\theta\hat\sigma$ is integrable in $[0,1]^2$, we can simplify \eqref{eq:R2-4-aux} using the Dominated Convergence Theorem and the Mehler--Heine formula \eqref{eq:Mehler-Heine-Gegenbauer} as
	\begin{align*}
		(-1)^{\beta}\log(2)\int_{0}^{1}\int_{0}^{1}\frac{\BesselJ{1}(\hat\theta)^2\BesselJ{1}(\hat\sigma)^2}{\hat\theta\hat\sigma}\dd \hat\theta\dd\hat\sigma&=(-1)^{\beta}\log(2)\biggl(\int_{0}^{1}\frac{\BesselJ{1}(\hat\theta)^2}{\hat\theta}\dd \hat\theta\biggr)^2\\
		&=(-1)^{\beta}\frac{\log(2)}{4}\Bigl(1-\BesselJ{0}(1)^2-\BesselJ{1}(1)^2\Bigr)^2,
	\end{align*} 
	where in the last equality we have used \cite[Formula 1.8.3-2]{PrudnikovBrychkovMarichev1986}.
\end{proof}

\begin{lemma}\label{lemma:integral-with-Hilb}
	Let $\beta\in\set{0,1}$. The following equality holds:
	\begin{align*}
		\int_{0}^{k^{-1/2}}\Gegenbauer{k}{3/2}(\cos\theta)\Gegenbauer{k-\beta}{3/2}(\cos\theta)\theta\dd\theta=\frac{k^2}{2}(1+O(k^{-1/2})).
	\end{align*}
\end{lemma}

\begin{proof}
	To control the error in the asymptotic expansion we use the following version of Hilb's formula for the case of the Gegenbauer polynomials $\Gegenbauer{n}{3/2}$ (see \cite[Theorem 8.21.12]{Szego1975}):
	\begin{equation}\label{eq:Hilb-Gegenbauer-3}
		\Gegenbauer{k}{3/2}(\cos\theta)=\frac{(k+2)(k+1)}{k+3/2}\frac{\theta^{1/2}}{\sin^{3/2}\theta}\BesselJ{1}((k+3/2)\theta)+\text{Error},
	\end{equation}
	where $\BesselJ{1}$ is the Bessel function of the first kind of order $1$ and
	\begin{equation*}
		\text{Error}=\begin{cases}
			(\theta^{1/2}/\sin\theta)O(k^{-3/2}) & k^{-1}\leq \theta\leq \pi/2,\\
			(\theta^{3}/\sin\theta)O(k) & 0< \theta\leq k^{-1}.\\
		\end{cases}
	\end{equation*}
	Set $\tilde{k}=(k+3/2)$. In this case $\theta\in[0,k^{-1/2}]$. Using that $\sin\theta\sim \theta$, we can rewrite \eqref{eq:Hilb-Gegenbauer-3} as
	\begin{equation*}
		\Gegenbauer{k}{3/2}(\cos\theta)=\frac{k}{\theta}\BesselJ{1}(\tilde{k}\theta)(1+O(k^{-1}))+\text{Error},
	\end{equation*}
	where
	\begin{equation*}
		\text{Error}=\begin{cases}
			\theta^{-1/2}O(k^{-3/2}) & k^{-1}\leq \theta\leq k^{-1/2},\\
			\theta^2 O(k) & 0< \theta\leq k^{-1}.
		\end{cases}
	\end{equation*}
	For the polynomials $\Gegenbauer{k-\beta}{3/2}$, Hilb's formula takes the form
    \begin{equation*}
		\Gegenbauer{k-\beta}{3/2}(\cos\theta)=\frac{k}{\theta}\BesselJ{1}((\tilde{k}-\beta)\theta)(1+O(k^{-1}))+\text{Error},
	\end{equation*}
    where the error term remains the same. Hence,
	\begin{multline*}
		\Gegenbauer{k}{3/2}(\cos\theta)\Gegenbauer{k-\beta}{3/2}(\cos\theta)=\frac{k^2}{\theta^2}\BesselJ{1}(\tilde{k}\theta)\BesselJ{1}((\tilde{k}-\beta)\theta)(1+O(k^{-1}))+\text{Error}^2\\
		+\frac{k}{\theta}\BesselJ{1}(\tilde{k}\theta)\text{Error}(1+O(k^{-1}))+\frac{k}{\theta}\BesselJ{1}((\tilde{k}-\beta)\theta)\text{Error}(1+O(k^{-1})).
	\end{multline*}
	Using these estimates, the integral of the lemma equals
	\begin{align*}
		\MoveEqLeft\int_{0}^{k^{-1/2}}\frac{k^2}{\theta}\BesselJ{1}(\tilde{k}\theta)\BesselJ{1}((\tilde{k}-\beta)\theta)(1+O(k^{-1}))\dd\theta+\int_{0}^{k^{-1/2}}\theta\,\text{Error}^2\dd\theta\\
		&\qquad+\int_{0}^{k^{-1/2}}k\Bigl(\BesselJ{1}(\tilde{k}\theta)+\BesselJ{1}((\tilde{k}-\beta)\theta)\Bigr)\text{Error}(1+O(k^{-1}))\dd\theta.
	\end{align*}
	The integrals
	\begin{equation*}
		\int_{0}^{k^{-1/2}}\theta\,\text{Error}^2\dd\theta+\int_{0}^{k^{-1/2}}k\Bigl(\BesselJ{1}(\tilde{k}\theta)+\BesselJ{1}((\tilde{k}-\beta)\theta)\Bigr)\text{Error}(1+O(k^{-1}))\dd\theta
	\end{equation*}
	can be estimated as $O(k^{-3/4})$ using that $\abs{\BesselJ{1}(x)}\leq 1$ for all $x\in\R$. For the remaining integral, making the substitution $\theta\mapsto \hat{\theta}/k$ we have
	\begin{align*}
		\MoveEqLeft\int_{0}^{k^{-1/2}}\frac{k^2}{\theta}\BesselJ{1}(\tilde{k}\theta)\BesselJ{1}((\tilde{k}-\beta)\theta)(1+O(k^{-1}))\dd\theta\\
        &=\int_{0}^{k^{1/2}}\frac{k^2}{\hat\theta}\BesselJ{1}(\tilde{k}\hat\theta/k)\BesselJ{1}((\tilde{k}-\beta)\hat\theta/k)(1+O(k^{-1}))\dd\hat\theta\\
		&=k^2(1+O(k^{-1}))\biggl(\int_{0}^{\infty}\frac{\BesselJ{1}(\tilde{k}\hat\theta/k)\BesselJ{1}((\tilde{k}-\beta)\hat\theta/k)}{\hat\theta}\dd\hat\theta\\
        &\hspace{4cm}-\int_{k^{1/2}}^{\infty}\frac{\BesselJ{1}(\tilde{k}\hat\theta/k)\BesselJ{1}((\tilde{k}-\beta)\hat\theta/k)}{\hat\theta}\dd\hat\theta\biggr).
	\end{align*}
	Using that $\BesselJ{1}(x)=O(x^{-1/2})$ as $x\tendsto\infty$, we derive the estimate
	\begin{align*}
		\int_{k^{1/2}}^{\infty}\frac{\BesselJ{1}(\tilde{k}\hat\theta/k)\BesselJ{1}((\tilde{k}-\beta)\hat\theta/k)}{\hat\theta}\dd\hat\theta=O(k^{-1/2}).
	\end{align*}
	Finally, using \cite[Formulas 6.574-1 and 6.574-2]{GradshteynRyzhik2007} we obtain
	\begin{equation*}
		\int_{0}^{\infty}\frac{\BesselJ{1}(\tilde{k}\hat\theta/k)\BesselJ{1}((\tilde{k}-\beta)\hat\theta/k)}{\hat\theta}\dd\hat\theta=\frac{1}{2}\frac{\tilde{k}-\beta}{\tilde{k}}=\frac{1}{2}(1+O(k^{-1})).
	\end{equation*}
	The lemma follows.
\end{proof}

\begin{lemma}\label{lemma:integral-log-R8}
	The following equality holds:
	\begin{align*}
		\MoveEqLeft\iint_{R_8}f(\theta,\sigma)\sin\theta\sin\sigma\log(1-\cos\theta\cos\sigma)\dd\theta\dd\sigma\\
		&=-\frac{k^4\log{k}}{2}+k^4\biggl(\frac{4+8G+\pi-4\gamma\pi+4\pi\log{2}}{8\pi}-\frac{\log{2}}{4}\biggr)+o(k^4),
	\end{align*}
    where $G$ is Catalan's constant and $\gamma$ is the Euler--Mascheroni constant.
\end{lemma}

\begin{proof}
	Using the substitution
	\begin{equation}\label{eq:substitution}
		(\theta,\sigma)\mapsto(\hat\theta/k,\hat\sigma/k),
	\end{equation}
	the integral of the lemma equals
	\begin{equation}\label{eq:aux-log-R8-1}
		\frac{1}{k^2}\int_{0}^{k^{1/2}}\int_{0}^{k^{1/2}}f(\hat\theta/k,\hat\sigma/k)\sin(\hat\theta/k)\sin(\hat\sigma/k)\log(1-\cos(\hat\theta/k)\cos(\hat\sigma/k))\dd\hat\theta\dd\hat\sigma.
	\end{equation}
	In this region we have
	\begin{align*}
		\log(1-\cos(\hat\theta/k)\cos(\hat\sigma/k))
		&=\log\biggl(\frac{\hat\theta^2+\hat\sigma^2}{2k^2}(1+O(k^{-1}))\biggr)\\
		&=\log(\hat\theta^2+\hat\sigma^2)-2\log{k}-\log(2)+\log(1+O(k^{-1})).
	\end{align*}
	Hence, the integral in \eqref{eq:aux-log-R8-1} can be written as
	\begin{equation*}
		I_1+I_2+I_3+I_4,
	\end{equation*}
	where
	\begin{align*}
		I_1&=\frac{1}{k^2}\int_{0}^{k^{1/2}}\int_{0}^{k^{1/2}}f(\hat\theta/k,\hat\sigma/k)\sin(\hat\theta/k)\sin(\hat\sigma/k)\log(\hat\theta^2+\hat\sigma^2)\dd\hat\theta\dd\hat\sigma,\\
		I_2&=-\frac{2\log(k)}{k^2}\int_{0}^{k^{1/2}}\int_{0}^{k^{1/2}}f(\hat\theta/k,\hat\sigma/k)\sin(\hat\theta/k)\sin(\hat\sigma/k)\dd\hat\theta\dd\hat\sigma,\\
		I_3&=-\frac{\log(2)}{k^2}\int_{0}^{k^{1/2}}\int_{0}^{k^{1/2}}f(\hat\theta/k,\hat\sigma/k)\sin(\hat\theta/k)\sin(\hat\sigma/k)\dd\hat\theta\dd\hat\sigma,\\
		I_4&=\frac{\log(1+O(k^{-1}))}{k^2}\int_{0}^{k^{1/2}}\int_{0}^{k^{1/2}}f(\hat\theta/k,\hat\sigma/k)\sin(\hat\theta/k)\sin(\hat\sigma/k)\dd\hat\theta\dd\hat\sigma.
	\end{align*}
	Reversing the change of variables \eqref{eq:substitution}, we can use \cref{lemma:integral-Riesz-R8} with $s=0$ and \cite[Formula 6.574-2]{GradshteynRyzhik2007} to derive
	\begin{equation*}
		I_3=-k^4\frac{\log(2)}{4}+o(k^4).
	\end{equation*}
	Similarly, we have $I_4=o(k^4)$. To compute $I_2$ we use the estimates
	\begin{equation*}
		\sin(\hat\theta/k)=\frac{\hat\theta}{k}(1+O(k^{-1})),\quad \sin(\hat\sigma/k)=\frac{\hat\theta}{k}(1+O(k^{-1})).
	\end{equation*}
	Then, we can simplify $I_2$ as
	\begin{equation*}
		I_2=\frac{-2\log(k)}{k^2}(1+O(k^{-1}))\biggl(\int_{0}^{k^{1/2}}\Gegenbauer{k}{3/2}(\cos(\hat\theta/k))\Gegenbauer{k-\beta}{3/2}(\cos(\hat\theta/k))\frac{\hat\theta}{k}\dd\hat\theta\biggr)^2.
	\end{equation*}
	Reversing the substitution \eqref{eq:substitution} and using \cref{lemma:integral-with-Hilb}, we obtain
	\begin{equation*}
		I_2=-\frac{k^4}{2}\log{k}+o(k^4).
	\end{equation*}
	Finally, we claim that 
	\begin{equation*}
		I_1=k^4\biggl(\frac{4+8G+\pi-4\gamma\pi+4\pi\log{2}}{8\pi}\biggr)+o(k^4).
	\end{equation*}
	To prove this claim, it suffices to compute the limit $\lim_{k\tendsto\infty}I_1/k^4$. Using the asymptotic equivalences $\sin(\hat\theta/k)\sim\hat\theta/k$ and $\sin(\hat\sigma/k)\sim\hat\sigma/k$, we have
	\begin{align}
		&\lim_{k\tendsto\infty}\frac{1}{k^4}I_1=\lim_{k\tendsto\infty}\frac{1}{k^8}\int_{0}^{k^{1/2}}\int_{0}^{k^{1/2}}f(\hat\theta/k,\hat\sigma/k)\hat\theta\hat\sigma\log(\hat\theta^2+\hat\sigma^2)(1+o(1))\dd\hat\theta\dd\hat\sigma\label{eq:aux-log-R8-2}\\
		&=\lim_{k\tendsto\infty}\int_{0}^{\infty}\int_{0}^{\infty}\chi_{[0,k^{1/2}]}(\hat\theta)\chi_{[0,k^{1/2}]}(\hat\sigma)\frac{f(\hat\theta/k,\hat\sigma/k)\hat\theta\hat\sigma\log(\hat\theta^2+\hat\sigma^2)}{k^8}(1+o(1))\dd\hat\theta\dd\hat\sigma.\notag
	\end{align}
	It can be checked, using the bounds in \eqref{eq:bound-gegenbauer} for the Gegenbauer polynomials, that the function $g\from [0,\infty)\times [0,\infty)$ given by
	\begin{equation*}
		g(\hat\theta,\hat\sigma)=\begin{cases}
			\hat\theta\hat\sigma\log(\hat\theta^2+\hat\sigma^2)	& (\hat\theta,\hat\sigma)\in[0,1]\times[0,1],\\
			\hat\theta\hat\sigma^{-2}\log(\hat\theta^2+\hat\sigma^2)	& (\hat\theta,\hat\sigma)\in[0,1]\times(1,\infty),\\
			\hat\theta^{-2}\hat\sigma\log(\hat\theta^2+\hat\sigma^2)	& (\hat\theta,\hat\sigma)\in(1,\infty)\times[0,1],\\
			\hat\theta^{-2}\hat\sigma^{-2}\log(\hat\theta^2+\hat\sigma^2)	& (\hat\theta,\hat\sigma)\in(1,\infty)\times(1,\infty),
		\end{cases}
	\end{equation*}
	is integrable over $[0,\infty)\times [0,\infty)$ and satisfies
	\begin{equation*}
		\abs[\bigg]{\frac{\chi_{[0,k^{1/2}]}(\hat\theta)\chi_{[0,k^{1/2}]}(\hat\sigma)}{k^8}f(\hat\theta/k,\hat\sigma/k)\hat\theta\hat\sigma\log(\hat\theta^2+\hat\sigma^2)}\lesssim g(\hat\theta,\hat\sigma).
	\end{equation*}
	Hence, using the Dominated Convergence Theorem and the Mehler--Heine formula \eqref{eq:Mehler-Heine-Gegenbauer}, we have from \eqref{eq:aux-log-R8-2} that
	\begin{align*}
		I_1=k^4\int_{0}^{\infty}\int_{0}^{\infty}\frac{\BesselJ{1}(\hat\theta)^2\BesselJ{1}(\hat\sigma)^2}{\hat\theta\hat\sigma}\log(\hat\theta^2+\hat\sigma^2)\dd \hat\theta\dd\hat\sigma+o(k^4).
	\end{align*}
	Finally, the integral in this expression can be computed using the computer algebra system \texttt{Mathematica} \cite{Mathematica}:
	\begin{equation*}
		\int_{0}^{\infty}\int_{0}^{\infty}\frac{\BesselJ{1}(\hat\theta)^2\BesselJ{1}(\hat\sigma)^2}{\hat\theta\hat\sigma}\log(\hat\theta^2+\hat\sigma^2)\dd \hat\theta\dd\hat\sigma=\frac{4+8G+\pi-4\gamma\pi+4\pi\log{2}}{8\pi}.
	\end{equation*}
	Adding the contributions from $I_1$ to $I_4$, the lemma follows.
\end{proof}

\subsubsection{Proof of \cref{lemma:integral-dpp}}\label{sec:proof-lemma-integral-dpp}

Making the substitution $(\xi_{+},\xi_{-})\mapsto(\cos\theta,\cos\sigma)$ and calling
\begin{equation*}
	f(\theta,\sigma)=\Gegenbauer{k}{3/2}(\cos\theta)\Gegenbauer{k-\beta}{3/2}(\cos\theta)\Gegenbauer{k}{3/2}(\cos\sigma)\Gegenbauer{k-\beta}{3/2}(\cos\sigma),
\end{equation*}
the integral of the lemma equals
\begin{equation*}
	\int_{0}^{\pi/2}\int_{0}^{\pi}\frac{f(\theta,\sigma)\sin\theta\sin\sigma}{(1-\cos\theta\cos\sigma)^{s/2}}\dd \theta \dd \sigma.
\end{equation*}
We divide the region $[0,\pi]\times [0,\pi/2]$ as in \cref{fig:regions-dpp}. Over the regions $R_1$ to $R_4$ we have $1-\cos\theta\cos\sigma\geq 1$. Hence, using \cref{lemma:integral-R1,lemma:integral-R2-R3,lemma:integral-R4,lemma:integral-R5,lemma:integral-R6-R7}, we have
\begin{equation*}
	\iint_{\bigunion_{i=1}^{7}R_j}\frac{f(\theta,\sigma)\sin\theta\sin\sigma}{(1-\cos\theta\cos\sigma)^{s/2}}\dd \theta \dd \sigma=o(k^{4+s}).
\end{equation*}
Finally, from \cref{lemma:integral-Riesz-R8} we have
\begin{align*}
	\iint_{R_8}\frac{f(\theta,\sigma)\sin\theta\sin\sigma}{(1-\cos\theta\cos\sigma)^{s/2}}\dd \theta \dd \sigma=2^{s/2}k^{4+s}\int_0^{\infty}\int_0^{\infty} \frac{\BesselJ{1}(x)^2\BesselJ{1}(y)^2}{xy(x^2+y^2)^{s/2}}\dd x\dd y +o(k^{4+s}).
\end{align*}
The proof is now complete.\qed

\subsubsection{Proof of \cref{lemma:integral-dpp-log}}\label{sec:proof-lemma-integral-dpp-log}

Making the substitution $(\xi_{+},\xi_{-})\mapsto(\cos\theta,\cos\sigma)$ and calling
\begin{equation*}
	f(\theta,\sigma)=\Gegenbauer{k}{3/2}(\cos\theta)\Gegenbauer{k-\beta}{3/2}(\cos\theta)\Gegenbauer{k}{3/2}(\cos\sigma)\Gegenbauer{k-\beta}{3/2}(\cos\sigma),
\end{equation*}
the integral of the lemma equals
\begin{equation*}
	\int_{0}^{\pi/2}\int_{0}^{\pi}f(\theta,\sigma)\sin\theta\sin\sigma\log(1-\cos\theta\cos\sigma)\dd \theta \dd \sigma.
\end{equation*}
We divide the region $[0,\pi]\times [0,\pi/2]$ as in \cref{fig:regions-dpp}. From \cref{lemma:integral-log-R1}, we have
\begin{align*}
	\MoveEqLeft\iint_{R_1}f(\theta,\sigma)\sin\theta\sin\sigma\log(1-\cos\theta\cos\sigma)\dd\theta\dd\sigma\\
	&=(-1)^{\beta}\frac{\log(2)}{4}\Bigl(1-\BesselJ{0}(1)^2-\BesselJ{1}(1)^2\Bigr)^2k^4+o(k^4).
\end{align*}
Over the regions $R_{2}$, $R_{3}$, and $R_{4}$ we have $\abs{\log(1-\cos\theta\cos\sigma)}\leq 1$. Therefore, using \cref{lemma:integral-R2-R3,lemma:integral-R4},
\begin{align*}
	\MoveEqLeft\abs[\bigg]{\iint_{R_j}f(\theta,\sigma)\sin\theta\sin\sigma\log(1-\cos\theta\cos\sigma)\dd \theta\dd\sigma}\\
	&\leq\iint_{R_j}\abs[\big]{f(\theta,\sigma)\sin\theta\sin\sigma}\dd\theta\dd\sigma=o(k^4), \qquad 2\leq j\leq 4.
\end{align*}
Now, over the regions $R_5$ to $R_7$ we have $\log(1-\cos\theta\cos\sigma)=O(\log{k})$. Hence, using \cref{lemma:integral-R5,lemma:integral-R6-R7} with $s=0$, we have
\begin{equation*}
	\iint\limits_{R_5\union R_6\union R_7}f(\theta,\sigma)\sin\theta\sin\sigma\log(1-\cos\theta\cos\sigma)\dd \theta\dd\sigma=O(k^{15/4}\log{k})=o(k^{4}).
\end{equation*}
Finally, from \cref{lemma:integral-log-R8} we obtain
\begin{align*}
	\MoveEqLeft\iint_{R_8}f(\theta,\sigma)\sin\theta\sin\sigma\log(1-\cos\theta\cos\sigma)\dd \theta\dd\sigma\\
	&=-\frac{k^4\log{k}}{2}+k^4\biggl(\frac{4+8G+\pi-4\gamma\pi+4\pi\log{2}}{8\pi}-\frac{\log{2}}{4}\biggr)+o(k^4).\pushQED{\qed}\qedhere
\end{align*}
To prove \cref{lemma:integral-dpp-hyper-1,lemma:integral-dpp-hyper-2,lemma:integral-dpp-hyper-3} we need to consider a slightly more exhaustive division of the region of integration. In this case, we consider the subdivision sketched in \cref{fig:regions-dpp-hypersingular}. 

\begin{figure}[htbp]
	\centering
	\begin{tikzpicture}[scale=0.9]
		\def\R{0.8\textwidth}
		\def\r{0.4\textwidth}
		\def\s{0.18\textwidth}
		\def\t{0.09\textwidth}
		\coordinate (A0) at (0,0);
		\coordinate (A1) at (\R,0);
		\coordinate (A2) at (0,\r);
		\coordinate (A3) at (\r,0);
		\coordinate (A6) at (\t,0);
		\coordinate (A7) at (0,\t);
		\coordinate (A8) at (\s,0);
		\coordinate (A9) at (0,\s);
		\draw (A0) node[below left]{$0$};
		\draw[->] (-1,0)--(0.9\textwidth,0) node[below]{$\theta$};
		\draw[->] (0,-1)--(0,0.5\textwidth) node[left]{$\sigma$};
		\draw (0,0) rectangle (\R,\r);
		\draw (\r,0) rectangle (\R,\r);
		\draw (\r,0) rectangle (\R,\r) node[pos=.5]{$\tilde{R}_{1}$};
		\draw (\s,\s) rectangle (\r,\r) node[pos=.5]{$\tilde{R}_{2}$};
		\draw (\t,\s) rectangle (\s,\r) node[pos=.5]{$\tilde{R}_{3}$};
		\draw (\s,\t) rectangle (\r,\s) node[pos=.5]{$\tilde{R}_{4}$};
		\draw (0,\s) rectangle (\t,\r) node[pos=.5]{$\tilde{R}_{5}$};
		\draw (\s,0) rectangle (\r,\t) node[pos=.5]{$\tilde{R}_{6}$};
		\draw (\t,\t) rectangle (\s,\s) node[pos=.5]{$\tilde{R}_{7}$};
		\draw (0,\t) rectangle (\t,\s) node[pos=.5]{$\tilde{R}_{8}$};
		\draw (\t,0) rectangle (\s,\t) node[pos=.5]{$\tilde{R}_{9}$};
		\draw (0,0) rectangle (\t,\t) node[pos=.5]{$\tilde{R}_{10}$};
		\draw (A1) node[below]{$\pi$};
		\draw (A2) node[left]{$\pi/2$};
		\draw (A3) node[below]{$\pi/2$};
		\draw (A6) node[below]{$k^{-1}$};
		\draw (A7) node[left]{$k^{-1}$};
		\draw (A8) node[below]{$k^{-1/2}$};
		\draw (A9) node[left]{$k^{-1/2}$};
	\end{tikzpicture}
	\caption{Subdivision of the rectangle $\tilde{R}=[0,\pi]\times[0,\pi/2]$ used in the proofs of \cref{lemma:integral-dpp-hyper-1,lemma:integral-dpp-hyper-2,lemma:integral-dpp-hyper-3}.}
	\label{fig:regions-dpp-hypersingular}
\end{figure}
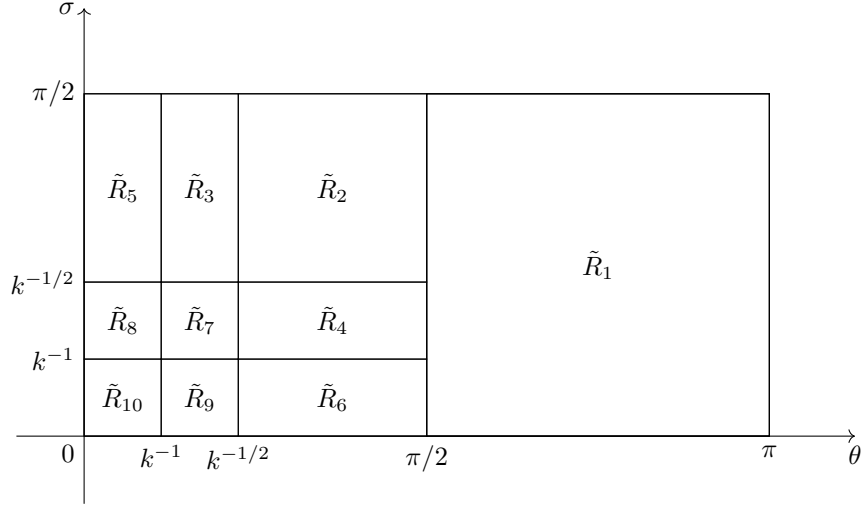

\begin{lemma}\label{lemma:integral-hyper-tilde-R7}
	The following equality holds:
	\begin{equation*}
		\iint_{\tilde{R}_7}\frac{f(\theta,\sigma)\sin\theta\sin\sigma}{(1-\cos\theta\cos\sigma)^2}\dd\theta\dd\sigma=4k^8\int_{1}^{\infty}\int_{1}^{\infty}\frac{\BesselJ{1}(x)^2\BesselJ{1}(y)^2}{xy(x^2+y^2)^2}\dd x\dd y+o(k^8).
	\end{equation*}
\end{lemma}

\begin{proof}
	It suffices to compute the limit
	\begin{equation}\label{eq:aux-hyper-R5-1}
		\lim_{k\tendsto\infty}\frac{1}{k^8}\int_{k^{-1}}^{k^{-1/2}}\int_{k^{-1}}^{k^{-1/2}}\frac{f(\theta,\sigma)\sin\theta\sin\sigma}{(1-\cos\theta\cos\sigma)^2}\dd\theta\dd\sigma.
	\end{equation}
	Over this region, we have $\sin\theta\sim\theta$, $\sin\sigma\sim\sigma$, and
	\begin{equation*}
		\frac{1}{(1-\cos\theta\cos\sigma)^2}\sim\frac{4}{(\theta^2+\sigma^2)^2}.
	\end{equation*}
	Hence, the limit in \eqref{eq:aux-hyper-R5-1} equals
	\begin{equation*}
		4\lim_{k\tendsto\infty}\frac{1}{k^8}\int_{k^{-1}}^{k^{-1/2}}\int_{k^{-1}}^{k^{-1/2}}\frac{f(\theta,\sigma)\theta\sigma}{(\theta^2+\sigma^2)^2}(1+o(1))\dd\theta\dd\sigma.
	\end{equation*}
	Making the substitution $(\theta,\sigma)\mapsto(\hat\theta/k,\hat\sigma/k)$, this expression reads as
	\begin{align}\label{eq:aux-hyper-R5-2}
		\MoveEqLeft 4\lim_{k\tendsto\infty}\int_{1}^{k^{1/2}}\int_{1}^{k^{1/2}}\frac{1}{k^8}\frac{f(\hat\theta/k,\hat\sigma/k)\hat\theta\hat\sigma}{(\hat\theta^2+\hat\sigma^2)^2}(1+o(1))\dd\hat\theta\dd\hat\sigma\\
		&=4\lim_{k\tendsto\infty}\int_{1}^{\infty}\int_{1}^{\infty}\frac{\chi_{[1,k^{1/2}]}(\hat\theta)\chi_{[1,k^{1/2}]}(\hat\sigma)}{k^8}\frac{f(\hat\theta/k,\hat\sigma/k)\hat\theta\hat\sigma}{(\hat\theta^2+\hat\sigma^2)^2}(1+o(1))\dd\hat\theta\dd\hat\sigma.\notag
	\end{align}
	It can be checked, using the bounds in \eqref{eq:bound-gegenbauer}, that the function $g\from[1,\infty]\times[1,\infty]$ given by
	\begin{equation*}
		g(\hat\theta,\hat\sigma)=\frac{1}{\hat\theta^2\hat\sigma^2(\hat\theta^2+\hat\sigma^2)^2}
	\end{equation*}
	satisfies
	\begin{equation*}
		\abs[\bigg]{\frac{\chi_{[1,k^{1/2}]}(\hat\theta)\chi_{[1,k^{1/2}]}(\hat\sigma)}{k^8}\frac{f(\hat\theta/k,\hat\sigma/k)\hat\theta\hat\sigma}{(\hat\theta^2+\hat\sigma^2)^2}}\lesssim g(\hat\theta,\hat\sigma)
	\end{equation*}
	and is integrable in $[1,\infty]\times[1,\infty]$. Hence, using the Dominated Convergence Theorem and the Mehler--Heine formula \eqref{eq:Mehler-Heine-Gegenbauer}, the expression \eqref{eq:aux-hyper-R5-2} equals
	\begin{equation*}
		4\int_{1}^{\infty}\int_{1}^{\infty}\frac{\BesselJ{1}(\hat\theta)^2\BesselJ{1}(\hat\sigma)^2}{\hat\theta\hat\sigma(\hat\theta^2+\hat\sigma^2)^2}\dd \hat\theta\dd \hat\sigma.\qedhere
	\end{equation*}
\end{proof}

\begin{lemma}\label{lemma:integral-hyper-tilde-R8-R9}
	The following equality holds:
	\begin{align*}
		\iint_{\tilde{R}_8}\frac{f(\theta,\sigma)\sin\theta\sin\sigma}{(1-\cos\theta\cos\sigma)^2}\dd\theta\dd\sigma&=\iint_{\tilde{R}_9}\frac{f(\theta,\sigma)\sin\theta\sin\sigma}{(1-\cos\theta\cos\sigma)^2}\dd\theta\dd\sigma\\
		&=4k^8\int_{1}^{\infty}\int_{0}^{1}\frac{\BesselJ{1}(x)^2\BesselJ{1}(y)^2}{xy(x^2+y^2)^2}\dd x\dd y+o(k^8).
	\end{align*}
\end{lemma}

\begin{proof}
	Since the analysis over the regions $\tilde{R}_8$ and $\tilde{R}_9$ is equivalent, we will focus on $\tilde{R}_8$. To prove the lemma it suffices to compute the following limit:
	\begin{equation}\label{eq:aux-hyper-tilde-R8-R9-1}
		\lim_{k\tendsto\infty}\frac{1}{k^8}\int_{k^{-1}}^{k^{-1/2}}\int_{0}^{k^{-1}}\frac{f(\theta,\sigma)\sin\theta\sin\sigma}{(1-\cos\theta\cos\sigma)^2}\dd\theta\dd\sigma.
	\end{equation}
	Over this region, we have $\sin\theta\sim\theta$, $\sin\sigma\sim\sigma$, and
	\begin{equation*}
		\frac{1}{(1-\cos\theta\cos\sigma)^2}\sim\frac{4}{(\theta^2+\sigma^2)^2}.
	\end{equation*}
	Hence, the limit in \eqref{eq:aux-hyper-tilde-R8-R9-1} equals
	\begin{equation*}
		4\lim_{k\tendsto\infty}\frac{1}{k^8}\int_{k^{-1}}^{k^{-1/2}}\int_{0}^{k^{-1}}\frac{f(\theta,\sigma)\theta\sigma}{(\theta^2+\sigma^2)^2}(1+o(1))\dd\theta\dd\sigma.
	\end{equation*}
	Making the substitution $(\theta,\sigma)\mapsto(\hat\theta/k,\hat\sigma/k)$, this expression reads as
	\begin{align}\label{eq:aux-hyper-tilde-R8-R9-2}
		\MoveEqLeft 4\lim_{k\tendsto\infty}\int_{1}^{k^{1/2}}\int_{0}^{1}\frac{1}{k^8}\frac{f(\hat\theta/k,\hat\sigma/k)\hat\theta\hat\sigma}{(\hat\theta^2+\hat\sigma^2)^2}(1+o(1))\dd\hat\theta\dd\hat\sigma\\
		&=4\lim_{k\tendsto\infty}\int_{1}^{\infty}\int_{0}^{1}\frac{1}{k^8}\chi_{[1,k^{1/2}]}(\hat\sigma)\frac{f(\hat\theta/k,\hat\sigma/k)\hat\theta\hat\sigma}{(\hat\theta^2+\hat\sigma^2)^2}(1+o(1))\dd\hat\theta\dd\hat\sigma.\notag
	\end{align}
	It can be checked, using the bounds in \eqref{eq:bound-gegenbauer}, that the function $g\from[0,1]\times[1,\infty]$ given by
	\begin{equation*}
		g(\hat\theta,\hat\sigma)=\frac{\hat\theta}{\hat\sigma^4}
	\end{equation*}
	satisfies
	\begin{equation*}
		\abs[\bigg]{\frac{\chi_{[1,k^{1/2}]}(\hat\sigma)}{k^8}\frac{f(\hat\theta/k,\hat\sigma/k)\hat\theta\hat\sigma}{(\hat\theta^2+\hat\sigma^2)^2}}\lesssim g(\hat\theta,\hat\sigma)
	\end{equation*}
	and is integrable in $[0,1]\times[1,\infty]$. Hence, using the Dominated Convergence Theorem and the Mehler--Heine formula \eqref{eq:Mehler-Heine-Gegenbauer}, the expression \eqref{eq:aux-hyper-R5-2} equals
	\begin{equation*}
		4\int_{1}^{\infty}\int_{0}^{1}\frac{\BesselJ{1}(\hat\theta)^2\BesselJ{1}(\hat\sigma)^2}{\hat\theta\hat\sigma(\hat\theta^2+\hat\sigma^2)^2}\dd \hat\theta\dd \hat\sigma.\qedhere
	\end{equation*}
\end{proof}

\begin{lemma}\label{lemma:integral-hyper-R-minus-R10}
	The following equality holds:
	\begin{align*}
		\MoveEqLeft\iint_{\tilde{R}\setminus \tilde{R}_{10}}\frac{f(\theta,\sigma)\sin\theta\sin\sigma}{(1-\cos\theta\cos\sigma)^2}\dd\theta\dd\sigma\\
		&=4k^8\biggl(\int_{1}^{\infty}\int_{1}^{\infty}\frac{\BesselJ{1}(x)^2\BesselJ{1}(y)^2}{xy(x^2+y^2)^2}\dd x\dd y+2\int_{1}^{\infty}\int_{0}^{1}\frac{\BesselJ{1}(x)^2\BesselJ{1}(y)^2}{xy(x^2+y^2)^2}\dd x\dd y\biggr)+o(k^8).
	\end{align*}
\end{lemma}

\begin{proof}
	First, note that, by comparing the divisions of the region $[0,\pi]\times[0,\pi/2]$ given by \cref{fig:regions-dpp,fig:regions-dpp-hypersingular}, we have 
	\begin{equation*}
		\tilde{R}_1=\bigunion_{j=1}^{4}R_j,\quad \tilde{R}_2=R_5,\quad \tilde{R}_3\union\tilde{R}_5=R_6,\quad \tilde{R}_4\union\tilde{R}_6=R_7.
	\end{equation*}
	Hence, using \cref{lemma:integral-R1,lemma:integral-R2-R3,lemma:integral-R4,lemma:integral-R5,lemma:integral-R6-R7} with $s=4$ and taking into account that over $\tilde{R}_1$ we have $1-\cos\theta\cos\sigma\geq 1$, we obtain
	\begin{equation*}
		\iint_{\bigunion_{j=1}^{6}\tilde{R}_j}\frac{f(\theta,\sigma)\sin\theta\sin\sigma}{(1-\cos\theta\cos\sigma)^2}\dd\theta\dd\sigma=o(k^8).
	\end{equation*}
	Finally, using \cref{lemma:integral-hyper-tilde-R7,lemma:integral-hyper-tilde-R8-R9}, the lemma follows.
\end{proof}

\subsubsection{Proof of \cref{lemma:integral-dpp-hyper-1}}\label{sec:proof-lemma-integral-dpp-hyper-1}

Making the substitution $(\xi_{+},\xi_{-})\mapsto(\cos\theta,\cos\sigma)$, the integral of the lemma reads as
\begin{equation*}
	\iint_{\tilde{R}\setminus\tilde{R}_{10}}\frac{\sin\theta\sin\sigma}{(1-\cos\theta\cos\sigma)^2}\dd\theta\dd\sigma,
\end{equation*}
where $\tilde{R}\setminus\tilde{R}_{10}$ is as in \cref{fig:regions-dpp-hypersingular}. First, we have
\begin{equation*}
	\iint_{\tilde{R}_{1}}\frac{\sin\theta\sin\sigma}{(1-\cos\theta\cos\sigma)^2}\dd\theta\dd\sigma=\log{2}.
\end{equation*}
Now, the integral over the region $\tilde{R}_2\union\tilde{R}_3\union\tilde{R}_4\union\tilde{R}_7=[k^{-1},\pi/2]^2$ is
\begin{equation*}
	\iint\limits_{\tilde{R}_2\union\tilde{R}_3\union\tilde{R}_4\union\tilde{R}_7} \frac{\sin\theta\sin\sigma}{(1-\cos\theta\cos\sigma)^2}\dd \theta\dd\sigma=-2\log(\sin(k^{-1}))=2\log{k}+O(k^{-2}),
\end{equation*}
where the last equality follows from $\sin(k^{-1})=k^{-1}(1+O(k^{-2}))$. The analysis over the regions $\tilde{R}_6\union\tilde{R}_9$ and $\tilde{R}_5\union\tilde{R}_8$ is equivalent. We have
\begin{align*}
	\iint\limits_{\tilde{R}_6\union\tilde{R}_9} \frac{\sin\theta\sin\sigma}{(1-\cos\theta\cos\sigma)^2}\dd \theta\dd\sigma=\iint\limits_{\tilde{R}_5\union\tilde{R}_8} \frac{\sin\theta\sin\sigma}{(1-\cos\theta\cos\sigma)^2}\dd \theta\dd\sigma&=\log(1+\cos(k^{-1}))\\
    &=\log{2}+O(k^{-2}),
\end{align*}
where the last equality follows from $1+\cos(k^{-1})=2+O(k^{-2})$. Putting all the previous computations together, the lemma follows.\qed

\subsubsection{Proof of \cref{lemma:integral-dpp-hyper-2}}\label{sec:proof-lemma-integral-dpp-hyper-2}

Making the substitution $(\xi_{+},\xi_{-})\mapsto(\cos\theta,\cos\sigma)$, the new region of integration is the region $\tilde{R}\setminus \tilde{R}_{10}$ from \cref{fig:regions-dpp-hypersingular}. The result follows then from \eqref{eq:auxiliar-function-lemmas} and \cref{lemma:integral-hyper-R-minus-R10}.\qed

\subsubsection{Proof of \cref{lemma:integral-dpp-hyper-3}}\label{sec:proof-lemma-integral-dpp-hyper-3}

Making the substitution $(\xi_{+},\xi_{-})\mapsto(\cos\theta,\cos\sigma)$, the new region of integration is the region $\tilde{R}_{10}$ from \cref{fig:regions-dpp-hypersingular}. Then, to prove the lemma it suffices to compute the limit
\begin{equation}\label{eq:aux-hyper-R10-1}
	\lim_{k\tendsto\infty}\frac{1}{k^8}\int_{0}^{k^{-1}}\int_{0}^{k^{-1}}\frac{(\calK_k(1,1)^2-\calK_k(\cos\theta,\cos\sigma)^2)\sin\theta\sin\sigma}{(1-\cos\theta\cos\sigma)^2}\dd\theta\dd\sigma.
\end{equation}
Using the estimates $\sin\theta\sim\theta$, $\sin\sigma\sim\sigma$, and
\begin{equation*}
	\frac{1}{(1-\cos\theta\cos\sigma)^2}\sim\frac{4}{(\theta^2+\sigma^2)^2},
\end{equation*}
we can write \eqref{eq:aux-hyper-R10-1} as
\begin{equation*}
	4\lim_{k\tendsto\infty}\frac{1}{k^8}\int_{0}^{k^{-1}}\int_{0}^{k^{-1}}\frac{(\calK_k(1,1)^2-\calK_k(\cos\theta,\cos\sigma)^2)\theta\sigma}{(\theta^2+\sigma^2)^2}(1+o(1))\dd\theta\dd\sigma.
\end{equation*}
Making the substitution $(\theta,\sigma)\mapsto(\hat\theta/k,\hat\sigma/k)$, the previous expression reads as
\begin{equation}\label{eq:aux-hyper-R10-2}
	4\lim_{k\tendsto\infty}\frac{1}{k^8}\int_{0}^{1}\int_{0}^{1}\frac{(\calK_k(1,1)^2-\calK_k(\cos(\hat\theta/k),\cos(\hat\sigma/k))^2)\hat\theta\hat\sigma}{(\hat\theta^2+\hat\sigma^2)^2}(1+o(1))\dd\hat\theta\dd\hat\sigma.
\end{equation}
Applying the multivariate Mean Value Theorem to the function $h\from[0,1]^2\to\R$ defined as 
\begin{equation*}
	h(\hat\theta,\hat\sigma)=\calK_k(\cos(\hat\theta/k),\cos(\hat\sigma/k))^2,
\end{equation*}
we have
\begin{align*}
	\abs{\calK_k(1,1)^2-\calK_k(\cos(\hat\theta/k),\cos(\hat\sigma/k))^2}&=\abs{h(0,0)-h(\hat\theta,\hat\sigma)}\\
    &\leq \sup_{c\in[0,1]}\norm{\gradient h(c\hat\theta,c\hat\sigma)}\norm{(\hat\theta,\hat\sigma)}.
\end{align*}
Then, using \cref{eq:derivatives-Gegenbauer} for the derivatives of the Gegenbauer polynomials and the bounds in \eqref{eq:bound-gegenbauer-general}, it can be proved that
\begin{equation*}
	\abs{\calK_k(1,1)^2-\calK_k(\cos(\hat\theta/k),\cos(\hat\sigma/k))^2}=(\hat\theta^2+\hat\sigma^2)O(k^8).
\end{equation*}
Therefore, we have
\begin{equation*}
	\abs[\bigg]{\frac{1}{k^8}\frac{(\calK_k(1,1)^2-\calK_k(\cos(\hat\theta/k),\cos(\hat\sigma/k))^2)\hat\theta\hat\sigma}{(\hat\theta^2+\hat\sigma^2)^2}}\lesssim \frac{\hat\theta\hat\sigma}{\hat\theta^2+\hat\sigma^2},
\end{equation*}
where the expression in the right hand side is integrable in $[0,1]^2$. Hence, using the Dominated Convergence Theorem and the Mehler--Heine formula \eqref{eq:Mehler-Heine-Gegenbauer}, we conclude that \eqref{eq:aux-hyper-R10-2} equals
\begin{align*}
	\int_{0}^{1}\int_{0}^{1}
	\frac{\hat\theta^2\hat\sigma^2-16\BesselJ{1}(\hat\theta)^2\BesselJ{1}(\hat\sigma)^2}{\hat\theta\hat\sigma(\hat\theta^2+\hat\sigma^2)^2}\dd \hat\theta\dd\hat\sigma,
\end{align*}
where we have used that $\calK_k(1,1)=\frac{1}{2}k^4+o(k^4)$. The lemma follows.\qed

\section*{Acknowledgments}

We would like to thank Javier Segura and Dmitrii Karp for their assistance with the estimations of certain special functions, and Carlos Beltrán for useful discussions on the topic.


\appendix

\section{Orthogonal polynomials and special functions}\label{appendix:orthogonal-polynomials}

\subsection{Legendre and Gegenbauer polynomials}

Classical orthogonal polynomials, and in particular Legendre and Gegenbauer polynomials, play a crucial role in this paper. In this section, we compile some properties of these polynomials that are used throughout this work. 

We denote the Legendre and Gegenbauer polynomials by $\Legendre{n}(x)$ and $\Gegenbauer{n}{\lambda}(x)$, respectively, where $n$ denotes the degree of the polynomial and $\lambda\in\R$ is a parameter. Both Legendre and Gegenbauer polynomials belong to the family of Jacobi polynomials $\Jacobi{n}{\alpha}{\beta}$, which form a complete orthogonal system in $[-1,1]$ with respect to the weight $w(x)=(1-x)^{\alpha}(1+x)^{\beta}$. More specifically, we have $\Legendre{n}=\Jacobi{n}{0}{0}$ and
\begin{equation*}
	 \Gegenbauer{n}{\lambda}=\frac{\Gamma(n+2\lambda)\Gamma(\lambda+1/2)}{\Gamma(n+\lambda+1/2)\Gamma(2\lambda)}\Jacobi{n}{\lambda-1/2}{\lambda-1/2}.
\end{equation*}
Legendre polynomials satisfy $\Legendre{n}(1)=1$. 
For the Gegenbauer polynomials $\Gegenbauer{n}{\lambda}$, we have
\begin{equation}\label{eq:Gegenbauer(1)}
    \Gegenbauer{n}{\lambda}(1)=\frac{\Gamma(2\lambda+n)}{\Gamma(2\lambda)\Gamma(n+1)}.
\end{equation}
Legendre polynomials satisfy the following recurrence formula, known as \emph{Bonnet's recurrence formula}: 
\begin{equation}\label{eq:recurrence-Legendre}
	(n+1)\Legendre{n+1}(x)=(2n+1)x\Legendre{n}(x)-n\Legendre{n-1}(x).
\end{equation}
In the case of the Gegenbauer polynomials, this recurrence relation takes the form
\begin{equation}\label{eq:recurrence-Gegenbauer-3/2}
	(n+1)\Gegenbauer{n+1}{3/2}(x)=(2n+3)x\Gegenbauer{n}{3/2}(x)-(n+2)\Gegenbauer{n-1}{3/2}(x).
\end{equation}
The derivatives of the Gegenbauer polynomials are given by \cite[8.935-1]{GradshteynRyzhik2007}:
\begin{equation}\label{eq:derivatives-Gegenbauer}
	\frac{\dd^j}{\dd x^j}\Gegenbauer{n}{\lambda}(x)=2^j\frac{\Gamma(\lambda+j)}{\Gamma(\lambda)}\Gegenbauer{n-j}{\lambda+j}(x).
\end{equation}
From \cite[Eq. (7.33.6)]{Szego1975} we have, for $n\geq 1$, the following asymptotic estimates for the Gegenbauer polynomials:
\begin{equation}\label{eq:bound-gegenbauer-general}
    \Gegenbauer{n}{\lambda}(\cos\theta)= \begin{cases}
        \theta^{-\lambda}O(n^{\lambda-1}), & cn^{-1}\leq \theta\leq \pi/2,\\
        O(n^{2\lambda-1}), & 0\leq \theta\leq cn^{-1},\\
    \end{cases}
\end{equation}
where $c$ is a fixed positive constant. The first bound can be used for the whole interval $0<\theta\leq \pi/2$. In the case $\lambda=3/2$, this bounds take the form
\begin{equation}\label{eq:bound-gegenbauer}
	\Gegenbauer{n}{3/2}(\cos\theta)= \begin{cases}
		\theta^{-3/2}O(n^{1/2}), & cn^{-1}\leq \theta\leq \pi/2,\\
		O(n^{2}), & 0\leq \theta\leq cn^{-1}.\\
	\end{cases}
\end{equation}
Near the endpoints, a more precise asymptotic is provided by the Mehler--Heine formula (see \cite[Theorem 8.1.1]{Szego1975}):
\begin{equation}\label{eq:Mehler-Heine-Gegenbauer}
	\lim_{n\tendsto\infty}\frac{1}{n^2}\Gegenbauer{n}{3/2}\Bigl(\cos\frac{z}{n}\Bigr)=\frac{1}{z}\BesselJ{1}(z),
\end{equation}
where $\BesselJ{1}$ is the Bessel function of the first kind of order $1$ and the limit is uniform in compact subsets of $\C$.

We conclude this section with some useful relations between Legendre and Gegenbauer polynomials. From \cite[4.3.1-2]{PrudnikovBrychkovMarichev1986}, we have the following identity:
\begin{equation}\label{eq:sum-legendre-even-odd}
    \sum_{j=0}^{n}(4j+2\beta+1)\Legendre{2j+\beta}(x)=\frac{\dd}{\dd x}\Legendre{2n+\beta+1}(x)=\Gegenbauer{2n+\beta}{3/2}(x),\qquad \beta\in\set{0,1}.
\end{equation}
From \cite[4.7.29]{Szego1975}, and using the recurrence relation \eqref{eq:recurrence-Gegenbauer-3/2} for the Gegenbauer polynomials, we obtain the following two identities:
\begin{align}
	\Legendre{n}(x)&=\frac{1}{n+1}\Bigl(\Gegenbauer{n}{3/2}(x)-x\Gegenbauer{n-1}{3/2}(x)\Bigr).\label{eq:legendre-gegenbauer-n}\\
	\Legendre{n+1}(x)&=\frac{1}{n+1}\Bigl(x\Gegenbauer{n}{3/2}(x)-\Gegenbauer{n-1}{3/2}(x)\Bigr).\label{eq:legendre-gegenbauer-n+1}
\end{align}

\subsection{Generalized hypergeometric functions}

We compile here the basic definitions and results about generalized hypergeometric functions that we use in this work. 

Let $p,q$ be nonnegative integers with $p\leq q+1$. Let $a_1,\dotsc,a_p,b_1,\dotsc,b_q\in\C$. The generalized hypergeometric series $\hypergeom{p}{q}(a_1,\dotsc,a_p;b_1,\dotsc,b_q;z)$ is defined as the following series: 
\begin{equation}\label{eq:hypergeometric-series}
    \hypergeom{p}{q}(a_1,\dotsc,a_p;b_1,\dotsc,b_q;z)=\sum_{k=0}^{\infty}\frac{\Pochhammer{a_1}{k}\dotsb\Pochhammer{a_p}{k}}{\Pochhammer{b_1}{k}\dotsb\Pochhammer{b_q}{k}}\frac{z^k}{k!}.
\end{equation}
The case $p=q+1$ is of special importance. In that case, if one or more of the parameters $a_j$ is a nonpositive integer, then the series \eqref{eq:hypergeometric-series} terminates and the generalized hypergeometric function is a polynomial in $z$. If none of the $a_j$ is a nonpositive integer, then the radius of convergence of the series \eqref{eq:hypergeometric-series} is 1, and outside the open disk $\abs{z}<1$ the generalized hypergeometric function is defined by analytic continuation with respect to $z$. On the circle $\abs{z}=1$, setting $\gamma=\sum_{j=1}^{q}b_j-\sum_{j=1}^{q+1}a_j$, the series \eqref{eq:hypergeometric-series} is absolutely convergent if $\Re(\gamma)>0$, convergent except at $z=1$ if $-1<\Re(\gamma)\leq 0$, and divergent if $\Re(\gamma)\leq -1$. The next classical result (see \cite[Theorems 2 and 4]{Buhring1992}) gives us the behavior of the generalized hypergeometric series $\hypergeom{q+1}{q}(a_1,\dotsc,a_{q+1};b_1,\dotsc,b_q;z)$ near $z=1$.

\begin{proposition}\label{prop:hypergeom32-limit-divergent}
	Let $\gamma=\sum_{j=1}^{q}b_j-\sum_{j=1}^{q+1}a_j<0$ . Then,
	\begin{equation*}
		\lim_{z\tendsto 1}\frac{\hypergeom{q+1}{q}(a_1,\dotsc,a_{q+1};b_1,\dotsc,b_q;z)}{(1-z)^{\gamma}}=\frac{\Gamma(-\gamma)\Gamma(b_1)\dotsb\Gamma(b_{q})}{\Gamma(a_1)\dotsb\Gamma(a_{q+1})}.
	\end{equation*}
\end{proposition}


\bibliographystyle{amsplain}

\end{document}